\documentclass{amsart}
\usepackage{url}
\usepackage{color}
\usepackage{times}
\usepackage{cite}
\usepackage{enumerate,latexsym}
\usepackage{latexsym}
\usepackage{amsmath,amssymb}
\usepackage{amsthm}
\usepackage{verbatim}
\usepackage{mathrsfs}
\newtheorem{thm}{Theorem}[section]
\newtheorem*{theorem*}{Theorem}
\newtheorem*{acknowledgement*}{Acknowledgement}
\newtheorem{cor}[thm]{Corollary}

\newtheorem{lem}[thm]{Lemma}
\newtheorem{prop}[thm]{Proposition}
\theoremstyle{definition}

\theoremstyle{remark}
\newtheorem{rem}[thm]{Remark}

\numberwithin{equation}{section}

\newcommand{\set}[1]{\left\{#1\right\}}
\newcommand{\Real}{\mathbb R}

\newcommand{\func}[1]{\ensuremath{{\mathrm{#1}}\:}}

\newcommand{\dist}[0]{\mathrm{dist}}

\newcommand{\spt}[0]{\func{spt}}

\title{An integer degree for asymptotically conical self-expanders}

\author{Jacob Bernstein}
\address{Department of Mathematics, Johns Hopkins University, 3400 N. Charles Street, Baltimore, MD 21218}
\email{bernstein@math.jhu.edu}

\author{Lu Wang}
\address{Department of Mathematics, University of Wisconsin-Madison, 480 Lincoln Drive, Madison, WI 53706}
\email{luwang@math.wisc.edu}
\thanks{The first author was partially supported by the NSF Grant DMS-1609340. The second author was partially supported by the NSF Grants DMS-1406240 and DMS-1834824, an Alfred P. Sloan Research Fellowship, and the office of the Vice Chancellor for Research and Graduate Education at the University of Wisconsin-Madison with funding from the Wisconsin Alumni Research Foundation and a Vilas Early Investigator Award.}

\begin{document}

\begin{abstract}
We establish the existence of an integer degree for the natural projection map from the space of parameterizations of asymptotically conical self-expanders to the space of parameterizations of the asymptotic cones when this map is proper. As an application we show that there is an open set in the space of cones in $\mathbb{R}^3$ for which each cone in the set has a strictly unstable self-expanding annuli asymptotic to it.
\end{abstract}

\maketitle

\section{Introduction}
A \emph{hypersurface}, i.e., a properly embedded codimension-one submanifold, $\Sigma\subset\mathbb{R}^{n+1}$, is a \emph{self-expander} if
\begin{equation} \label{ExpanderEqn}
\mathbf{H}_\Sigma-\frac{\mathbf{x}^\perp}{2}=\mathbf{0}.
\end{equation}
Here 
$$
\mathbf{H}_\Sigma=\Delta_\Sigma\mathbf{x}=-H_\Sigma\mathbf{n}_\Sigma=-\mathrm{div}_\Sigma(\mathbf{n}_\Sigma)\mathbf{n}_\Sigma
$$
is the mean curvature vector, $\mathbf{n}_\Sigma$ is the unit normal, and $\mathbf{x}^\perp$ is the normal component of the position vector. Self-expanders arise naturally in the study of mean curvature flow. Indeed, $\Sigma$ is a self-expander if and only if the family of homothetic hypersurfaces
$$
\left\{\Sigma_t\right\}_{t>0}=\left\{\sqrt{t}\, \Sigma\right\}_{t>0}
$$
is a \emph{mean curvature flow} (MCF), that is, a solution to the flow
$$
\left(\frac{\partial \mathbf{x}}{\partial t}\right)^\perp=\mathbf{H}_{\Sigma_t}.
$$
Self-expanders are expected to model the behavior of a MCF as it emerges from a conical singularity \cite{AIC}. They are also expected to model the long time behavior of the flow \cite{EHAnn}.

Given an integer $k\geq 2$ and $\alpha\in (0,1)$ let $\Gamma$ be a $C^{k,\alpha}_{*}$-asymptotically conical $C^{k,\alpha}$-hypersurface in $\mathbb{R}^{n+1}$ and let $\mathcal{L}(\Gamma)$ be the link of the asymptotic cone of $\Gamma$.  For instance, if $\lim_{\rho\to 0^+} \rho \Gamma=\mathcal{C}$ in $C^{k, \alpha}_{loc} (\mathbb{R}^{n+1}\setminus\{\mathbf{0}\})$, where $\mathcal{C}$ is a cone, then $\Gamma$ is $C^{k,\alpha}_{*}$-asymptotically conical with asymptotic cone $\mathcal{C}$. For technical reasons, the actual definition is slightly weaker -- see Section 3 of \cite{BWBanachManifold} for the details. We denote the space of $C^{k,\alpha}_{*}$-asymptotically conical $C^{k,\alpha}$-hypersurfaces in $\mathbb{R}^{n+1}$ by $\mathcal{ACH}^{k,\alpha}_n$.

In \cite{BWBanachManifold}, the authors showed that the space $\mathcal{ACE}_n^{k,\alpha}(\Gamma)$ -- see \eqref{ACEDef} below -- of asymptotically conical parameterizations of self-expanders modeled on $\Gamma$ (modulo reparameterizations fixing the parameterization of the asymptotic cone) possesses a natural Banach manifold structure modeled on $C^{k, \alpha}(\mathcal{L}(\Gamma); \mathbb{R}^{n+1})$. The authors further showed that the map
$$
\Pi\colon \mathcal{ACE}_n^{k,\alpha}(\Gamma)\to C^{k, \alpha}(\mathcal{L}(\Gamma); \mathbb{R}^{n+1})
$$
given by $\Pi([\mathbf{f}])=\mathrm{tr}^1_{\infty}[\mathbf{f}]$ is smooth and Fredholm of index $0$. Here $\mathrm{tr}^1_\infty[\mathbf{f}]$ is the trace at infinity of $\mathbf{f}$ -- see \eqref{TraceEqn}. As such, by work of Smale \cite{Smale}, as long as $\Pi$ is proper it possesses a well-defined mod 2 degree -- see \cite{BWProperness} for natural situations in which $\Pi$ is proper. In this paper, we show that as long as $\Pi$ is proper it has a well-defined \emph{integer} degree. Namely, let
$$
\mathcal{V}_{\mathrm{emb}}^{k,\alpha}(\Gamma)= \set{\varphi\in C^{k,\alpha}(\mathcal{L}(\Gamma);\mathbb{R}^{n+1}) \colon \mbox{$\mathscr{E}^{\mathrm{H}}_1[\varphi]$ is an embedding}},
$$
be the space of parameterizations of embedded cones. Here $\mathscr{E}^\mathrm{H}_1[\varphi]$ is the homogeneous degree-one extension of $\varphi$ -- see \eqref{HomoExtEqn}. For any $\varphi\in \mathcal{V}_{\mathrm{emb}}^{k,\alpha}(\Gamma)$ a regular value of $\Pi$ so that $\Pi^{-1}(\varphi)$ is a finite set let
$$
\mathrm{deg}(\Pi, \varphi)= \sum_{[\mathbf{f}]\in \Pi^{-1}(\varphi)} (-1)^{\mathrm{ind}([\mathbf{f}])}
$$
where $\mathrm{ind}([\mathbf{f}])$ is the \emph{Morse index} of $[\mathbf{f}]$ -- see Section \ref{IndexSec}. We claim that in many situations the map $\mathrm{deg}$ depends only on the component of $\varphi$ in $\mathcal{V}_{\mathrm{emb}}^{k,\alpha}(\Gamma)$.
More precisely,

\begin{thm}\label{MainThm}
Given $\Gamma\in\mathcal{ACH}^{k,\alpha}_n$, let $\mathcal{V}\subset \mathcal{V}_{\mathrm{emb}}^{k,\alpha}(\Gamma)$ be a connected open set and let $\mathcal{U}\subset \mathcal{ACE}_n^{k,\alpha}(\Gamma)$ be open and so $\Pi|_{\mathcal{U}}\colon \mathcal{U}\to \mathcal{V}$ is proper. Then for $\varphi\in\mathcal{V}$ which is a regular value of $\Pi|_{\mathcal{U}}$,  
$$
\mathrm{deg}(\Pi|_{\mathcal{U}}, \varphi)=\sum_{[\mathbf{f}]\in \mathcal{U}\cap \Pi^{-1}(\varphi) }(-1)^{\mathrm{ind}([\mathbf{f}])}
$$
is independent of $\varphi$ -- i.e., $\Pi|_{\mathcal{U}}$ has a well-defined integer degree $\mathrm{deg}(\Pi|_\mathcal{U},\mathcal{V})$.
\end{thm}

This is analogous to a result proved by White \cite{WhiteEI} for a large class of compact critical points of elliptic variational problems.  While the argument in this article follows the outline laid out in \cite{WhiteEI}, the details are quite different and are substantially more involved.  This is due to technical issues introduced both by the non-compactness of the problem and, more seriously, by the fast growth of the weight.

As an application of Theorem \ref{MainThm} we obtain the existence of many strictly unstable asymptotically conical self-expanders.

\begin{thm}\label{ApplicationThm}
Let $\Gamma\in\mathcal{ACH}^{k,\alpha}_2$ be an annulus. There is an open set $\mathcal{V}_0\subset\mathcal{V}_{\mathrm{emb}}^{k,\alpha}(\Gamma)$ so that for each $\varphi\in\mathcal{V}_0$ there is an element $[\mathbf{f}]\in\mathcal{ACE}^{k,\alpha}_2(\Gamma)$ with $\mathrm{tr}_\infty^1[\mathbf{f}]=\varphi$ and so that $\mathbf{f}(\Gamma)$ is a strictly unstable self-expander.
\end{thm}

\begin{rem}
It is necessary to invoke the integer degree for the map $\Pi$, instead of the mod $2$ degree, in order to conclude that the self-expanders in Theorem \ref{ApplicationThm} are strictly unstable. 
\end{rem}

\begin{rem}
In \cite{AIC}, Angenent-Ilmanen-Chopp show that there is a critical value $\delta_*$ so that for $0<\delta<\delta_*$ there are no connected rotationally symmetric self-expanders asymptotic to the rotationally symmetric double cone 
$$
\mathcal{C}_{\delta}=\set{x_1^2+\cdots +x_n^2= \delta^2 x_{n+1}^2},
$$
while for $\delta>\delta_*$ there is at least one connected rotationally symmetric self-expanders asymptotic to $\mathcal{C}_\delta$. In \cite{Helmensdorfer}, Helmensdorfer further analyzed this problem and showed that, for $\delta>\delta_*$, there is a second connected rotationally symmetric self-expander asymptotic to $\mathcal{C}_\delta$. It is likely that the solution constructed by Angenent-Ilmanen-Chopp is strictly stable while the one found by Helmensdorfer is strictly unstable (and there is a unique weakly stable example asymptotic to $\mathcal{C}_{\delta_*}$), but, to our knowledge, a complete proof has not appeared in the literature.
\end{rem}

\section{Notation and background} \label{NotationSec}
\subsection{Basic notions} \label{NotionSubsec}
Denote a (open) ball in $\mathbb{R}^n$ of radius $R$ and center $x$ by $B_R^n(x)$ and the closed ball by $\bar{B}^n_R(x)$. We often omit the superscript $n$ when its value is clear from context. We also omit the center when it is the origin. 

For an open set $U\subset\mathbb{R}^{n+1}$, a \emph{hypersurface in $U$}, $\Sigma$, is a smooth, properly embedded, codimension-one  submanifold of $U$. We also consider hypersurfaces of lower regularity and given an integer $k\geq 2$ and $\alpha\in (0,1)$ we define a \emph{$C^{k,\alpha}$-hypersurface in $U$} to be a properly embedded, codimension-one $C^{k,\alpha}$ submanifold of $U$. When needed, we distinguish between a point $p\in\Sigma$ and its \emph{position vector} $\mathbf{x}(p)$.

Consider the hypersurface $\mathbb{S}^n\subset\mathbb{R}^{n+1}$, the unit $n$-sphere in $\mathbb{R}^{n+1}$. A \emph{hypersurface in $\mathbb{S}^n$}, $\sigma$, is a closed, embedded, codimension-one smooth submanifold of $\mathbb{S}^n$ and \emph{$C^{k,\alpha}$-hypersurfaces in $\mathbb{S}^n$} are defined likewise. Observe that $\sigma$ is a closed codimension-two submanifold of $\mathbb{R}^{n+1}$ and so we may associate to each point $p\in\sigma$ its position vector $\mathbf{x}(p)$. Clearly, $|\mathbf{x}(p)|=1$.

A \emph{cone} is a set $\mathcal{C}\subset\mathbb{R}^{n+1}\setminus\{\mathbf{0}\}$ that is dilation invariant around the origin. That is, $\rho\mathcal{C}=\mathcal{C}$ for all $\rho>0$. The \emph{link} of the cone is the set $\mathcal{L}[\mathcal{C}]=\mathcal{C}\cap\mathbb{S}^{n}$. The cone is \emph{regular} if its link is a smooth hypersurface in $\mathbb{S}^{n}$ and \emph{$C^{k,\alpha}$-regular} if its link is a $C^{k,\alpha}$-hypersurface in $\mathbb{S}^n$. For any hypersurface $\sigma\subset\mathbb{S}^n$ the \emph{cone over $\sigma$}, $\mathcal{C}[\sigma]$, is the cone defined by 
$$
\mathcal{C}[\sigma]=\left\{\rho p\colon p\in\sigma, \rho>0\right\}\subset\mathbb{R}^{n+1}\setminus\{\mathbf{0}\}.
$$
Clearly, $\mathcal{L}[\mathcal{C}[\sigma]]=\sigma$. 

\subsection{Function spaces} \label{FunctionSubsec}
Let $\Sigma$ be a properly embedded, $C^{k,\alpha}$ submanifold of an open set $U\subset\mathbb{R}^{n+1}$. There is a natural Riemannian metric, $g_\Sigma$, on $\Sigma$ of class $C^{k-1,\alpha}$ induced from the Euclidean one. As we always take $k\geq 2$, the Christoffel symbols of this metric, in appropriate coordinates, are well-defined and of regularity $C^{k-2,\alpha}$. Let $\nabla_\Sigma$ be the covariant derivative on $\Sigma$. Denote by $d_\Sigma$ the geodesic distance on $\Sigma$ and by $B^\Sigma_R(p)$ the (open) geodesic ball in $\Sigma$ of radius $R$ and center $p\in\Sigma$. For $R$ small enough so that $B_{R}^\Sigma(p)$ is strictly geodesically convex and $q\in B^\Sigma_R(p)$, denote by $\tau^\Sigma_{p,q}$ the parallel transport along the unique minimizing geodesic in $B^\Sigma_R(p)$ from $p$ to $q$. 

Throughout the rest of this subsection, let $\Omega$ be a domain in $\Sigma$, $l$ an integer in $[0,k]$, $\gamma\in (0,1)$ and $d\in\mathbb{R}$. Suppose $l+\gamma\leq k+\alpha$. We first consider the following norm for functions on $\Omega$:
$$
\Vert f\Vert_{l; \Omega}=\sum_{i=0}^l \sup_{\Omega} |\nabla_\Sigma^i f|.
$$
We then let
$$
C^l(\Omega)=\left\{f\in C_{loc}^l(\Omega)\colon \Vert f\Vert_{l; \Omega}<\infty\right\}.
$$
We next define the H\"{o}lder semi-norms for functions $f$ and tensor fields $T$ on $\Omega$: 
$$
[f]_{\gamma; \Omega} =\sup_{\substack{p,q\in\Omega \\ q\in B^\Sigma_{\delta}(p)\setminus\{p\}}} \frac{|f(p)-f(q)|}{d_\Sigma(p,q)^\gamma} 
\mbox{ and } 
[T]_{\gamma; \Omega} =\sup_{\substack{p,q\in\Omega \\ q\in B^\Sigma_{\delta}(p)\setminus\{p\}}} \frac{|T(p)-(\tau^\Sigma_{p,q})^* T(q)|}{d_\Sigma(p,q)^\gamma},
$$
where $\delta=\delta(\Sigma,\Omega)>0$ so that for all $p\in\Omega$, $B^\Sigma_\delta(p)$ is strictly geodesically convex. We further define the norm for functions on $\Omega$:
$$
\Vert f\Vert_{l, \gamma; \Omega}=\Vert f\Vert_{l; \Omega}+[\nabla_\Sigma^l f]_{\gamma; \Omega},
$$
and let 
$$
C^{l, \gamma}(\Omega)=\left\{f\in C_{loc}^{l, \gamma}(\Omega)\colon \Vert f\Vert_{l, \gamma; \Omega}<\infty\right\}.
$$

We also define the following weighted norm for functions on $\Omega$:
$$
\Vert f\Vert_{l; \Omega}^{(d)}=\sum_{i=0}^l\sup_{p\in\Omega} \left(|\mathbf{x}(p)|+1\right)^{-d+i} |\nabla_\Sigma^i f(p)|.
$$
We then let 
$$
C^{l}_d(\Omega)=\left\{f\in C^l_{loc}(\Omega)\colon \Vert f\Vert_{l; \Omega}^{(d)}<\infty\right\}.
$$
We further define the following weighted H\"{o}lder semi-norms for functions $f$ and tensor fields $T$ on $\Omega$:
\begin{align*}
[f]_{\gamma; \Omega}^{(d)} & =\sup_{\substack{p,q\in\Omega \\ q\in B^\Sigma_{\delta_p}(p)\setminus\{p\}}} \left((|\mathbf{x}(p)|+1)^{-d+\gamma}+(|\mathbf{x}(q)|+1)^{-d+\gamma}\right) \frac{|f(p)-f(q)|}{d_\Sigma(p,q)^\gamma}, \mbox{ and}, \\
[T]_{\gamma; \Omega}^{(d)} & =\sup_{\substack{p,q\in\Omega \\ q\in B^\Sigma_{\delta_p}(p)\setminus\{p\}}} \left((|\mathbf{x}(p)|+1)^{-d+\gamma}+(|\mathbf{x}(q)|+1)^{-d+\gamma}\right) \frac{|T(p)-(\tau^\Sigma_{p,q})^* T(q)|}{d_\Sigma(p,q)^\gamma},
\end{align*}
where $\eta=\eta(\Omega,\Sigma)\in \left(0,\frac{1}{4}\right)$ so that for any $p\in\Sigma$, letting $\delta_p=\eta (|\mathbf{x}(p)|+1)$, $B_{\delta_p}^\Sigma(p)$ is strictly geodesically convex. Next we define the norm for functions on $\Omega$:
$$
\Vert f\Vert_{l, \gamma; \Omega}^{(d)}=\Vert f\Vert_{l; \Omega}^{(d)}+[\nabla_\Sigma^l f]_{\gamma; \Omega}^{(d-l)},
$$
and we let
$$
C^{l,\gamma}_d(\Omega)=\left\{f\in C^{l,\gamma}_{loc}(\Omega)\colon \Vert f\Vert_{l, \gamma; \Omega}^{(d)}<\infty\right\}.
$$
We follow the convention that $C^{l,0}_{loc}=C^{l}_{loc}$, $C^{l,0}=C^l$ and $C^{l,0}_d=C^l_d$ and that $C^{0, \gamma}_{loc}=C^\gamma_{loc}$,  $C^{0,\gamma}=C^\gamma$ and $C^{0,\gamma}_d=C^\gamma_d$. The notation for the corresponding norms is abbreviated in the same fashion.

Finally, we define the following weighted integral norm for functions on $\Omega$:
$$
\Vert f\Vert^{W, (d)}_{l;\Omega}=\left(\int_{\Omega} \sum_{i=0}^l |\nabla^i_\Sigma f|^2 e^{d|\mathbf{x}|^2}\, d\mathcal{H}^n\right)^{\frac{1}{2}},
$$
and we then let 
$$
W^{l}_{d}(\Omega)=\set{f\in W^{l,2}_{loc}(\Omega)\colon\Vert f\Vert^{W, (d)}_{l;\Omega}<\infty}.
$$

In all above definitions of various norms, we often omit the domain $\Omega$ when it is clear from context. These norms can be extended in a straightforward manner to vector-valued functions and tensor fields.  It is a standard exercise to verify that these spaces equipped with the corresponding norms are Banach spaces.

\subsection{Homogeneous functions and homogeneity at infinity} \label{HomoFuncSubsec}
Fix a $C^{k,\alpha}$-regular cone $\mathcal{C}$ with its link $\mathcal{L}$. By our definition, $\mathcal{C}$ is a $C^{k,\alpha}$-hypersurface in $\mathbb{R}^{n+1}\setminus\{\mathbf{0}\}$. For $R>0$, let $\mathcal{C}_R=\mathcal{C}\setminus\bar{B}_R$. There is an $\eta=\eta(\mathcal{L},R)>0$ so that for any $p\in\mathcal{C}_R$, $B^{\mathcal{C}}_{\delta_p}(p)$ is strictly geodesically convex, where $\delta_p=\eta(|\mathbf{x}(p)|+1)$. We also fix an integer $l\in [0,k]$ and $\gamma\in [0,1)$ with $l+\gamma\leq k+\alpha$.

A map $\mathbf{f}\in C^{l,\gamma}_{loc}(\mathcal{C}; \mathbb{R}^M)$ is \emph{homogeneous of degree $d$} if $\mathbf{f}(\rho p)=\rho^d \mathbf{f}(p)$ for all $p\in\mathcal{C}$ and $\rho>0$. Given a map $\varphi\in C^{l,\gamma}(\mathcal{L}; \mathbb{R}^M)$ the \emph{homogeneous extension of degree $d$} of $\varphi$ is the map $\mathscr{E}_d^{\mathrm{H}}[\varphi]\in C^{l,\gamma}_{loc}(\mathcal{C}; \mathbb{R}^M)$ defined by 
\begin{equation} \label{HomoExtEqn}
\mathscr{E}_d^{\mathrm{H}}[\varphi](p)=|\mathbf{x}(p)|^d \varphi(|\mathbf{x}(p)|^{-1}p).
\end{equation}
Conversely, given a homogeneous $\mathbb{R}^M$-valued map of degree $d$, $\mathbf{f}\in C^{l,\gamma}_{loc}(\mathcal{C}; \mathbb{R}^M)$, let $\varphi=\mathrm{tr}[\mathbf{f}]\in C^{l,\beta}(\mathcal{L}; \mathbb{R}^M)$, the \emph{trace} of $\mathbf{f}$, be the restriction of $\mathbf{f}$ to $\mathcal{L}$. Clearly, $\mathbf{f}$ is the homogeneous extension of degree $d$ of $\varphi$.

A map $\mathbf{g}\in C^{l,\gamma}_{loc}(\mathcal{C}_R; \mathbb{R}^M)$ is \emph{asymptotically homogeneous of degree $d$} if 
$$
\lim_{\rho\to 0^+} \rho^d \mathbf{g}(\rho^{-1}p)=\mathbf{f}(p) \mbox{ in $C^{l,\gamma}_{loc}(\mathcal{C}; \mathbb{R}^M)$}
$$
for some $\mathbf{f}\in C^{l,\gamma}_{loc}(\mathcal{C}; \mathbb{R}^M)$ that is homogeneous of degree $d$. For such a $\mathbf{g}$ we define the \emph{trace at infinity} of $\mathbf{g}$ by $\mathrm{tr}^d_\infty[\mathbf{g}]=\mathrm{tr}[\mathbf{f}]$. 
We define
$$
C^{l,\gamma}_{d,\mathrm{H}}(\mathcal{C}_R; \mathbb{R}^M)=\left\{\mathbf{g}\in C^{l,\gamma}_d(\mathcal{C}_R; \mathbb{R}^M)\colon \mbox{$\mathbf{g}$ is asymptotically homogeneous of degree $d$}\right\}.
$$
It is straightforward to verify that $C^{l,\gamma}_{d,\mathrm{H}}(\mathcal{C}_R; \mathbb{R}^M)$ is a closed subspace of $C^{l,\gamma}_d(\mathcal{C}_R; \mathbb{R}^M)$ and that
$$
\mathrm{tr}^d_\infty\colon C^{l,\gamma}_{d,\mathrm{H}}(\mathcal{C}_R; \mathbb{R}^M)\to C^{l,\gamma}(\mathcal{L}; \mathbb{R}^M)
$$
is a bounded linear map. Finally, $\mathbf{x}|_{\mathcal{C}_R}\in C^{k,\alpha}_{1,\mathrm{H}}(\mathcal{C}_R; \mathbb{R}^{n+1})$ and $\mathrm{tr}_\infty^1[\mathbf{x}|_{\mathcal{C}_R}]=\mathbf{x}|_{\mathcal{L}}$.

\subsection{Asymptotically conical hypersurfaces} \label{ACHSubsec}
A $C^{k,\alpha}$-hypersurface, $\Sigma\subset\mathbb{R}^{n+1}$, is \emph{$C^{k,\alpha}_{*}$-asymptotically conical} if there is a $C^{k,\alpha}$-regular cone, $\mathcal{C}\subset\mathbb{R}^{n+1}$, and a homogeneous (of degree $0$) transverse section, $\mathbf{V}$, on $\mathcal{C}$ such that $\Sigma$, outside some compact set, is given by the $\mathbf{V}$-graph of a function in $C^{k,\alpha}_1\cap C^{k}_{1,0}(\mathcal{C}_R)$ for some $R>1$. Here a transverse section is a regularized version of the unit normal -- see Section 2.4 of \cite{BWBanachManifold} for the precise definition. Observe, that by the Arzel\`{a}-Ascoli theorem one has that, for every $l,\gamma$ with $l+\gamma<k+\alpha$,
$$
\lim_{\rho\to 0^+} \rho\Sigma = \mathcal{C} \mbox{ in } C^{l, \gamma}_{loc}(\mathbb{R}^{n+1}\setminus \{\mathbf{0}\}).
$$
Clearly, the asymptotic cone, $\mathcal{C}$, is uniquely determined by $\Sigma$ and so we denote it by $\mathcal{C}(\Sigma)$. Let $\mathcal{L}(\Sigma)$ denote the link of $\mathcal{C}(\Sigma)$ and, for $R>0$, let $\mathcal{C}_R(\Sigma)=\mathcal{C}(\Sigma)\setminus \bar{B}_R$. Denote the space of $C^{k,\alpha}_{*}$-asymptotically conical $C^{k,\alpha}$-hypersurfaces in $\mathbb{R}^{n+1}$ by $\mathcal{ACH}^{k,\alpha}_n$.

Finally, let $K$ be a compact subset of $\Sigma$ and let $\Sigma^\prime=\Sigma\setminus K$. By definition, we may choose $K$ large enough so $\pi_{\mathbf{V}}$ -- the projection of a neighborhood of $\mathcal{C}(\Sigma)$ along $\mathbf{V}$ --  restricts to a $C^{k,\alpha}$ diffeomorphism of $\Sigma^\prime$ onto $\mathcal{C}_R(\Sigma)$. Denote its inverse by $\theta_{\mathbf{V}; \Sigma^\prime}$. 

\subsection{Traces at infinity} \label{TraceSubsec}
Fix an element $\Sigma\in\mathcal{ACH}_n^{k,\alpha}$. Let $l$ be an integer in $[0,k]$ and $\gamma\in [0,1)$ such that $l+\gamma<k+\alpha$. A map $\mathbf{f}\in C_{loc}^{l,\gamma}(\Sigma; \mathbb{R}^M)$ is \emph{asymptotically homogeneous of degree $d$} if $\mathbf{f}\circ\theta_{\mathbf{V}; \Sigma^\prime}\in C^{l,\gamma}_{d,\mathrm{H}}(\mathcal{C}_R(\Sigma); \mathbb{R}^M)$ where $\mathbf{V}$ is a homogeneous transverse section on $\mathcal{C}(\Sigma)$ and $\Sigma^\prime, \theta_{\mathbf{V}; \Sigma^\prime}$ are introduced in the previous subsection. The \emph{trace at infinity} of $\mathbf{f}$ is then
\begin{equation} \label{TraceEqn}
\mathrm{tr}_\infty^d[\mathbf{f}]=\mathrm{tr}_\infty^d[\mathbf{f}\circ\theta_{\mathbf{V}; \Sigma^\prime}] \in C^{l,\gamma}(\mathcal{L}(\Sigma); \mathbb{R}^M).
\end{equation}
Whether $\mathbf{f}$ is asymptotically homogeneous of degree $d$ and the definition of $\mathrm{tr}_{\infty}^d$ are independent of the choice of homogeneous transverse sections on $\mathcal{C}(\Sigma)$. Clearly, $\mathbf{x}|_{\Sigma}$ is asymptotically homogeneous of degree $1$ and $\mathrm{tr}_\infty^{1}[\mathbf{x}|_\Sigma]=\mathbf{x}|_{\mathcal{L}(\Sigma)}$.

We next define the space
$$
C_{d,\mathrm{H}}^{l,\gamma}(\Sigma; \mathbb{R}^M)=\left\{\mathbf{f}\in C_{d}^{l,\gamma}(\Sigma; \mathbb{R}^M)\colon \mbox{$\mathbf{f}$ is asymptotically homogeneous of degree $d$}\right\}.
$$
One can check that $C_{d,\mathrm{H}}^{l,\gamma}(\Sigma; \mathbb{R}^M)$ is a closed subspace of $C_d^{l,\gamma}(\Sigma; \mathbb{R}^M)$, and the map 
$$
\mathrm{tr}_{\infty}^d\colon C_{d,\mathrm{H}}^{l,\gamma}(\Sigma; \mathbb{R}^M)\to C^{l,\gamma}(\mathcal{L}(\Sigma); \mathbb{R}^M)
$$
is a bounded linear map. We further define the set $C^{l,\gamma}_{d,0}(\Sigma;\mathbb{R}^M)\subset C_{d,\mathrm{H}}^{l,\gamma}(\Sigma; \mathbb{R}^M)$ to be the kernel of $\mathrm{tr}_\infty^d$.

\subsection{Asymptotically conical embeddings} \label{ACESubsec}
Fix an element $\Gamma\in \mathcal{ACH}^{k,\alpha}_n$. We define the space of $C^{k,\alpha}_{*}$-asymptotically conical embeddings of $\Gamma$ into $\mathbb{R}^{n+1}$ to be
$$
\mathcal{ACH}^{k,\alpha}_n(\Gamma)=\left\{\mathbf{f}\in C_{1}^{k,\alpha}\cap C^k_{1,\mathrm{H}}(\Gamma; \mathbb{R}^{n+1})\colon\mbox{$\mathbf{f}$ and $\mathscr{E}_{1}^{\mathrm{H}}\circ\mathrm{tr}_\infty^1[\mathbf{f}]$ are embeddings}\right\}.
$$
Clearly, $\mathcal{ACH}^{k,\alpha}_n(\Gamma)$ is an open set in the Banach space $C^{k,\alpha}_{1}\cap C^{k}_{1,\mathrm{H}}(\Gamma; \mathbb{R}^{n+1})$ with the $\Vert\cdot \Vert_{k,\alpha}^{(1)}$ norm. The hypotheses on $\mathbf{f}$, $\mathrm{tr}_\infty^1[\mathbf{f}]\in C^{k,\alpha}(\mathcal{L}(\Gamma);\mathbb{R}^{n+1})$ ensure
$$
\mathcal{C}[\mathbf{f}]=\mathscr{E}_{1}^{\mathrm{H}}\circ\mathrm{tr}_\infty^1[\mathbf{f}]\colon \mathcal{C}(\Gamma)\to \mathbb{R}^{n+1}\setminus\{\mathbf{0}\}
$$
is a $C^{k,\alpha}$ embedding. As this map is homogeneous of degree one, it parameterizes the $C^{k,\alpha}$-regular cone $\mathcal{C}(\mathbf{f}(\Gamma))$ -- see  \cite[Proposition 3.3]{BWBanachManifold}.

Finally, we introduce a natural equivalence relation on $\mathcal{ACH}_{n}^{k,\alpha}(\Gamma)$. First, say a $C^{k,\alpha}$ diffeomorphism $\phi\colon\Gamma\to \Gamma$ \emph{fixes infinity} if $\mathbf{x}|_{\Gamma}\circ \phi\in \mathcal{ACH}^{k,\alpha}_n(\Gamma)$ and
$$
\mathrm{tr}^1_{\infty}[\mathbf{x}|_{\Gamma}\circ \phi]=\mathbf{x}|_{\mathcal{L}(\Gamma)}.
$$ 
Two elements $\mathbf{f}, \mathbf{g}\in\mathcal{ACH}_{n}^{k,\alpha}(\Gamma)$ are equivalent, written $\mathbf{f}\sim\mathbf{g}$, provided there is a $C^{k,\alpha}$ diffeomorphism $\phi\colon\Gamma\to\Gamma$ that fixes infinity so that $\mathbf{f}\circ\phi=\mathbf{g}$. Given $\mathbf{f}\in \mathcal{ACH}_{n}^{k,\alpha}(\Gamma)$ let $[\mathbf{f}]$ be the equivalence class of $\mathbf{f}$. Following \cite{BWBanachManifold}, we define the space
\begin{equation}
\label{ACEDef}
\mathcal{ACE}_n^{k,\alpha}(\Gamma)=\set{[\mathbf{f}]\colon \mathbf{f}\in \mathcal{ACH}^{k,\alpha}_n(\Gamma) \mbox{ and $\mathbf{f}(\Gamma)$ satisfies \eqref{ExpanderEqn}}}.
\end{equation}

\subsection{Various symmetric bilinear forms} \label{BilinearFormSubsec}
Let $\Sigma\subset\mathbb{R}^{n+1}$ be a $C^2$-hypersurface and $\mathbf{v}$ a transverse section on $\Sigma$. Suppose $\mathbf{f}\colon\Sigma\to\mathbb{R}^{n+1}$ is a $C^2$ proper embedding. We set $\Lambda=\mathbf{f}(\Sigma)$ and $\mathbf{w}=\mathbf{v}\circ\mathbf{f}^{-1}$. We also assume $\mathbf{w}$ is transverse to $\Lambda$. If $g_{\mathbf{f}}$ is the pullback metric of the Euclidean one via $\mathbf{f}$, then we let
$$
\Omega_\mathbf{f}=\sqrt{\frac{\det g_{\mathbf{f}}}{\det g_\Sigma}}\, e^{\frac{|\mathbf{f}|^2-|\mathbf{x}|^2}{4}}
$$
which is the ratio of the pullback weighted measure on $\Lambda$ and the one on $\Sigma$, and 
$$
\Omega_{\mathbf{f},\mathbf{v}}=\left(\frac{\mathbf{v}\cdot (\mathbf{n}_\Lambda\circ\mathbf{f})}{\mathbf{v}\cdot\mathbf{n}_\Sigma}\right)^2 \Omega_\mathbf{f}.
$$
Observe that $\Omega_{\mathbf{x}|_\Sigma}=\Omega_{\mathbf{x}|_\Sigma,\mathbf{v}}=1$. 

First we introduce natural weighted inner products for functions on $\Sigma$. Suppose $u,v\in C_c^0 (\Sigma)$. Let
$$
B_\Sigma[u,v]=\int_{\Sigma} uv e^{\frac{|\mathbf{x}|^2}{4}} \, d\mathcal{H}^n
$$
and
$$
B_{\Sigma,\mathbf{v}}[u,v]=B_\Sigma[(\mathbf{v}\cdot\mathbf{n}_\Sigma)u,(\mathbf{v}\cdot\mathbf{n}_\Sigma)v]=\int_\Sigma uv(\mathbf{v}\cdot\mathbf{n}_\Sigma)^2 e^{\frac{|\mathbf{x}|^2}{4}} \, d\mathcal{H}^n.
$$
Likewise, let
$$ 
B_\mathbf{f}[u,v]=B_\Lambda\left[u\circ\mathbf{f}^{-1},v\circ\mathbf{f}^{-1}\right]=\int_\Sigma uv \Omega_\mathbf{f} e^{\frac{|\mathbf{x}|^2}{4}} \, d\mathcal{H}^n
$$
and
$$
B_{\mathbf{f},\mathbf{v}}[u,v]=B_{\Lambda,\mathbf{w}} \left[u\circ\mathbf{f}^{-1},v\circ\mathbf{f}^{-1}\right]=\int_{\Sigma} uv(\mathbf{v}\cdot(\mathbf{n}_\Lambda\circ\mathbf{f}))^2 \Omega_\mathbf{f} e^{\frac{|\mathbf{x}|^2}{4}} \, d\mathcal{H}^n.
$$
Thus, we have that
$$
B_{\mathbf{f}}[u,v]=B_\Sigma[\Omega_\mathbf{f} u, v] \mbox{ and } B_{\mathbf{f},\mathbf{v}}[u,v]=B_{\Sigma,\mathbf{v}}[\Omega_{\mathbf{f},\mathbf{v}} u, v].
$$

Next, we introduce the following Dirichlet energy forms: for $u,v\in C^1_c(\Sigma)$,
$$
D_\Sigma[u,v]=\int_\Sigma \nabla_\Sigma u\cdot\nabla_\Sigma v \, e^{\frac{|\mathbf{x}|^2}{4}}\, d\mathcal{H}^n, 
$$
and 
$$
D_{\Sigma,\mathbf{v}}[u,v]=D_\Sigma[(\mathbf{v}\cdot\mathbf{n}_\Sigma)u,(\mathbf{v}\cdot\mathbf{n}_\Sigma) v].
$$
If $\Sigma$ is a $C^3$-hypersurface, then 
$$
D_{\Sigma,\mathbf{v}}[u,v]=\int_\Sigma \left(\nabla_\Sigma u\cdot\nabla_\Sigma v+a_{\Sigma,\mathbf{v}}uv\right) (\mathbf{v}\cdot\mathbf{n}_\Sigma)^2 e^{\frac{|\mathbf{x}|^2}{4}} \, d\mathcal{H}^n
$$
where 
$$
a_{\Sigma,\mathbf{v}}=-(\mathbf{v}\cdot\mathbf{n}_\Sigma)^{-1} \left(\Delta_\Sigma+\frac{1}{2}\mathbf{x}\cdot\nabla_\Sigma\right) (\mathbf{v}\cdot\mathbf{n}_\Sigma).
$$
Likewise, let 
$$
D_\mathbf{f}[u,v]=D_{\Lambda}\left[u\circ\mathbf{f}^{-1},v\circ\mathbf{f}^{-1}\right]=\int_\Sigma (g_\mathbf{f}^{-1})^{ij}\nabla_i u\nabla_j v \, \Omega_\mathbf{f} e^{\frac{|\mathbf{x}|^2}{4}} \, d\mathcal{H}^n,
$$
and
$$
D_{\mathbf{f},\mathbf{v}}[u,v]=D_{\Lambda,\mathbf{w}}\left[u\circ\mathbf{f}^{-1},v\circ\mathbf{f}^{-1}\right].
$$
When $\Lambda$ is of class $C^3$,
$$
D_{\mathbf{f},\mathbf{v}}[u,v]=\int_\Sigma \left((g_\mathbf{f}^{-1})^{ij}\nabla_i u\nabla_j v+(a_{\Lambda,\mathbf{w}}\circ\mathbf{f}) uv\right) (\mathbf{v}\cdot(\mathbf{n}_\Lambda\circ\mathbf{f}))^2 \Omega_\mathbf{f} e^{\frac{|\mathbf{x}|^2}{4}} \, d\mathcal{H}^n.
$$

Finally, we consider the Jacobi operator (also called the stability operator) on $\Sigma$,
$$
L_\Sigma=\Delta_\Sigma+\frac{1}{2}\mathbf{x}\cdot\nabla_\Sigma+|A_\Sigma|^2-\frac{1}{2}.
$$
The operator $L_\Sigma$ is associated to the symmetric bilinear form
$$
Q_\Sigma[u,v]=\int_\Sigma \left(\nabla_\Sigma u\cdot\nabla_\Sigma v-|A_\Sigma|^2 uv+\frac{1}{2} uv\right) e^{\frac{|\mathbf{x}|^2}{4}} \, d\mathcal{H}^n
$$
in the sense that for $u,v\in C^2_c(\Sigma)$, 
$$
Q_\Sigma[u,v]=-B_\Sigma[u,L_\Sigma v].
$$
In particular, $L_\Sigma$ is symmetric with respect to the inner product $B_\Sigma$. 

We also consider the $\mathbf{v}$-Jacobi operator
$$
L_{\Sigma,\mathbf{v}}=(\mathbf{v}\cdot\mathbf{n}_\Sigma) L_\Sigma\left((\mathbf{v}\cdot\mathbf{n}_\Sigma)\times\right).
$$
It will also be convenient to consider the related operator
$$
L_{\Sigma,\mathbf{v}}^\prime=(\mathbf{v}\cdot\mathbf{n}_\Sigma)^{-2} L_{\Sigma,\mathbf{v}}.
$$
These operators are both associated to the symmetric bilinear form
$$
Q_{\Sigma,\mathbf{v}}[u,v] =Q_\Sigma[(\mathbf{v}\cdot\mathbf{n}_\Sigma)u,(\mathbf{v}\cdot\mathbf{n}_\Sigma)v].
$$
When $\Sigma$ is of class $C^3$,
$$
Q_{\Sigma,\mathbf{v}}[u,v] = \int_\Sigma \left(\nabla_\Sigma u\cdot\nabla_\Sigma v+a^\prime_{\Sigma,\mathbf{v}}uv\right) (\mathbf{v}\cdot\mathbf{n}_\Sigma)^2 e^{\frac{|\mathbf{x}|^2}{4}}\, d\mathcal{H}^n
$$
where
$$
a^\prime_{\Sigma,\mathbf{v}}=-(\mathbf{v}\cdot\mathbf{n}_\Sigma)^{-1} L_\Sigma(\mathbf{v}\cdot\mathbf{n}_\Sigma).
$$
Specifically one has, for $u,v\in C^2_c(\Sigma)$,
$$
Q_{\Sigma,\mathbf{v}}[u,v]=-B_\Sigma[u,L_{\Sigma,\mathbf{v}} v]=-B_{\Sigma,\mathbf{v}}[u,L_{\Sigma,\mathbf{v}}^\prime v],
$$
and so $L_{\Sigma,\mathbf{v}}$ and $L_{\Sigma,\mathbf{v}}^\prime$ are symmetric with respect to $B_\Sigma$ and $B_{\Sigma,\mathbf{v}}$, respectively.

Likewise, we consider the symmetric bilinear form
\begin{align*}
Q_{\mathbf{f}}[u,v] & = Q_{\Lambda}\left[u\circ\mathbf{f}^{-1},v\circ\mathbf{f}^{-1}\right] \\
& = \int_\Sigma \left((g_\mathbf{f}^{-1})^{ij}\nabla_i u\nabla_j v-(|A_\Lambda|^2\circ\mathbf{f})uv+\frac{1}{2}uv\right)\Omega_{\mathbf{f}} e^{\frac{|\mathbf{x}|^2}{4}} \, d\mathcal{H}^n.
\end{align*}
Let
$$
L_\mathbf{f}v=(L_\Lambda (v\circ\mathbf{f}^{-1}))\circ\mathbf{f}.
$$
Observe, that for $u,v\in C^2_c(\Sigma)$,
$$
Q_{\mathbf{f}}[u,v]= -B_\mathbf{f}[u,L_\mathbf{f} v],
$$
so $L_\mathbf{f}$ is symmetric with respect to $B_\mathbf{f}$. There also is a symmetric bilinear form
$$
Q_{\mathbf{f},\mathbf{v}}[u,v]= Q_{\Lambda,\mathbf{w}}\left[u\circ\mathbf{f}^{-1},v\circ\mathbf{f}^{-1}\right]. 
$$
When $\Lambda$ is of class $C^3$,
$$
Q_{\mathbf{f},\mathbf{v}}[u,v]  = \int_\Sigma \left((g_\mathbf{f}^{-1})^{ij}\nabla_iu\nabla_jv+(a_{\Lambda,\mathbf{w}}^\prime\circ\mathbf{f})uv\right)(\mathbf{v}\cdot(\mathbf{n}_\Lambda\circ\mathbf{f}))^2\Omega_\mathbf{f} e^{\frac{|\mathbf{x}|^2}{4}} \, d\mathcal{H}^n.
$$
If we let
$$
L_{\mathbf{f},\mathbf{v}} v=(L_{\Lambda,\mathbf{w}}(v\circ\mathbf{f}^{-1}))\circ\mathbf{f} \mbox{ and } L_{\mathbf{f},\mathbf{v}}^\prime=(\mathbf{v}\cdot(\mathbf{n}_\Lambda\circ\mathbf{f}))^{-2} L_{\mathbf{f},\mathbf{v}},
$$
then we observe that for $u,v\in C^2_c(\Sigma)$,
$$
Q_{\mathbf{f},\mathbf{v}}[u,v]=-B_\mathbf{f}[u,L_{\mathbf{f},\mathbf{v}} v]=-B_{\mathbf{f},\mathbf{v}}[u,L_{\mathbf{f},\mathbf{v}}^\prime v].
$$
As such, $L_{\mathbf{f},\mathbf{v}}^\prime$ is symmetric with respect to $B_{\mathbf{f},\mathbf{v}}$.

We will always write the quadratic form associated to $B_\Sigma$ as $B_\Sigma[u]=B_\Sigma[u,u]$ and do so for the other bilinear forms.

\subsection{First and second variations of the $E$-functional} \label{VariationSubsec} 
Fix a $C^2$-hypersurface $\Sigma\subset\mathbb{R}^{n+1}$ and a transverse section $\mathbf{v}$ on $\Sigma$. Suppose $\mathbf{f}\colon\Sigma\to\mathbb{R}^{n+1}$ is a $C^2$ proper embedding and that $\mathbf{w}=\mathbf{v}\circ\mathbf{f}^{-1}$ is transverse to $\Lambda=\mathbf{f}(\Sigma)$. We also assume $u,v\in C^2_c(\Sigma)$. For $|s|,|t|$ sufficiently small, 
$$
\mathbf{f}_{s,t}=\mathbf{f}+(su+tv)\mathbf{v}
$$
is a proper embedding of $\Sigma$ into $\mathbb{R}^{n+1}$ and $\mathbf{v}\circ\mathbf{f}_{s,t}^{-1}$ is a transverse section on $\mathbf{f}_{s,t}(\Sigma)$.

We have by the first variation formula (see \cite[Proposition 4.1]{BWBanachManifold}),
$$
{\frac{d}{ds}\vline}_{s=0} E[\mathbf{f}_{s,t}(\Sigma)]=-\int_{\Sigma} u\,  \Xi_{\mathbf{f},\mathbf{v}}[tv] \Omega_{\mathbf{f}_{0,t}} e^{\frac{|\mathbf{x}|^2}{4}} \, d\mathcal{H}^n=-B_{\mathbf{f}_{0,t}}[\Xi_{\mathbf{f},\mathbf{v}}[tv],u]
$$
where
$$
\Xi_{\mathbf{f},\mathbf{v}}[tv]=\mathbf{v}\cdot \left(\mathbf{H}-\frac{\mathbf{x}^\perp}{2}\right)[\mathbf{f}_{0,t}].
$$
It will also be convenient to introduce
$$
\tilde{\Xi}_{\mathbf{f},\mathbf{v}}[tv]=\Omega_{\mathbf{f}_{0,t}} \Xi_{\mathbf{f},\mathbf{v}}[tv] \mbox{ and } \hat{\Xi}_{\mathbf{f},\mathbf{v}}[tv]=(\mathbf{v}\cdot\mathbf{n}_{\Sigma})^{-2}\Omega_{\mathbf{f}_{0,t}}\Xi_{\mathbf{f},\mathbf{v}}[tv].
$$
We remark that the quantity $\tilde{\Xi}_{\mathbf{f},\mathbf{v}}$ is closer to what is considered in \cite{WhiteEI}, however for various reasons we find $\hat{\Xi}_{\mathbf{f},\mathbf{v}}$ easier to work with. We also observe that
$$
B_{\mathbf{f}_{0,t}}[\Xi_{\mathbf{f},\mathbf{v}}[tv],u]=B_\Sigma[\tilde{\Xi}_{\mathbf{f},\mathbf{v}}[tv],u]=B_{\Sigma,\mathbf{v}}[\hat{\Xi}_{\mathbf{f},\mathbf{v}}[tv],u].
$$
Thus, we have that
$$
{\frac{d}{ds}\vline}_{s=0} E[\mathbf{f}_{s,t}(\Sigma)]=-B_\Sigma[\tilde{\Xi}_{\mathbf{f},\mathbf{v}}[tv],u]=-B_{\Sigma,\mathbf{v}}[\hat{\Xi}_{\mathbf{f},\mathbf{v}}[tv],u].
$$
In particular, $\mathbf{f}_{0,t}(\Sigma)$ is a self-expander if and only if
$$
\Xi_{\mathbf{f},\mathbf{v}}[tv]=\tilde{\Xi}_{\mathbf{f},\mathbf{v}}[tv]=\hat{\Xi}_{\mathbf{f},\mathbf{v}}[tv]=0.
$$ 

We next compute the second variation at critical points of $E$. If $\mathbf{f}(\Sigma)$ is a self-expander, then, by \cite[Proposition 4.2]{BWBanachManifold}, we have
$$
{\frac{d^2}{dtds}\vline}_{t=s=0} E[\mathbf{f}_{s,t}(\Sigma)]=Q_{\mathbf{f},\mathbf{v}}[u,v]=-B_{\Sigma,\mathbf{v}}[\Omega_{\mathbf{f},\mathbf{v}}L^\prime_{\mathbf{f},\mathbf{v}} v, u],
$$
while differentiating the expression for the first variation formula directly gives, at critical points of $E$,
$$
{\frac{d^2}{dtds}\vline}_{t=s=0} E[\mathbf{f}_{s,t}(\Sigma)]=-B_{\Sigma,\mathbf{v}}[D\hat{\Xi}_{\mathbf{f},\mathbf{v}}(0)v, u].
$$
Hence, as long as $\mathbf{f}(\Sigma)$ is a self-expander and $\mathbf{v}\circ\mathbf{f}^{-1}$ is transverse to $\mathbf{f}(\Sigma)$,
\begin{equation} \label{2ndVarEqn}
D\hat{\Xi}_{\mathbf{f},\mathbf{v}}(0)=\Omega_{\mathbf{f},\mathbf{v}}L_{\mathbf{f},\mathbf{v}}^\prime
\end{equation}
is symmetric with respect to $B_{\Sigma,\mathbf{v}}$.

\subsection{Fast decay of eigenfunctions of $L_{\Sigma}$ and $L_{\Sigma,\mathbf{v}}^\prime$} \label{DecayEigenfunSubsec}
Let $\Sigma\in\mathcal{ACH}^{k,\alpha}_n$ be a self-expander. For $p\in\Sigma$ let $r(p)=|\mathbf{x}(p)|$. It follows from the expander equation, that there exist constants $R_0=R_0(\Sigma)$ and $C_0=C_0(\Sigma)$ so that on $E_{R_0}=\Sigma\setminus B_{R_0}$,
\begin{equation} \label{REstimatesEqn}
||\nabla_\Sigma r|-1|\leq C_0r^{-4} \leq \frac{1}{4}, \, |\nabla_\Sigma^2 r^2-2g_\Sigma|\leq C_0r^{-2} \leq\frac{1}{4},
\end{equation}
and,
\begin{equation} \label{AEstimatesEqn}
|A_\Sigma|^2\leq C_0r^{-2} \leq\frac{1}{4} \mbox{ and } \left|H_{S_\rho}-\frac{n-1}{\rho}\right| \leq C_0 \rho^{-3} \leq \frac{1}{4} \rho^{-1} \mbox{ for $\rho>R_0$}.
\end{equation}
Here $S_\rho=\Sigma\cap \partial B_\rho$ and $H_{S_\rho}$ is the mean curvature of $S_\rho$. Thus $(E_{R_0},g_\Sigma,r)$ is a weakly conical end in the sense of \cite[Section 2]{Bernstein}.

We use results from \cite{Bernstein} to show the eigenfunctions of $L_\Sigma$ or $L_{\Sigma,\mathbf{v}}^\prime$ that below a certain threshold, actually decay extremely fast.

\begin{prop} \label{DecayEigenfunProp}
Suppose that $\Sigma\in\mathcal{ACH}^{k,\alpha}_n$ is a self-expander, $\mathbf{v}\in C^{k,\alpha}_0\cap C^{k}_{0,\mathrm{H}}(\Sigma;\mathbb{R}^{n+1})$ is a transverse section on $\Sigma$ so $\mathcal{C}[\mathbf{v}]=\mathscr{E}^{\mathrm{H}}_1\circ\mathrm{tr}_\infty^1[\mathbf{v}]$ is a transverse section on $\mathcal{C}(\Sigma)$, and $\mu\in\mathbb{R}$. If $u\in W^{1,2}_{loc}\cap W^0_\gamma(E_R)$ for some $\gamma>0$ and $R>R_0$, and satisfies $L_\Sigma u=\mu u$ or $L_{\Sigma,\mathbf{v}}^\prime u=\mu u$, then, for any $\beta<\frac{1}{2}$,
$$
\Vert e^{\beta |\mathbf{x}|^2} u\Vert_{k,\alpha; E_{4R}}<\infty.
$$
\end{prop}

\begin{proof}
Let $U=(\mathbf{v}\cdot\mathbf{n}_\Sigma) u$. Suppose $u\in W^{1,2}_{loc}\cap W^0_\gamma(E_R)$ and so is $U$. We first observe that $L_\Sigma, L_{\Sigma,\mathbf{v}}^\prime$ are uniformly elliptic operators with $C^{k-2,\alpha}_1$ coefficients (cf. (5.21) of \cite{BWBanachManifold}). If $L_\Sigma u=\mu u$ or $L_{\Sigma,\mathbf{v}}^\prime u=\mu u$ (i.e., $L_\Sigma U=\mu U$), then the standard elliptic regularity theory (see, e.g.,  \cite[Chapters 6 and 8]{GilbargTrudinger}) implies that $u\in C^{k,\alpha}_{loc}(E_{2R})$ and $u$ decays, in the pointwise sense, like $e^{-\frac{\gamma}{4}|\mathbf{x}|^2}$. Furthermore, it follows from \cite[Theorems 9.1 and 7.2]{Bernstein} that $u\in W_\beta^0(E_{2R})$ for all $\beta<1/2$. Hence, the claim follows from the local boundedness for weak solutions and the Schauder estimates (cf. \cite[Theorems 8.7 and 6.2]{GilbargTrudinger}).
\end{proof}

\section{A modification of the smooth dependence theorem from \cite{BWBanachManifold}} \label{ModifiedSmoothDependSec}
In this section we slightly modify the maps obtained in \cite[Theorem 7.1]{BWBanachManifold}. This modification brings the maps closer in line with those obtained in \cite[Theorem 3.2]{WhiteEI} and can be thought of as picking charts around each point of $\mathcal{ACE}^{k,\alpha}_n(\Gamma)$ different from the ones given in \cite{BWBanachManifold}. One important point of emphasis, our proof of the modified theorem uses  \cite[Theorem 7.1]{BWBanachManifold} and it is not clear that one could prove the modified theorem directly. The issue has to do with the rapid growth of the weight, which can only be tamed using decay coming from \cite[Theorem 7.1]{BWBanachManifold}.  

Suppose that $\Sigma$ is a $C^2$-hypersurface in $\mathbb{R}^{n+1}$ and $\mathbf{v}$ is a transverse section on $\Sigma$. If $\mathbf{f}\colon\Sigma\to\mathbb{R}^{n+1}$ is a $C^2$ proper embedding, then we define 
$$
\Xi_\mathbf{v}[\mathbf{f}]=\mathbf{v}\cdot\left(\mathbf{H}-\frac{\mathbf{x}^\perp}{2}\right)[\mathbf{f}] \mbox{ and } \hat{\Xi}_\mathbf{v}[\mathbf{f}]=(\mathbf{v}\cdot\mathbf{n}_\Sigma)^{-2}\Omega_{\mathbf{f}} \Xi_\mathbf{v}[\mathbf{f}].
$$

\begin{thm} \label{ModifiedSmoothDependThm}
Let $\Sigma\in\mathcal{ACH}^{k,\alpha}_n$ be a self-expander, and let $\mathbf{v}\in C^{k,\alpha}_0\cap C^{k}_{0,\mathrm{H}}(\Sigma;\mathbb{R}^{n+1})$ be a transverse section on $\Sigma$ so that $\mathcal{C}[\mathbf{v}]$ is a transverse section on $\mathcal{C}(\Sigma)$ and that 
$$
\Delta_\Sigma\mathbf{v}+\frac{1}{2}\mathbf{x}\cdot\nabla_\Sigma\mathbf{v}\in C^{k-2,\alpha}_{-2}(\Sigma;\mathbb{R}^{n+1}).
$$ 
Let 
$$
\mathcal{K}_\mathbf{v}=\set{\kappa\in W^1_{\frac{1}{4}}(\Sigma)\colon L_{\Sigma,\mathbf{v}}^\prime \kappa=0}.
$$
Then $\dim\mathcal{K}_\mathbf{v}<\infty$, and there exist smooth maps
\begin{align*}
\hat{F}_{\mathbf{v}}\colon \mathcal{U}_1\times\mathcal{U}_2 & \to \mathcal{ACH}_{n}^{k,\alpha}(\Sigma), \mbox{ and}, \\
\hat{G}_{\mathbf{v}}\colon \mathcal{U}_1\times\mathcal{U}_2 & \to \mathcal{K}_{\mathbf{v}},
\end{align*}
where $\mathcal{U}_1$ is some neighborhood of $\mathbf{x}|_{\mathcal{L}(\Sigma)}$ in $C^{k,\alpha}(\mathcal{L}(\Sigma); \mathbb{R}^{n+1})$ and $\mathcal{U}_2$ is some neighborhood of $0$ in $\mathcal{K}_{\mathbf{v}}$, such that the following hold:
\begin{enumerate}
\item \label{TraceFItem} For $(\varphi,\kappa)\in\mathcal{U}_1\times\mathcal{U}_2$, $\mathrm{tr}_\infty^1[\hat{F}_{\mathbf{v}}[\varphi,\kappa]]=\varphi$. 
\item \label{IdentityItem} $\hat{F}_{\mathbf{v}}[\mathbf{x}|_{\mathcal{L}(\Sigma)},0]=\mathbf{x}|_\Sigma$. 
\item \label{EStationaryItem} For $(\varphi,\kappa)\in\mathcal{U}_1\times\mathcal{U}_2$, $\hat{F}_{\mathbf{v}}[\varphi,\kappa]$ is $E$-stationary if and only if $\hat{G}_{\mathbf{v}}[\varphi,\kappa]=0$. 
\item \label{FDiffItem} For $\kappa\in\mathcal{K}_{\mathbf{v}}$, $D_2 \hat{F}_{\mathbf{v}}(\mathbf{x}|_{\mathcal{L}(\Sigma)},0)\kappa=\kappa\mathbf{v}$.
\item \label{SubmanifoldItem} $\hat{G}_{\mathbf{v}}^{-1}(0)$ is a {smooth} submanifold of codimension equal to $\dim \mathcal{K}_{\mathbf{v}}$. It contains $\{0\}\times\mathcal{K}_{\mathbf{v}}$ in its tangent space at $(\varphi,0)$. Equivalently, $D_1 \hat{G}_{\mathbf{v}}(\mathbf{x}|_{\mathcal{L}(\Sigma)},0)$ is of rank equal to $\dim\mathcal{K}_{\mathbf{v}}$ and $D_2 \hat{G}_{\mathbf{v}}(\mathbf{x}|_{\mathcal{L}(\Sigma)},0)=0$. 
\item \label{ExpanderMeanCurvatureItem} One has
$$
\hat{\Xi}_{ \mathbf{v}}[\hat{F}_{\mathbf{v}}[\varphi, \kappa]]=\hat{G}_{\mathbf{v}}[\varphi, \kappa]\in \mathcal{K}_{\mathbf{v}}.
$$
\end{enumerate}
\end{thm}

\begin{rem}
Such a transverse section $\mathbf{v}$ always exists. For instance, outside a compact set of $\Sigma$, one may choose $\mathbf{v}$ to be $\mathbf{V}\circ\pi_{\mathbf{V}}|_\Sigma$ for $\mathbf{V}$ a homogeneous transverse section on $\mathcal{C}(\Sigma)$, and then extend $\mathbf{v}$ to be a transverse section on $\Sigma$. It is straightforward to verify that $\mathbf{v}$ satisfies the properties in Theorem \ref{ModifiedSmoothDependThm}.
\end{rem}

\begin{rem}
The only difference with \cite[Theorem 7.1]{BWBanachManifold} is the weight terms that appear in Item \eqref{ExpanderMeanCurvatureItem}. Namely, in that result one has 
$$
\Xi_{ \mathbf{v}}[{F}_{\mathbf{v}}[\varphi, \kappa]]={G}_{\mathbf{v}}[\varphi, \kappa]\in\mathcal{K}_{\mathbf{v}}.
$$
\end{rem}

To prove Theorem \ref{ModifiedSmoothDependThm} we will need several auxiliary propositions and lemmas. For $\beta,\mu\in\mathbb{R}$ we define
$$
\mathscr{L}_{\Sigma,\beta}^\mu=\Delta_\Sigma+2\beta\mathbf{x}\cdot\nabla_\Sigma-\mu.
$$
Observe that, for $f,g\in C^{2}_c(\Sigma)$,
$$
\int_\Sigma f (\mathscr{L}_{\Sigma,\beta}^0 g) e^{\beta |\mathbf{x}|^2} \, d\mathcal{H}^n = -\int_{\Sigma} \nabla_\Sigma f\cdot\nabla_\Sigma g \, e^{\beta |\mathbf{x}|^2} \, d\mathcal{H}^n.
$$
For brevity, we write $\mathscr{L}_\Sigma^\mu=\mathscr{L}_{\Sigma,\frac{1}{4}}^\mu$ and $\mathscr{L}_\Sigma=\mathscr{L}_{\Sigma,\frac{1}{4}}^{\frac{1}{2}}$.
\begin{prop} \label{ModifiedIsoProp}
Let $\Sigma\in\mathcal{ACH}^{k,\alpha}_n$ be a self-expander. For all $\beta\in \left[\frac{1}{4}, \frac{3}{8}\right]$,
$$
\mathscr{L}_{\Sigma}\colon W^2_\beta (\Sigma)\to W^0_\beta (\Sigma)
$$
is an isomorphism.
\end{prop}

\begin{proof}
First the Lax-Milgram theorem  implies that given an $f\in W^0_\beta (\Sigma)$ there is a unique solution $u\in W^1_{\frac{1}{4}}(\Sigma)$ to $\mathscr{L}_\Sigma u=f$ in the distributional sense. Now suppose that $f$ is smooth and has compact support. As such, there is an $R^\prime\geq R_0(\Sigma)$ sufficiently large so that $\spt(f)\subset B_{R^\prime}$ and $\Sigma\setminus\bar{B}_{R^\prime}$ is a weakly conical end by the definition of $R_0$ given in Section \ref{DecayEigenfunSubsec}. By standard elliptic regularity -- see, e.g., Chapter 6 and 8 of \cite{GilbargTrudinger} -- $\Sigma$ is indeed smooth and so is $u$. As $\mathscr{L}_{\Sigma} u=0$ in $\Sigma\backslash B_{R^\prime}$ it follows from Proposition \ref{DecayEigenfunProp} that
\begin{equation} \label{WeightW2Eqn}
\int_{\Sigma} \left(|\nabla_\Sigma^2 u|^2+|\nabla_\Sigma u|^2 +u^2\right)  e^{\frac{7}{16}|\mathbf{x}|^2} \, d\mathcal{H}^n<\infty.
\end{equation}
In other words, $u\in W^{2}_{\frac{7}{16}}(\Sigma)$. As $\beta<\frac{7}{16}$, this fact will be used to justify integration by parts in several places.

Indeed, as
$$
\mathscr{L}_{\Sigma}^0 (u^2)=2 |\nabla_\Sigma u|^2 +2 u \mathscr{L}_{\Sigma}^0 u= 2 |\nabla_\Sigma u|^2+ u^2 +2u f,
$$
integrating by parts, which is justified by \eqref{WeightW2Eqn}, gives
\begin{align*}
0 &= \int_{\Sigma} \mathscr{L}_{\Sigma, \beta}^0(u^2) e^{\beta |\mathbf{x}|^2}\, d\mathcal{H}^n=\int_{\Sigma} \left( \mathscr{L}_{\Sigma}^0(u^2)+\left(2\beta-\frac{1}{2}\right) \mathbf{x}^\top \cdot \nabla_\Sigma u^2\right) e^{\beta |\mathbf{x}|^2} \, d\mathcal{H}^n\\
&= \int_{\Sigma} \left( 2 |\nabla_\Sigma u|^2+ u^2 +2u f+ \left( 4\beta-1\right) u \mathbf{x}^\top\cdot \nabla_\Sigma u \right) e^{\beta |\mathbf{x}|^2} \, d\mathcal{H}^n\\
&\geq  \int_{\Sigma} \left( 2 |\nabla_\Sigma u|^2+ u^2 -\frac{1}{2} u^2 -2 f^2-  \left( 4\beta-1\right) \left( |\nabla_\Sigma u|^2 +\frac{1}{4}|\mathbf{x}^\top|^2 u^2 \right) \right) e^{\beta |\mathbf{x}|^2} \, d\mathcal{H}^n,
\end{align*}
where the last inequality used the absorbing inequality and the fact that $\beta\geq \frac{1}{4}$. Hence, by \eqref{PI2eqn} of Proposition \ref{GradControlsL2Prop}, one obtains
$$
\int_{\Sigma} \left( 2 |\nabla_\Sigma u|^2+ \frac{1}{2} u^2 -2 f^2- 4\beta  |\nabla_\Sigma u|^2  \right) e^{\beta |\mathbf{x}|^2} \, d\mathcal{H}^n \leq 0.
$$
Hence, as $\beta\leq 3/8$, 
$$
\int_{\Sigma} \left(\frac{1}{2} |\nabla_{\Sigma} u|^2 +\frac{1}{2}u^2 \right) e^{\beta|\mathbf{x}|^2} \, d\mathcal{H}^n\leq 2 \int_{\Sigma} f^2 e^{\beta|\mathbf{x}|^2}d\mathcal{H}^n.
$$
That is,  for $\beta\in \left[\frac{1}{4},\frac{3}{8}\right]$, 
\begin{equation} \label{WeightW1Eqn}
\Vert u \Vert_{1}^{W, (\beta)} \leq 4\Vert f \Vert_0^{W, (\beta)}.
\end{equation}

Next consider the Bochner type formula
\begin{align*}
\mathscr{L}_{\Sigma}^0 \left(\frac{1}{2} |\nabla_\Sigma u|^2\right)& =\nabla_{\Sigma} u \cdot \nabla_{\Sigma}(\mathscr{L}_{\Sigma}^0 u)+  |\nabla^2_\Sigma u|^2+\mathrm{Ric}_\Sigma (\nabla_\Sigma u, \nabla_\Sigma u)\\
& \quad -\frac{\mathbf{x}\cdot \mathbf{n}_\Sigma}{2} A_{\Sigma}(\nabla_{\Sigma} u, \nabla_{\Sigma} u)-\frac{1}{2}|\nabla_{\Sigma}u |^2 \\
&\geq \nabla_{\Sigma} u \cdot \nabla_{\Sigma}(\mathscr{L}_{\Sigma}^0 u)+  |\nabla^2_\Sigma u|^2-\left(\frac{1}{2}+K_{\Sigma}\right)|\nabla_{\Sigma}u |^2
\end{align*}
where the inequality uses the Gauss equation, the self-expander equation and 
\begin{equation}\label{KSigmaDef}
K_{\Sigma}=\sup_{p\in\Sigma} (1+|\mathbf{x}(p)|^2)|A_{\Sigma}(p)|^2< \infty.
\end{equation}
Integrating by parts (which can be justified by \eqref{WeightW2Eqn} together with a standard cutoff argument) yields
\begin{align*}
\int_{\Sigma} |\nabla_\Sigma^2 u|^2 & e^{\beta |\mathbf{x}|^2} \, d\mathcal{H}^n \leq  \int_{\Sigma} \left( \mathscr{L}_{\Sigma}^0 \left(\frac{1}{2} |\nabla_\Sigma u|^2\right)- \nabla_\Sigma u\cdot\nabla_\Sigma (\mathscr{L}_\Sigma^0 u ) \right) e^{\beta |\mathbf{x}|^2} \, d\mathcal{H}^n \\
& + \left(K_\Sigma+\frac{1}{2} \right) \int_{\Sigma} |\nabla_{\Sigma} u|^2  e^{\beta |\mathbf{x}|^2} \, d\mathcal{H}^n\\
& \leq \int_{\Sigma} \left( \frac{\left( 1-4\beta\right)}{2}\mathbf{x}^\top \cdot \nabla_\Sigma \left( \frac{1}{2} |\nabla_\Sigma u|^2 \right)+ ( \mathscr{L}_{\Sigma, \beta}^0 u) (\mathscr{L}_\Sigma^0 u ) \right) e^{\beta |\mathbf{x}|^2} \, d\mathcal{H}^n \\
& + \left(K_\Sigma+\frac{1}{2} \right) \int_{\Sigma} |\nabla_{\Sigma} u|^2  e^{\beta |\mathbf{x}|^2} \, d\mathcal{H}^n.
\end{align*}

One computes,
\begin{align*}
\left|  \int_{\Sigma}\mathbf{x}^\top \cdot \nabla_\Sigma \left( \frac{1}{2} |\nabla_\Sigma u|^2 \right) e^{\beta |\mathbf{x}|^2} \, d\mathcal{H}^n\right| & = \left| \int_{\Sigma} \nabla^2_\Sigma u (\nabla_\Sigma u , \mathbf{x}^\top) e^{\beta |\mathbf{x}|^2} \, d\mathcal{H}^n\right|\\
&\leq  \int_{\Sigma} |\nabla^2_\Sigma u| |\mathbf{x}^\top | |\nabla_\Sigma u|e^{\beta |\mathbf{x}|^2} \, d\mathcal{H}^n\\
&\leq  \int_{\Sigma} \left(  |\nabla^2_\Sigma u| ^2 +\frac{1}{4} |\mathbf{x}^\top |^2 |\nabla_\Sigma u|^2 \right)e^{\beta |\mathbf{x}|^2} \, d\mathcal{H}^n.
\end{align*}
Hence, by \eqref{PI2eqn} of Proposition \ref{GradControlsL2Prop} applied to $|\nabla_\Sigma u|$ and the fact that $\frac{1}{4}\leq \beta\leq \frac{3}{8}$, 
\begin{align*}
\left( \frac{ 1-4\beta}{2}\right) \int_{\Sigma}\mathbf{x}^\top \cdot \nabla_\Sigma \left(\frac{1}{2} |\nabla_\Sigma u|^2 \right) e^{\beta |\mathbf{x}|^2} \, d\mathcal{H}^n &\leq 2\beta \int_{\Sigma} |\nabla^2_\Sigma u| ^2  e^{\beta |\mathbf{x}|^2} \, d\mathcal{H}^n\\
&\leq \frac{3}{4}\int_{\Sigma} |\nabla^2_\Sigma u| ^2  e^{\beta |\mathbf{x}|^2} \, d\mathcal{H}^n.
\end{align*}
Similarly, as $\frac{1}{4}\leq \beta\leq \frac{3}{8}$, 
\begin{align*}
\int_{\Sigma} ( \mathscr{L}_{\Sigma, \beta}^0 u) (\mathscr{L}_\Sigma^0 u ) & e^{\beta |\mathbf{x}|^2} d\mathcal{H}^n =\int_{\Sigma} \left( (2\beta-\frac{1}{2})\mathbf{x}^\top \cdot \nabla_\Sigma u + \frac{1}{2} u+f\right) \left(\frac{1}{2} u+f\right)e^{\beta |\mathbf{x}|^2} d\mathcal{H}^n \\
&\leq \int_{\Sigma}\left(\frac{1}{4} |\mathbf{x}^\top \cdot \nabla_\Sigma u| \left(\frac{1}{2}|u|+|f|\right) + \left(\frac{1}{2} u+f\right)^2\right) e^{\beta |\mathbf{x}|^2} \, d\mathcal{H}^n \\
&\leq  \int_{\Sigma} \left(\frac{1}{128} |\mathbf{x}^\top \cdot \nabla_\Sigma u|^2+ 3\left(\frac{1}{2}|u|+|f|\right)^2 \right) e^{\beta |\mathbf{x}|^2} \, d\mathcal{H}^n\\
&\leq \frac{1}{8} \int_{\Sigma}|\nabla^2_\Sigma u| ^2  e^{\beta |\mathbf{x}|^2} \, d\mathcal{H}^n + 2\int_{\Sigma}\left( u^2 +4 f^2 \right) e^{\beta |\mathbf{x}|^2}  \, d\mathcal{H}^n,
\end{align*}
where the last inequality uses \eqref{PI1eqn} of Proposition \ref{GradControlsL2Prop}. 

Hence, we have arrived at
\begin{align*}
 \frac{1}{8}\int_{\Sigma} |\nabla_\Sigma^2 u|^2 e^{\beta |\mathbf{x}|^2}d\mathcal{H}^n &\leq  \int_{\Sigma} \left\{2 \left( u^2 +4 f^2\right) + \left(K_\Sigma+\frac{1}{2} \right)  |\nabla_{\Sigma} u|^2\right\}  e^{\beta |\mathbf{x}|^2} \, d\mathcal{H}^n.
\end{align*}
Combining this with the previous estimate \eqref{WeightW1Eqn} gives 
$$
\Vert u \Vert_{2}^{W,(\beta)} \leq C(K_\Sigma) \Vert f \Vert_{0}^{W,(\beta)}.
$$

The proof is completed by observing that for a fixed $f\in W^{0}_{\beta}(\Sigma)$, by a standard cutoff and mollifier trick, there is a sequence of smooth, compactly supported functions on $\Sigma$, $f_j$, so that $f_j\to f$ in $W^{ 0}_{\beta}(\Sigma)$. Hence, if $u_j$ is the unique solution of $\mathscr{L}_\Sigma u_j=f_j$, then $\set{u_j}$ is a Cauchy sequence in $W^2_\beta(\Sigma)$. Hence, letting $j\to \infty$ one obtains a solution $u\in W^2_\beta(\Sigma)$ of $\mathscr{L}_\Sigma u=f$. The uniqueness of such solutions are ensured by the uniqueness of the solutions in $W^1_{\frac{1}{4}}(\Sigma)$.
\end{proof}

We recall the Banach space $\mathcal{D}^{k,\alpha}(\Sigma)$ introduced in Section 5 of \cite{BWBanachManifold}:
$$
\mathcal{D}^{k,\alpha}(\Sigma)=\set{f\in C^{k,\alpha}_1\cap C^{k-1,\alpha}_{0} \cap C^{k-2,\alpha}_{-1}(\Sigma)\colon \mathbf{x}\cdot\nabla_\Sigma f\in C^{k-2,\alpha}_{-1}(\Sigma)}
$$
equipped with the norm
$$
\Vert f\Vert^*_{k,\alpha}=\Vert f\Vert^{(-1)}_{k-2,\alpha}+\sum_{i=k-1}^k \Vert\nabla_\Sigma^i f\Vert^{(1-k)}_{\alpha}+\Vert\mathbf{x}\cdot\nabla_\Sigma f\Vert^{(-1)}_{k-2,\alpha}.
$$
We will adopt the convention that if $X_1$ and $X_2$ are Banach spaces, then $X_1\cap X_2$ is the Banach space with the norm
$$
\Vert f\Vert_{X_1\cap X_2}=\Vert f\Vert_{X_1}+\Vert f\Vert_{X_2}.
$$

\begin{cor} \label{FredholmCor}
Let $\Sigma$ and $\mathbf{v}$ be as given in Theorem \ref{ModifiedSmoothDependThm}. For $\beta\in \left[\frac{1}{4}, \frac{3}{8}\right]$ the maps
$$
L_{\Sigma}, L_{\Sigma, \mathbf{v}}^\prime\colon \mathcal{D}^{k, \alpha}\cap W^2_{\beta}(\Sigma)\to C^{k-2, \alpha}_{-1}\cap W^0_\beta (\Sigma)
$$
are both Fredholm of index $0$. 
\end{cor}

\begin{proof}
By Proposition \ref{ModifiedIsoProp} and \cite[Theorem 5.7]{BWBanachManifold} the map 
$$
\mathscr{L}_\Sigma\colon \mathcal{D}^{k,\alpha}\cap W^2_\beta(\Sigma)\to C^{k-2,\alpha}_{-1}\cap W^0_\beta(\Sigma)
$$
is an isomorphism. As $\Sigma$ is asymptotically conical, we have $|A_\Sigma|\in C^{k-2,\alpha}_{-1}(\Sigma)$. Thus the Poincar\'{e} inequality, \eqref{PI1eqn} of Proposition \ref{GradControlsL2Prop}, and the Arzel\`{a}-Ascoli theorem imply that $L_{\Sigma}$ and $\mathscr{L}_{\Sigma}$ differ by a compact operator. This proves  that $L_{\Sigma}$ is Fredholm of index zero. 

Observe that $L_{\Sigma,\mathbf{v}}^\prime=M_\mathbf{v}\circ L_{\Sigma,\mathbf{v}}$ where $M_{\mathbf{v}}[f]=(\mathbf{v}\cdot\mathbf{n}_\Sigma)^{-2} f$ is the operator given by multiplication by $(\mathbf{v}\cdot\mathbf{n}_\Sigma)^{-2}$. As $|\mathbf{v}\cdot\mathbf{n}_\Sigma|\in C^{k-1,\alpha}_0(\Sigma)$ and has a strictly positive lower bound, it is not hard to see that $M_\mathbf{v}$ is an isomorphism of $C^{k-2,\alpha}_{-1}\cap W_\beta^0(\Sigma)$. And the Poincar\'{e} inequality, \eqref{PI1eqn} of Proposition \ref{GradControlsL2Prop}, together with Lemma 5.10 and Proposition 5.11 of \cite{BWBanachManifold} implies that $L_{\Sigma,\mathbf{v}}$ and $\mathscr{L}_\Sigma$ differ by a compact operator from $\mathcal{D}^{k, \alpha}\cap W^2_{\beta}(\Sigma)$ to $C^{k-2, \alpha}_{-1}\cap W^0_\beta (\Sigma)$. Hence $L^\prime_{\Sigma,\mathbf{v}}$ is Fredholm of index $0$.
\end{proof}

\begin{lem}\label{CokerLem} 
Let $\Sigma$ and $\mathbf{v}$ be as given in Theorem \ref{ModifiedSmoothDependThm}. If $\kappa \in \mathcal{K}_\mathbf{v}\setminus\set{0}$, then there are no solutions in $ W^2_{\frac{1}{4}} (\Sigma)$ to $L_{\Sigma,\mathbf{v}}^\prime u=\kappa$.
\end{lem}

\begin{proof}
We argue by contradictions. If there were such a solution, then integration by parts would be valid by the fact that $u\in W^2_{\frac{1}{4}}(\Sigma)$ and Proposition \ref{DecayEigenfunProp}. Thus,
$$
B_{\Sigma, \mathbf{v}}[\kappa, \kappa]=B_{\Sigma, \mathbf{v}}[\kappa, L_{\Sigma, \mathbf{v}}^\prime u]=-Q_{\Sigma, \mathbf{v}}[\kappa, u]=B_{\Sigma, \mathbf{v}}[L_{\Sigma, \mathbf{v}}^\prime  \kappa, u]=0.
$$
That is, $\kappa$ is identically zero.
\end{proof}

For $\Sigma$ and $\mathbf{v}$ as given in Theorem \ref{ModifiedSmoothDependThm}, by Corollary \ref{FredholmCor}, $\dim \mathcal{K}_\mathbf{v}<\infty$. Let $\Pi_{\mathbf{v}}$ be the orthogonal projection to $\mathcal{K}_\mathbf{v}$ with respect to $B_{\Sigma,\mathbf{v}}$ and $\Pi^\perp_{\mathbf{v}}=\mathrm{Id}-\Pi_{\mathbf{v}}$. Recall the smooth map obtained in \cite[Theorem 7.1]{BWBanachManifold}, $F_\mathbf{v}$ from a neighborhood of $(\mathbf{x}|_{\mathcal{L}(\Sigma)},0)$ in $C^{k,\alpha}(\mathcal{L}(\Sigma); \mathbb{R}^{n+1})\times\mathcal{K}_\mathbf{v}$ to a neighborhood of $\mathbf{x}|_\Sigma$ in $\mathcal{ACH}^{k,\alpha}_n(\Sigma)$ so that
$$
F_\mathbf{v}[\mathbf{x}|_{\mathcal{L}(\Sigma)},0]=\mathbf{x}|_\Sigma, \Xi_\mathbf{v}[F_\mathbf{v}[\varphi,\kappa]]\in\mathcal{K}_\mathbf{v} \mbox{ and }  \mathscr{L}_\Sigma F_{\mathbf{v}}[\varphi,\kappa]\in C^{k-2,\alpha}_{-1}(\Sigma; \mathbb{R}^{n+1}).
$$
Denote by $\mathcal{B}_R(f; X)$ the open ball in Banach space $X$ with radius $R$ and center $f$. 

\begin{lem} \label{SmoothMapLem}
Let $\Sigma$ and $\mathbf{v}$ be as given in Theorem \ref{ModifiedSmoothDependThm}. For $\beta\in \left(0,\frac{1}{2}\right)$ and sufficiently small $r_0>0$ the map 
$$
\hat{\Theta}_\mathbf{v}\colon \mathcal{B}_{r_0}(\mathbf{x}|_{\mathcal{L}(\Sigma)}; C^{k,\alpha}(\mathcal{L}(\Sigma);\mathbb{R}^{n+1}))\times \mathcal{B}_{r_0}(0; \mathcal{D}^{k,\alpha}\cap W^2_\beta(\Sigma)) \to C^{k-2,\alpha}_{-1}\cap W^0_\beta(\Sigma)
$$
given by 
$$
\hat{\Theta}_\mathbf{v}[\varphi,u]=\Omega_{F_\mathbf{v}^\prime[\varphi,u]}^{-1}\Pi^\perp_{\mathbf{v}}\circ\hat{\Xi}_\mathbf{v}[F^\prime_\mathbf{v}[\varphi,u]] \mbox{ where $F_\mathbf{v}^\prime[\varphi,u]=F_\mathbf{v}[\varphi,0]+u\mathbf{v}$}
$$
is a well-defined smooth map.
\end{lem}

\begin{proof}
First, we write
\begin{align*}
\hat{\Theta}_\mathbf{v}[\varphi,u] & =(\mathbf{v}\cdot\mathbf{n}_\Sigma)^{-2} \left(\Xi_\mathbf{v}[F_\mathbf{v}^\prime[\varphi,u]]-\Xi_\mathbf{v}[F_\mathbf{v}[\varphi,0]]\right)+(\mathbf{v}\cdot\mathbf{n}_\Sigma)^{-2}\Xi_\mathbf{v}[F_\mathbf{v}[\varphi,0]] \\
& \quad -\Omega^{-1}_{F_\mathbf{v}^\prime[\varphi,u]} \Pi_{\mathbf{v}}\circ\hat{\Xi}_\mathbf{v}[F_\mathbf{v}^\prime[\varphi,u]].
\end{align*}
By Proposition \ref{DecayEigenfunProp}, $\mathcal{K}_{\mathbf{v}}\subset\mathcal{D}^{k,\alpha}\cap W^2_\beta(\Sigma)$. By \cite[Lemmas 7.2 and 7.3]{BWBanachManifold}, the map 
$$
(\varphi,u)\mapsto\Xi_\mathbf{v}[F_\mathbf{v}^\prime[\varphi,u]]-\Xi_\mathbf{v}[F_\mathbf{v}[\varphi,0]]
$$
is a smooth map from a neighborhood of $(\mathbf{x}|_{\mathcal{L}(\Sigma)},0)$ in $C^{k,\alpha}(\mathcal{L}(\Sigma))\times (\mathcal{D}^{k,\alpha}\cap W^2_\beta(\Sigma))$ to $C^{k-2,\alpha}_{-1}\cap W^0_\beta(\Sigma)$. Similarly, one shows that for $r_0>0$ sufficiently small the maps
$$
\varphi\mapsto \Xi_\mathbf{v}[F_\mathbf{v}[\varphi,0]] \mbox{ and } (\varphi,u)\mapsto\Pi_{\mathbf{v}}\circ\hat{\Xi}_\mathbf{v}[F_\mathbf{v}^\prime[\varphi,u]]
$$
are both smooth maps into $\mathcal{K}_\mathbf{v}$. Finally, we observe that, in view of Proposition \ref{DecayEigenfunProp}, the operator given by multiplication by $\Omega_{\mathbf{f}}^{-1}$ for $\mathbf{f}\in\mathcal{ACH}^{k,\alpha}_n(\Sigma)$ with $\mathrm{tr}^1_\infty[\mathbf{f}]$ sufficiently close to $\mathbf{x}|_{\mathcal{L}(\Sigma)}$ is a smooth map from $\mathcal{K}_\mathbf{v}$ to $C^{k-2,\alpha}_{-1}\cap W^0_\beta(\Sigma)$, and that the operator given by multiplication by $(\mathbf{v}\cdot\mathbf{n}_\Sigma)^{-2}$ is an isomorphism of $C^{k-2,\alpha}_{-1}\cap W^0_\beta(\Sigma)$. Therefore we have shown $\hat{\Theta}_\mathbf{v}$ is a smooth map.
\end{proof}

We are now ready to prove the main result of this section.

\begin{proof}[Proof of Theorem \ref{ModifiedSmoothDependThm}]
First, Corollary \ref{FredholmCor} gives $\dim \mathcal{K}_\mathbf{v}<\infty$, and Proposition \ref{DecayEigenfunProp} gives $\mathcal{K}_\mathbf{v}\subset\mathcal{D}^{k,\alpha}\cap W^2_{\frac{3}{8}}(\Sigma)$. Let
\begin{align*}
\hat{\mathcal{K}}_{\mathbf{v}}^\perp & =\left\{f\in C^{k-2,\alpha}_{-1}\cap W^0_{\frac{3}{8}} (\Sigma)\colon B_{\Sigma, \mathbf{v}}[f, \kappa]=0 \mbox{ for all $\kappa\in\mathcal{K}_{\mathbf{v}}$}\right\}, \mbox{ and}, \\
\hat{\mathcal{K}}_{\mathbf{v},*}^\perp & =\left\{f\in\mathcal{D}^{k,\alpha}\cap W^2_{\frac{3}{8}}(\Sigma)\colon B_{\Sigma, \mathbf{v}}[f, \kappa]=0 \mbox{ for all $\kappa\in\mathcal{K}_{\mathbf{v}}$}\right\}.
\end{align*}
We consider the map
$$
\hat{\Theta}_{\mathbf{v}}^\prime\colon \mathcal{B}_{r_0} (\mathbf{x}|_{\mathcal{L}(\Sigma)}; C^{k,\alpha}(\mathcal{L}(\Sigma); \mathbb{R}^{n+1}))\times\mathcal{B}_{\frac{r_0}{2}}(0; \mathcal{K}_{\mathbf{v}})\times\mathcal{B}_{\frac{r_0}{2}} (0; \hat{\mathcal{K}}^\perp_{\mathbf{v},*}) \to \hat{\mathcal{K}}^\perp_{\mathbf{v}}
$$
defined by 
$$
\hat{\Theta}_{\mathbf{v}}^\prime [\varphi,\kappa,u]=\hat{\Theta}_\mathbf{v}[\varphi,\kappa+u].
$$
By Lemma \ref{SmoothMapLem}, this is a well-defined smooth map.  
 
By \eqref{2ndVarEqn}, one sees that $D_3\hat{\Theta}_{\mathbf{v}}^\prime (\mathbf{x}|_{\mathcal{L}(\Sigma)},0,0)$ is given by $L_{\Sigma,\mathbf{v}}'$ restricted to $\hat{\mathcal{K}}^\perp_{\mathbf{v},*}$. Thus, it follows from Corollary \ref{FredholmCor} and Lemma \ref{CokerLem} that $D_3\hat{\Theta}_{\mathbf{v}}^\prime (\mathbf{x}|_{\mathcal{L}(\Sigma)},0,0)$ is an isomorphism between $\hat{\mathcal{K}}^\perp_{\mathbf{v},*}$ and $\hat{\mathcal{K}}^\perp_{\mathbf{v}}$. Hence, by the implicit function theorem, there is a neighborhood $\mathcal{U}=\mathcal{U}_1\times\mathcal{U}_2$ of $(\mathbf{x}|_{\mathcal{L}(\Sigma)},0)$ in $C^{k,\alpha}(\mathcal{L}(\Sigma); \mathbb{R}^{n+1})\times\mathcal{K}_{\mathbf{v}}$ and a {smooth} map $\hat{F}_{\mathbf{v},*}\colon \mathcal{U}\to\hat{\mathcal{K}}^\perp_{\mathbf{v},*}$ that gives the unique solution in a neighborhood of $0\in\hat{\mathcal{K}}_{\mathbf{v},*}^\perp$ to
$$
\hat{\Theta}_{\mathbf{v}}^\prime[\varphi,\kappa,\hat{F}_{\mathbf{v},*}[\varphi,\kappa]]=0.
$$

Now let 
\begin{align*}
\hat{F}_{\mathbf{v}}[\varphi,\kappa] & =F_{\mathbf{v}}[\varphi, 0]+(\kappa+\hat{F}_{\mathbf{v},*}[\varphi,\kappa])\mathbf{v}, \mbox{ and}, \\
\hat{G}_{\mathbf{v}}[\varphi,\kappa] & =\Pi_{\mathbf{v}}\circ\hat{\Xi}_{ \mathbf{v}}[\hat{F}_{\mathbf{v}}[\varphi, \kappa]].
\end{align*}
 Items \eqref{TraceFItem}, \eqref{IdentityItem}, \eqref{EStationaryItem} and \eqref{ExpanderMeanCurvatureItem} follow directly from the definitions of $\hat{F}_{\mathbf{v}}$ and $\hat{G}_{\mathbf{v}}$. Furthermore, Items  \eqref{FDiffItem} and \eqref{SubmanifoldItem} follow the exact same proofs as in \cite[Theorem 7.1]{BWBanachManifold}; for readers' convenience we include them below.

To establish Item \eqref{FDiffItem} it suffices to show $D_2\hat{F}_{\mathbf{v},*}(\mathbf{x}|_{\mathcal{L}(\Sigma)},0)=0$. Take any $\kappa\in\mathcal{K}_\mathbf{v}$. Observe that for $|s|<\epsilon$,
\begin{equation}  \label{ProjExpanderMCEqn}
\Pi^\perp_\mathbf{v}\circ\hat{\Xi}_\mathbf{v}[\mathbf{x}|_\Sigma+(s\kappa+\hat{F}_{\mathbf{v},*}[\mathbf{x}|_{\mathcal{L}(\Sigma)},s\kappa])\mathbf{v}]=0.
\end{equation}
Differentiating \eqref{ProjExpanderMCEqn} at $s=0$, we use \eqref{2ndVarEqn} and the chain rule to get
$$
\Pi^\perp_\mathbf{v}\circ L_{\Sigma,\mathbf{v}}^\prime\circ D_2\hat{F}_{\mathbf{v},*}(\mathbf{x}|_{\mathcal{L}(\Sigma)},0)\kappa=0.
$$
It follows from Lemma \ref{CokerLem} and the definition of $\hat{F}_{\mathbf{v},*}$ that 
$$
D_2\hat{F}_{\mathbf{v},*}(\mathbf{x}|_{\mathcal{L}(\Sigma)},0)\kappa\in\mathcal{K}_\mathbf{v}\cap\mathcal{K}_{\mathbf{v},*}^\perp=\set{0},
$$
proving the claim.

Next, Items \eqref{FDiffItem} and \eqref{ExpanderMeanCurvatureItem} together with \eqref{2ndVarEqn} imply that
$$
D_2\hat{G}_\mathbf{v}(\mathbf{x}|_{\mathcal{L}(\Sigma)},0)\kappa=L^\prime_{\Sigma,\mathbf{v}}[\mathbf{v}\cdot D_2\hat{F}_{\mathbf{v}}(\mathbf{x}|_{\mathcal{L}(\Sigma)},0)\kappa]=L^\prime_{\Sigma,\mathbf{v}}\kappa=0.
$$
Similarly, one uses the more general second variation formula, \cite[Proposition 4.2]{BWBanachManifold}, to compute that, for $\zeta\in C^{k,\alpha}(\mathcal{L}(\Sigma))$,
$$
D_1\hat{G}_\mathbf{v}(\mathbf{x}|_{\mathcal{L}(\Sigma)},0)[\zeta\mathbf{V}]=(\mathbf{v}\cdot\mathbf{n}_\Sigma)^{-1}L_{\Sigma}[\mathbf{n}_\Sigma\cdot D_1\hat{F}_\mathbf{v}(\mathbf{x}|_{\mathcal{L}(\Sigma)},0)[\zeta\mathbf{V}]].
$$
Thus, arguing as in \cite[Theorem 7.1]{BWBanachManifold}, it follows that $D_1\hat{G}_{\mathbf{v}}(\mathbf{x}|_{\mathcal{L}(\Sigma)},0)$ is of rank equal to $\dim \mathcal{K}_\mathbf{v}$. Hence we have shown Item \eqref{SubmanifoldItem}.
\end{proof}

It is convenient to let
$$
\hat{\mathcal{F}}_{\mathbf{v}}[\varphi, \kappa]=\mathbf{v}\cdot \left(\hat{F}_{\mathbf{v}}[\varphi, \kappa]-\hat{F}_{\mathbf{v}}[\varphi, 0]\right).
$$
With this notation
$$
\hat{F}_{\mathbf{v}}[\varphi, \kappa]=\hat{F}_{\mathbf{v}}[\varphi, 0]+ \hat{\mathcal{F}}_{\mathbf{v}}[\varphi, \kappa] \mathbf{v}.
$$
It follows from the proof of Theorem \ref{ModifiedSmoothDependThm}, that
\begin{equation} \label{OrthogFcalEqn}
\hat{\mathcal{F}}_{\mathbf{v}}[\varphi, \kappa]-\kappa=\hat{F}_{\mathbf{v}, *}[\varphi, \kappa]-\hat{F}_{\mathbf{v},*}[\varphi,0]\in \hat{\mathcal{K}}_{\mathbf{v}, *}^\perp\subset \mathcal{D}^{k, \alpha}\cap W^2_{\frac{3}{8}}(\Sigma).
\end{equation}
Moreover, 
$$
\hat{\Xi}_{\hat{F}_{\mathbf{v}}[\varphi, 0], \mathbf{v}}[\hat{\mathcal{F}}_{\mathbf{v}}[\varphi, \kappa]]=\hat{\Xi}_{\hat{F}_{\mathbf{v}}[\varphi, \kappa], \mathbf{v}}[0]=\hat{G}_{\mathbf{v}}[\varphi, \kappa].
$$
In particular, if $\hat{G}_\mathbf{v}[\varphi,\kappa]=0$, i.e., $\hat{F}_\mathbf{v}[\varphi,\kappa]$ is $E$-stationary, then
\begin{equation}\label{DiffIdentEqn}
\begin{split}
D_2 \hat{G}_{\mathbf{v}}(\varphi,\kappa) & =D\hat{\Xi}_{\hat{F}_{\mathbf{v}}[\varphi, \kappa],\mathbf{v}}(0)\circ D_2 \hat{\mathcal{F}}_{\mathbf{v}}(\varphi,\kappa) \\
& =\Omega_{\hat{F}_{\mathbf{v}}[\varphi, \kappa],\mathbf{v}}L_{\hat{F}_{\mathbf{v}}[\varphi, \kappa], \mathbf{v}}^\prime\circ D_2 \hat{\mathcal{F}}_{\mathbf{v}}(\varphi,\kappa).
\end{split}
\end{equation}

\section{Index and nullity} \label{IndexSec}
We recall the notion of index and nullity for asymptotically conical self-expanders and relate these integers to certain other spectral invariants. As observed in Section \ref{VariationSubsec}, the expander equation \eqref{ExpanderEqn} is an elliptic variational problem and so, as is usual for such problems, we define the \emph{(Morse) index} of an asymptotically conical  self-expander, $\Sigma$, to be
$$
\mathrm{ind}(\Sigma)=\sup \set{\dim V\colon V\subset C^{2}_c(\Sigma) \mbox{ so that } Q_{\Sigma}[u]<0, \forall u\in V\backslash \set{0}}.
$$
The \emph{nullity} of $\Sigma$, $\mathrm{null}(\Sigma)$ is defined to be the dimension of the kernel of $L_\Sigma$,
$$
\mathcal{K}=\set{\kappa\in W^{1}_{\frac{1}{4}}(\Sigma)\colon L_\Sigma\kappa=0}.
$$
Observe, that if $\mathbf{v}\in C^{k,\alpha}_0\cap C^{k}_{0,\mathrm{H}}(\Sigma;\mathbb{R}^{n+1})$ is a transverse section on $\Sigma$ so that $\mathcal{C}[\mathbf{v}]$ is transverse to $\mathcal{C}(\Sigma)$, then $\dim\mathcal{K}_\mathbf{v}=\dim\mathcal{K}$. One readily checks that $B_{\Sigma}$ and $B_{\Sigma, \mathbf{v}}$ extend to $W^0_{\frac{1}{4}}(\Sigma)$, and $D_{\Sigma}$, $D_{\Sigma, \mathbf{v}}$ $Q_{\Sigma}$ and $Q_{\Sigma, \mathbf{v}}$ extend to $W_{\frac{1}{4}}^1(\Sigma)$. 

\begin{lem} \label{SpectrumLem}
Let $\Sigma\in\mathcal{ACH}^{k,\alpha}_n$ be a self-expander, and let $\mathbf{v}\in C^{k,\alpha}_0\cap C^{k}_{0,\mathrm{H}}(\Sigma;\mathbb{R}^{n+1})$ be a transverse section on $\Sigma$ so that $\mathcal{C}[\mathbf{v}]$ is transverse to $\mathcal{C}(\Sigma)$. The operators 
$$
-L_\Sigma, -L_{\Sigma,\mathbf{v}}^\prime \colon W^{2}_{\frac{1}{4}}(\Sigma)\to W^{0}_{\frac{1}{4}}(\Sigma)
$$
are self-adjoint in $\left(W^0_{\frac{1}{4}}(\Sigma),B_\Sigma\right)$ and $\left(W^0_{\frac{1}{4}}(\Sigma),B_{\Sigma,\mathbf{v}}\right)$, respectively, and both have discrete spectrums. 
\end{lem}

\begin{proof}
First we observe that $W^2_{\frac{1}{4}}(\Sigma)$ is dense in $\left(W^0_{\frac{1}{4}}(\Sigma),B_\Sigma\right)$. If we choose $\gamma>0$ so 
$$
\sup_{\Sigma}|A_\Sigma|^2 \leq \gamma,
$$
then the Lax-Milgram theorem and Proposition \ref{ModifiedIsoProp} imply, that given an $f\in W^0_{\frac{1}{4}}(\Sigma)$ there is a unique solution $u\in W^2_{\frac{1}{4}}(\Sigma)$ to $(-L_\Sigma+\gamma)u=f$. Moreover, 
$$
\Vert u\Vert^{W,(\frac{1}{4})}_2 \leq C(\gamma) \Vert f\Vert_{0}^{W,(\frac{1}{4})}.
$$
That is, one can define a bounded operator, $R_\gamma=(-L_\Sigma+\gamma)^{-1}$ in $\left(W^0_{\frac{1}{4}}(\Sigma), B_\Sigma\right)$. Moreover, $R_\gamma$ is self-adjoint because $-L_\Sigma$ is symmetric with respect to $B_\Sigma$. Now one writes $-L_\Sigma=R_\gamma^{-1}-\gamma$. It is easy to check (see, e.g., \cite[pages 107-108]{Grigoryan}) that $R_\gamma^{-1}$ is a self-adjoint operator in $\left(W^0_{\frac{1}{4}}(\Sigma), B_\Sigma\right)$ and so is $-L_\Sigma$. 

To prove the spectrum of $-L_\Sigma$ is discrete, we follow the arguments in \cite[Theorem 10.20]{Grigoryan}. As the natural embedding of $W^2_{\frac{1}{4}}(\Sigma)$ into $W^0_{\frac{1}{4}}(\Sigma)$ is compact by the Poincar\'{e} inequality, \eqref{PI1eqn} of Proposition \ref{GradControlsL2Prop}, it follows that $R_\gamma$ is a compact operator. As the kernel of $R_\gamma$ is trivial, the Hilbert-Schmidt theorem implies that there is a countable orthonormal basis $\set{\psi_i}$ of $\left(W^0_{\frac{1}{4}}(\Sigma), B_\Sigma\right)$ that consists of the eigenfunctions $\psi_i$ of $R_\gamma$ with eigenvalues $\rho_i>0$ so $\rho_i\to 0$. Clearly, the $\psi_i$ are also eigenfunctions of $-L_\Sigma$ with eigenvalues $\mu_i=\rho_i^{-1}-\gamma$, and $\set{\mu_i}$ has no accumulating points and $\mu_i\to\infty$. Observe, that for $\mu\neq\mu_i$, if $f=\sum_{i} a_i \psi_i$, then
$$
(L_\Sigma-\mu)^{-1} f=\sum_{i} \frac{a_i}{\mu_i-\mu}\psi_i
$$
and it is bounded as $\inf_i |\mu_i-\mu|>0$. Hence the entire spectrum of $L_\Sigma$ is given by $\set{\mu_i}$, implying the discreteness of the spectrum of $L_\Sigma$.

As long as one observes that $a_{\Sigma,\mathbf{v}}^\prime$ is uniformly bounded (cf. \cite[(5.21)]{BWBanachManifold}), the proof for $L_{\Sigma,\mathbf{v}}^\prime$ is essentially the same.
\end{proof}

This allows us to relate the index to spectral properties of $L_{\Sigma}$ and $L_{\Sigma,\mathbf{v}}^\prime$. Indeed, 

\begin{prop} \label{IndexProp}
Let $\Sigma\in\mathcal{ACH}^{k,\alpha}_n$ be a self-expander, and let $\mathbf{v}\in C^{k,\alpha}_0\cap C^{k}_{0,\mathrm{H}}(\Sigma;\mathbb{R}^{n+1})$ be a transverse section on $\Sigma$ so that $\mathcal{C}[\mathbf{v}]$ is transverse to $\mathcal{C}(\Sigma)$. Then one has that $\mathrm{ind}(\Sigma)$ is equal to:
\begin{enumerate}
\item The maximum dimension of subspaces of $W^1_\beta(\Sigma)$, $\beta\geq \frac{1}{4}$ on which ${Q}_{\Sigma}$ or ${Q}_{\Sigma,\mathbf{v}}$ is negative definite.
\item The number of negative eigenvalues (counted with multiplicities) of $-L_{\Sigma}$ or $-L_{\Sigma,\mathbf{v}}^\prime$ in $W^1_\beta(\Sigma)$ for any $\beta\in \left(0,\frac{1}{2}\right)$.
\end{enumerate}
\end{prop}

\begin{proof}
As $C^2_c(\Sigma)$ is dense in $W^1_{\frac{1}{4}}(\Sigma)$ and $|A_\Sigma|$ is uniformly bounded, it follows from the dominated convergence theorem that $\mathrm{ind}(\Sigma)$ is equal to the maximum dimension of subspaces of $W^1_{\frac{1}{4}}(\Sigma)$ on which $Q_\Sigma$ is negative definite. Observe, that the map $M_\mathbf{v}$ given by the multiplication by $(\mathbf{v}\cdot\mathbf{n}_\Sigma)^{-1}$ is an isomorphism of $W^1_{\frac{1}{4}}(\Sigma)$ and $Q_{\Sigma,\mathbf{v}}[M_\mathbf{v} u, M_\mathbf{v} v]=Q_{\Sigma}[u,v]$. Thus, $Q_\Sigma$ is negative definite on a subspace $V\subset W^1_{\frac{1}{4}}(\Sigma)$ if and only if $Q_{\Sigma,\mathbf{v}}$ is so on $M_\mathbf{v}(V)$. Hence the first claim follows from the trivial inclusions for $\beta\geq \frac{1}{4}$, $C^\infty_c(\Sigma)\subset W^1_\beta(\Sigma)\subset W^1_{\frac{1}{4}}(\Sigma)$.

To prove the second, we appeal to Lemma \ref{SpectrumLem} and the min-max principle for self-adjoint operators. It follows that $\mathrm{ind}(\Sigma)$ is equal to the number of negative eigenvalues of $-L_\Sigma$ or $-L_{\Sigma,\mathbf{v}}^\prime$ in $W^1_{\frac{1}{4}}(\Sigma)$. Invoking Proposition \ref{DecayEigenfunProp}, we have that eigenfunctions of $-L_\Sigma$ or $-L_{\Sigma,\mathbf{v}}^\prime$ in $W^1_{\frac{1}{4}}(\Sigma)$ are in $W^1_{\beta}(\Sigma)$ for any $\beta\in\left(0,\frac{1}{2}\right)$ and vice versa. This proves the second claim. 
\end{proof}

We extend the definition of index and nullity to parameterizations in a natural way. For $\Sigma\in\mathcal{ACH}^{k,\alpha}_n$, if $\mathbf{f}\in\mathcal{ACH}^{k,\alpha}_n(\Sigma)$ is $E$-stationary, i.e., $\Lambda=\mathbf{f}(\Sigma)$ is a self-expander, then we define 
$$
\mathrm{ind}(\mathbf{f})=\mathrm{ind}(\Lambda) \mbox{ and }\mathrm{null}(\mathbf{f})=\mathrm{null}(\Lambda),
$$
and $\mathrm{ind}([\mathbf{f}])$ and $\mathrm{null}([\mathbf{f}])$ are defined likewise. Similarly, we have the following:

\begin{lem} \label{ParaSpectrumLem}
Let $\Sigma\in\mathcal{ACH}^{k,\alpha}_n$ and $\mathbf{v}\in C^{k,\alpha}_0\cap C^{k}_{0,\mathrm{H}}(\Sigma;\mathbb{R}^{n+1})$ a transverse section on $\Sigma$ so that $\mathcal{C}[\mathbf{v}]$ is transverse to $\mathcal{C}(\Sigma)$. If $\mathbf{f}\in\mathcal{ACH}^{k,\alpha}_n(\Sigma)$ is $E$-stationary, and $\mathbf{v}\circ\mathbf{f}^{-1}$ and $\mathcal{C}[\mathbf{v}]\circ\mathcal{C}[\mathbf{f}]^{-1}$ are transverse to, respectively, $\Lambda=\mathbf{f}(\Sigma)$ and $\mathcal{C}(\Lambda)$, then the operators
$$
-L_{\mathbf{f}}, -L_{\mathbf{f},\mathbf{v}}^\prime\colon \mathbf{f}^* W^2_{\frac{1}{4}}(\Lambda)\to \mathbf{f}^* W^0_{\frac{1}{4}}(\Lambda)
$$
are self-adjoint in $\left( \mathbf{f}^* W^0_{\frac{1}{4}}(\Lambda), B_\mathbf{f}\right)$ and $\left( \mathbf{f}^* W^0_{\frac{1}{4}}(\Lambda), B_{\mathbf{f},\mathbf{v}}\right)$, respectively, and both have discrete spectrums.
\end{lem}

\begin{prop} \label{ParaIndexProp}
Let $\Sigma\in\mathcal{ACH}^{k,\alpha}_n$ and $\mathbf{v}\in C^{k,\alpha}_0\cap C^{k}_{0,\mathrm{H}}(\Sigma;\mathbb{R}^{n+1})$ a transverse section on $\Sigma$ so that $\mathcal{C}[\mathbf{v}]$ is transverse to $\mathcal{C}(\Sigma)$. There is an $\epsilon_0>0$ sufficiently small, depending only on $K_\Sigma$ and $\inf_\Sigma|\mathbf{v}\cdot\mathbf{n}_\Sigma|$, so that if $\mathbf{f}\in \mathcal{B}_{\epsilon_0}(\mathbf{x}|_\Sigma; \mathcal{ACH}^{k,\alpha}_n(\Sigma))$ and is $E$-stationary, then $\mathrm{ind}(\mathbf{f})$ is equal to
\begin{enumerate}
\item The maximum dimension of subspaces of $C^2_c(\Sigma)$ on which ${Q}_{\mathbf{f}}$ is negative definite.
\item The maximum dimension of subspaces of $W^1_{\beta}(\Sigma)$, $\beta\geq \frac{3}{8}$ on which ${Q}_{\mathbf{f}}$ or ${Q}_{\mathbf{f},\mathbf{v}}$ is negative definite.
\item The number of negative eigenvalues (counted with multiplicities) of $-L_{\mathbf{f}}$ or $-L_{\mathbf{f},\mathbf{v}}^\prime$ in $W^1_{\beta}(\Sigma)$ for any $\beta\in \left[\frac{1}{8}, \frac{3}{8}\right]$.
\end{enumerate}
\end{prop}

\section{Concentration estimates for eigenfunctions and stability of certain quadratic forms}  \label{ConcentrationSec}
The main goal of this section is to show that, for eigenfunctions of $-L_\Sigma$ with small eigenvalue, one can prove effective estimates on concentration of such eigenfunctions in a compact ball.  Here $\Sigma$ is an asymptotically conical self-expander as in Section \ref{DecayEigenfunSubsec} and $R_0=R_0(\Sigma)$ is chosen so \eqref{REstimatesEqn} and \eqref{AEstimatesEqn} hold.  We also let $K_\Sigma$ be given by \eqref{KSigmaDef}. The estimates will depend only on $K_\Sigma, n$ and the measure of smallness of the eigenvalue. This will be crucial to showing that the quadratic forms from Section \ref{NotationSec}, are stable for perturbations of $\Sigma$, even those that change the asymptotic cone, when one entry is such an eigenfunction.

Let $S_\rho=\Sigma\cap\partial B_\rho$ and
$$
L(\rho)=\int_{S_{\rho}} u^2 |\nabla_\Sigma r| \, d\mathcal{H}^{n-1},
$$
and for $\gamma\in \Real$,
$$
L_\gamma(\rho)=e^{\gamma \rho^2}L(\rho).  
$$
Notice that, by the co-area formula, when $\rho\geq R_0$, 
\begin{equation}\label{CoAreaEqn}
 \frac{1}{2}\int_{\bar{E}_\rho} u^2   e^{\gamma r^2} d\mathcal{H}^{n} \leq \int_{\bar{E}_\rho} u^2 e^{\gamma r^2} |\nabla_\Sigma r|^2 d\mathcal{H}^{n}=\int_\rho^\infty L_\gamma(t) dt \leq  \int_{\bar{E}_\rho} u^2 e^{\gamma r^2} d\mathcal{H}^{n}.
\end{equation}

\begin{lem}\label{ConcentrationLem}
Let $\Sigma\in\mathcal{ACH}^{k,\alpha}_n$ be a self-expander. Fix an $\epsilon>0$ and $\mu\leq \frac{1}{4}$. If $u\in W^1_{\frac{1}{4}}(\Sigma)$ satisfies 
$$
-L_{\Sigma} u=\mu u,
$$
then there is an $R_1=R_1(K_\Sigma, n, \epsilon)\geq R_0$ so that 
$$
\int_{\Sigma\backslash B_{R_1}} \left(|\nabla_\Sigma u|^2+u^2\right) e^{\frac{3}{8} |\mathbf{x}|^2} \, d\mathcal{H}^n\leq \epsilon \int_{\Sigma} u^2  e^{\frac{|\mathbf{x}|^2}{4}} \, d\mathcal{H}^n.
$$
\end{lem}

\begin{proof}
First observe that, by Proposition \ref{DecayEigenfunProp}, $u\in W^2_{\beta}(\Sigma)$ for all $\beta\in \left(0, \frac{1}{2}\right)$. This fact, with $\beta=\frac{7}{16}$, will justify several integrations by parts. The first variation formula, \eqref{REstimatesEqn} and \eqref{AEstimatesEqn} give, for $\rho\geq R_0$ and $\gamma\leq \frac{1}{2}$, that
\begin{align*}
{L}'_\gamma(\rho)&\leq 2 e^{\gamma \rho^2} \int_{S_\rho} u \nabla_\Sigma u\cdot\frac{\nabla_\Sigma r}{|\nabla_\Sigma r|} \, d\mathcal{H}^{n-1}+\frac{n}{\rho} {L}_\gamma(\rho)+2\gamma \rho L_{\gamma}(\rho)  \\
&=-2 e^{\left(\gamma-\frac{1}{4}\right) \rho^2 } \int_{\bar{E}_\rho}\left( |\nabla_\Sigma u|^2  +u\mathscr{L}_\Sigma^0 u\right) e^{\frac{r^2}{4}} \, d\mathcal{H}^{n} +\frac{n}{\rho} {L}_\gamma(\rho)+2\gamma \rho L_{\gamma}(\rho)  \\
&=  -2 e^{\left(\gamma-\frac{1}{4}\right) \rho^2 }\int_{\bar{E}_\rho}\left( |\nabla_\Sigma u|^2  -|A_\Sigma|^2 u^2 +\frac{1}{2} u^2 +u L_\Sigma u\right) e^{\frac{r^2}{4}} \, d\mathcal{H}^{n}\\
& +\frac{n}{\rho} {L}_\gamma(\rho)+2\gamma \rho L_{\gamma}(\rho) \\
&\leq  -2(1-2\gamma) e^{\left(\gamma-\frac{1}{4}\right) \rho^2 }\int_{\bar{E}_\rho} |\nabla_\Sigma u|^2  e^{\frac{r^2}{4}} \, d\mathcal{H}^{n}+\frac{n}{\rho}{L}_\gamma(\rho)\\
&+2 e^{\left(\gamma-\frac{1}{4}\right) \rho^2 }\int_{\bar{E}_\rho} \left( |A_{\Sigma}|^2+\mu -\frac{1}{2}\right) u^2 e^{\frac{r^2}{4}} \, d\mathcal{H}^{n},
\end{align*}
where the  last inequality follows from the Poincar\'{e} inequality, Proposition \ref{GradControlsBoundaryProp}.  As $\gamma\leq \frac{1}{2}$, $\mu\leq \frac{1}{4}$ and $|A_{\Sigma}|^2\leq\frac{1}{4} $ on $E_{R_0}$ this means that, for $\rho\geq R_0$,
\begin{equation}\label{hatLdecay}
\frac{d}{d\rho}\left(\rho^{-n} {L}_\gamma(\rho)\right) \leq -2(1-2\gamma)\rho^{-n} e^{\left(\gamma-\frac{1}{4}\right) \rho^2 }\int_{\bar{E}_\rho} |\nabla_\Sigma u|^2  e^{\frac{r^2}{4}} \, d\mathcal{H}^{n}\leq 0.
\end{equation}
Hence, for $\rho\geq \rho_0\geq R_0$ and $\gamma\leq \frac{1}{2}$,
\begin{equation}\label{hatLdecayEst2}
\rho^{-n} L_{\gamma}(\rho)\leq \rho_0^{-n} L_{\gamma}(\rho_0) \mbox{ and } \frac{e^{\gamma\rho^2}}{\rho^n} \int_{S_\rho} u^2 \, d\mathcal{H}^{n-1} \leq \frac{2 e^{\gamma\rho_0^2}}{\rho_0^n} \int_{S_{\rho_0}} u^2 \, d\mathcal{H}^{n-1}.
\end{equation}
A consequence of this, the co-area formula and the mean value theorem is that, for $\rho\geq \rho_0 +1\geq R_0+1$ and $\gamma<\frac{1}{2}$,
\begin{equation}\label{TraceEst}
\rho^{-n}  e^{\gamma\rho^2} \int_{S_\rho} u^2 \, d\mathcal{H}^{n-1} \leq 2 \rho_0^{-n} \int_{\bar{E}_{\rho_0}} u^2  e^{\gamma r^2} \, d\mathcal{H}^{n} \leq 2 \rho_0^{-n} \int_{\Sigma} u^2  e^{\gamma r^2} \, d\mathcal{H}^{n} .
\end{equation}
Furthermore, combining \eqref{CoAreaEqn} and \eqref{hatLdecayEst2} one has for $\rho\geq \rho_1\geq R_0$ that
\begin{align*}
\int_{\bar{E}_\rho} u^2 e^{\frac{7}{16} {r^2}} \, d\mathcal{H}^n & \leq 2\int_{\rho}^\infty L_{7/16}(t) \, dt=2\int_{\rho}^\infty t^{-n} L_{1/2}(t) t^n e^{-\frac{t^2}{16}}\, dt \\
&\leq  2 \rho_{1}^{-n} L_{1/2}(\rho_1) \int_{\rho}^\infty t^n e^{-\frac{t^2}{16}} \, dt \\
&\leq C(n) \rho^{n-1} \rho_1^{-n}  e^{-\frac{\rho^2}{16}+\frac{\rho^2_1}{2}} \int_{S_{\rho_1}} u^2 \, d\mathcal{H}^{n-1}.
\end{align*}
Combining this with \eqref{TraceEst} for $\gamma=\frac{1}{4}$ and one has for $\rho\geq\rho_1\geq  \rho_0+1\geq R_0+1$, 
\begin{equation} \label{RefinedEst}
\begin{split}
\int_{\bar{E}_\rho} u^2 e^{\frac{7}{16} {r^2}} \, d\mathcal{H}^n &\leq 2 C(n) \rho^{n-1} \rho_0^{-n}  e^{-\frac{\rho^2}{16} +\frac{\rho^2_1}{4}}\int_{\bar{E}_{\rho_0}} u^2  e^{\frac{r^2}{4}} \, d\mathcal{H}^n.
\end{split}
\end{equation}
In particular, for any $\epsilon>0$, there is an $R=R(\epsilon, K_\Sigma,n) \geq R_0+1$ so that
$$
\int_{\bar{E}_{R}} u^2 e^{\frac{7}{16} r^2} \, d\mathcal{H}^n\leq \frac{\epsilon}{2} \int_{\Sigma} u^2 e^{\frac{r^2}{4}} \, d\mathcal{H}^n.
$$
	
To complete the proof we observe that computing as above gives that, for $\rho\geq R_0$, 
\begin{align*}
\frac{d}{d\rho} \left(\rho^{-n} L_{\frac{3}{8}}(\rho)\right) &\leq \frac{3}{4} \rho^{-n+1} L_{\frac{3}{8}}(\rho)+2\rho^{-n} e^{\frac{3}{8} \rho^2} \int_{S_\rho} u \nabla_\Sigma u\cdot\frac{\nabla_\Sigma r}{|\nabla_\Sigma r|} \, d\mathcal{H}^{n-1}\\
&= \frac{3}{4} \rho^{-n+1} e^{\frac{3}{8} \rho^2} L(\rho)-\frac{3}{2} \rho^{-n}e^{\frac{\rho^2}{8} }  \int_{\bar{E}_{\rho}} \left(|\nabla_\Sigma u|^2 +u \mathscr{L}_{\Sigma}^0 u \right) e^{\frac{r^2}{4} }  \, d\mathcal{H}^n
\\ & -\frac{1}{2}\rho^{-n} \int_{\bar{E}_{\rho}} \left( |\nabla_\Sigma u|^2 +  u \mathscr{L}_\Sigma^0 u + \frac{1}{4} r u \partial_r u \right)  e^{\frac{3}{8} r^2} \, d\mathcal{H}^n\\
&\leq -\frac{1}{2} \rho^{-n} \int_{\bar{E}_{\rho}} |\nabla_\Sigma u|^2 e^{\frac{3}{8} r^2} \, d\mathcal{H}^n-\frac{1}{8}\rho^{-n}\int_{\bar{E}_\rho} r u \partial_r u \, e^{\frac{3}{8} r^2} \, d\mathcal{H}^n\\
&-\frac{3}{2} \rho^{-n}e^{\frac{ \rho^2}{8}}  \int_{\bar{E}_{\rho}} u (\mathscr{L}_{\Sigma}^0 u) e^{\frac{r^2}{4} } \, d\mathcal{H}^n-\frac{1}{2}\rho^{-n} \int_{\bar{E}_{\rho}}  u (\mathscr{L}_\Sigma^0 u) e^{\frac{3}{8} r^2} \, d\mathcal{H}^n
\end{align*}
where the last inequality uses the Poincar\'{e} inequality, Proposition \ref{GradControlsBoundaryProp} and the integration by parts is justified by Proposition \ref{DecayEigenfunProp}. Hence, as $\mu\leq \frac{1}{4}$ and $|A_{\Sigma}|^2\leq \frac{1}{4}$,
\begin{align*}
\frac{d}{d\rho} \left(\rho^{-n} L_{\frac{3}{8}}(\rho)\right) &\leq -\frac{1}{2} \rho^{-n} \int_{\bar{E}_{\rho}} |\nabla_\Sigma u|^2 e^{\frac{3}{8} r^2} \, d\mathcal{H}^n-\frac{1}{8} \rho^{-n}\int_{\bar{E}_\rho} r u \partial_r u \, e^{\frac{3}{8} r^2} \, d\mathcal{H}^n\\
&+\frac{3}{2} \rho^{-n}e^{\frac{\rho^2}{8} } \int_{\bar{E}_{\rho}}\left( -u L_{\Sigma} u +|A_\Sigma|^2 u^2-\frac{1}{2}u^2\right)e^{\frac{r^2}{4} } \, d\mathcal{H}^n\\
&+\frac{1}{2}\rho^{-n} \int_{\bar{E}_{\rho}}  \left(  -u L_{\Sigma} u +|A_\Sigma|^2 u^2-\frac{1}{2}u^2\right) e^{\frac{3}{8} r^2} \, d\mathcal{H}^n\\
&\leq -\frac{1}{2} \rho^{-n} \int_{\bar{E}_{\rho}} |\nabla_\Sigma u|^2 e^{\frac{3}{8} r^2} \, d\mathcal{H}^n-\frac{1}{8}\rho^{-n}\int_{\bar{E}_\rho} r u \partial_r u \, e^{\frac{3}{8} r^2} \, d\mathcal{H}^n.
\end{align*}

Rearranging terms and using the absorbing inequality give
$$
\frac{1}{4} \rho^{-n} \int_{\bar{E}_{\rho}} |\nabla_\Sigma u|^2 e^{\frac{3}{8} r^2} \, d\mathcal{H}^n \leq -\frac{d}{d\rho} \left(\rho^{-n} L_{\frac{3}{8}}(\rho)\right)+\frac{1}{64} \rho^{-n}\int_{\bar{E}_\rho} r^2 u^2 e^{\frac{3}{8} r^2}\, d\mathcal{H}^n.
$$
Pick $\tau>0$ so that $\rho^2 e^{-\frac{\rho^2}{16}}\leq 1$ for $\rho\geq \tau$.  Hence, for $\rho_2\geq R+\tau +\rho_1 +1\geq R_0\geq 2$, if $2\rho_2\geq \rho\geq \rho_2$, this together with \eqref{RefinedEst} gives,
\begin{align*}
\frac{1}{4} (2\rho_2)^{-n} & \int_{\bar{E}_{2\rho_2}} |\nabla_\Sigma u|^2 e^{\frac{3}{8} r^2} \, d\mathcal{H}^n \leq \frac{1}{4} \rho^{-n} \int_{\bar{E}_{\rho}}  |\nabla_\Sigma u|^2 e^{\frac{3}{8} r^2} \, d\mathcal{H}^n\\
&\leq-\frac{d}{d\rho} \left(\rho^{-n}L_{\frac{3}{8}}(\rho)\right)+\frac{1}{64} \rho^{-n} \int_{\bar{E}_\rho} u^2 e^{\frac{7}{16} r^2} \, d\mathcal{H}^n\\
&\leq -\frac{d}{d\rho} \left(\rho^{-n} L_{\frac{3}{8}}(\rho)\right)+C(n) \rho^{-1} \rho_0^{-n}  e^{-\frac{\rho^2}{16} +\frac{\rho^2_1}{4}} \int_{\Sigma} u^2 e^{\frac{r^2}{4}} \, d\mathcal{H}^n\\
&\leq -\frac{d}{d\rho} \left(\rho^{-n} L_{\frac{3}{8}}(\rho)\right)+C(n) \rho_2^{-1} \rho_0^{-n}  e^{-\frac{\rho^2_2}{16} +\frac{\rho^2_1}{4}} \int_{\Sigma} u^2 e^{\frac{r^2}{4}} \, d\mathcal{H}^n.
\end{align*}
	
Pick $R^{\prime}=R^{\prime}(n, K_\Sigma,\epsilon)\geq R+\tau+\rho_1+2$ so that if $\rho_2\geq R^{\prime}\geq 2$, then 
$$
C(n) \rho_2^{-1} \rho_0^{-n}  e^{-\frac{\rho_2^2}{16}+\frac{\rho^2_1}{4}}\leq  \frac{1}{64} (2\rho_2)^{-n} \epsilon.
$$
Hence, for such $\rho_2$ and $2\rho_2\geq \rho \geq \rho_2$, 
$$
\frac{1}{4} (2\rho_2)^{-n} \int_{\bar{E}_{2\rho_2}} |\nabla_\Sigma u|^2 e^{\frac{3}{8} r^2} \, d\mathcal{H}^n\leq -\frac{d}{d\rho} \left(\rho^{-n}L_{\frac{3}{8}}(\rho)\right)+ \frac{\epsilon}{64} (2\rho_2)^{-n}\int_{\Sigma} u^2 e^{\frac{r^2}{4}} \, d\mathcal{H}^n.
$$
	
Integrating the above inequality in $\rho$ over $[\rho_2+1, \rho_2+2]\subset [\rho_2, 2\rho_2]$ yields
$$
\frac{1}{4} (2\rho_2)^{-n} \int_{\bar{E}_{2\rho_2}} |\nabla_\Sigma u|^2 e^{\frac{3}{8} r^2} \, d\mathcal{H}^n\leq (\rho_2+1)^{-n} L_{\frac{3}{8}}(\rho_2+1) +\frac{\epsilon}{64} (2\rho_2)^{-n} \int_{\Sigma} u^2 e^{\frac{r^2}{4}} \, d\mathcal{H}^n.
$$
Applying \eqref{TraceEst} with $\rho=\rho_2+1$ and $\gamma=\frac{3}{8}$ gives, 
$$
\frac{1}{4} (2\rho_2)^{-n} \int_{\bar{E}_{2\rho_2}} |\nabla_\Sigma u|^2 e^{\frac{3}{8} r^2} \, d\mathcal{H}^n\leq 2 \rho_2^{-n} \int_{\bar{E}_{\rho_2}} u^2 e^{\frac{3}{8}r^2} \, d\mathcal{H}^n +\frac{\epsilon}{64}(2\rho_2)^{-n} \int_{\Sigma} u^2 e^{\frac{r^2}{4}} \, d\mathcal{H}^n.
$$
Applying \eqref{RefinedEst}, gives 
\begin{align*}
\rho_2^{-n}\int_{\bar{E}_{\rho_2}} u^2 e^{\frac{3}{8}r^2} d\mathcal{H}^n  & \leq 2 C(n) \rho_2^{-1} \rho_0^{-n}  e^{-\frac{\rho_2^2}{16} +\frac{\rho^2_1}{4}}\int_{\Sigma}u^2  e^{\frac{r^2}{4}} \, d\mathcal{H}^n\\
&\leq \frac{\epsilon}{32}(2\rho_2)^{-n} \int_{\Sigma}u^2  e^{\frac{r^2}{4}} \, d\mathcal{H}^n,
\end{align*}
Therefore, 
$$
\int_{\bar{E}_{2\rho_2}} |\nabla_\Sigma u|^2 e^{\frac{3}{8} r^2} \, d\mathcal{H}^n\leq \frac{\epsilon}{2} \int_{\Sigma}u^2  e^{\frac{r^2}{4}} \, d\mathcal{H}^n.
$$
The result is thus proved by setting $R_1=\max\set{2R,R^\prime}$.
\end{proof}

We next prove estimates on certain integrals that will be needed in studying perturbation results for the quadratic forms.

\begin{lem}\label{EigenEst} 
Let $\Sigma\in\mathcal{ACH}^{k,\alpha}_n$ be a self-expander. Fix a $\mu\leq \frac{1}{4}$. Given $\delta>0$, there is an $\epsilon_1=\epsilon_1(K_{\Sigma}, n, \delta)<\frac{1}{16}$ and a $C_1=C_1(K_{\Sigma}, n)$ so that for $\mathbf{f}\in \mathcal{B}_{\epsilon_1}(\mathbf{x}|_\Sigma; \mathcal{ACH}^{k,\alpha}_n(\Sigma))$ and $u\in W^1_{\frac{1}{4}}(\Sigma)$ that satisfies 
$$
-L_{\Sigma} u=\mu u,
$$
one has
$$
\int_{\Sigma} \left(|\nabla_\Sigma u|^2+u^2\right)\left( \Omega_{\mathbf{f}} +\Omega_{\mathbf{f}}^{-1}\right) e^{\frac{|\mathbf{x}|^2}{4}} \, d\mathcal{H}^n \leq C^2_1 \int_{\Sigma} u^2  e^{\frac{|\mathbf{x}|^2}{4}} \, d\mathcal{H}^n
$$
and
$$
\int_{\Sigma} \left(|\nabla_\Sigma u|^2+ u^2\right) (1-\Omega_{\mathbf{f}})^2(1+\Omega_{\mathbf{f}}^{-1})  e^{\frac{|\mathbf{x}|^2}{4}} \, d\mathcal{H}^n\leq \delta^2 \int_{\Sigma} u^2  e^{\frac{|\mathbf{x}|^2}{4}} \, d\mathcal{H}^n.
$$
\end{lem}

\begin{proof}
First observe that, by taking $\epsilon_1<\frac{1}{16}$ one has a constant $C=C(n,K_{\Sigma})>2+K_\Sigma$ so that if $\mathbf{f}\in \mathcal{B}_{\epsilon_1}(\mathbf{x}|_\Sigma; \mathcal{ACH}^{k,\alpha}_n(\Sigma))$, then
$$
C^{-1}e^{-\frac{|\mathbf{x}|^2}{32}}\leq \Omega_{\mathbf{f}} \leq C e^{\frac{|\mathbf{x}|^2}{32}}
$$
and
$$
C^{-1} g_{\Sigma}^{-1}\leq g_{\mathbf{f}}^{-1}\leq C g_{\Sigma}^{-1}.
$$
We observe that, by the hypotheses, 
$$
Q_{\Sigma}[u]=-B_{\Sigma}[u, L_\Sigma u]=\mu B_\Sigma[u]\leq \frac{1}{4} B_\Sigma[u] 
$$
and hence
$$
D_{\Sigma}[u]\leq  Q_{\Sigma}[u] +\left(K_\Sigma-\frac{1}{2}\right) B_{\Sigma}[u]\leq K_\Sigma B_\Sigma[u].
$$		
		
Let $\epsilon=\epsilon(K_\Sigma,n,\delta)>0$ be a small constant whose precise value will be chosen at different points in the argument. By Lemma \ref{ConcentrationLem}, there is an $R_1=R_1(K_\Sigma,n,\epsilon)>0$ so that
\begin{equation}\label{kappaEst}
\int_{\Sigma\backslash B_{R_1}}  \left(|\nabla_\Sigma u|^2 +u^2\right) e^{\frac{3}{8}|\mathbf{x}|^2} \, d\mathcal{H}^n\leq \epsilon^2\int_{\Sigma}  u^2 e^{\frac{|\mathbf{x}|^2}{4}} \, d\mathcal{H}^n=\epsilon^2 B_\Sigma[u].
\end{equation}
With $\epsilon=1$, one has
\begin{align*}
&\int_{\Sigma} \left(|\nabla_\Sigma u|^2+u^2\right) \left( \Omega_{\mathbf{f}} +\Omega_{\mathbf{f}}^{-1}\right) e^{\frac{|\mathbf{x}|^2}{4}} \, d\mathcal{H}^n \leq 2C\int_{\Sigma} \left(|\nabla_\Sigma u|^2+u^2\right) e^{\frac{5}{16}|\mathbf{x}|^2} \, d\mathcal{H}^n\\
&\leq 2Ce^{\frac{R_1^2}{16}}\int_{\Sigma\cap \bar{B}_{R}} \left(|\nabla_\Sigma u|^2+u^2\right) e^{\frac{|\mathbf{x}|^2}{4}} \, d\mathcal{H}^n+ 2C \int_{\Sigma\backslash B_{R_1}} \left(|\nabla_\Sigma u|^2+u^2\right) e^{\frac{5}{16}|\mathbf{x}|^2} \, d\mathcal{H}^n\\
&\leq  2Ce^{\frac{R_1^2}{16}}\left( D_\Sigma[u]+ B_\Sigma[u]\right) + 2C B_{\Sigma}[u]\\
&\leq 2C\left(e^{\frac{R_1^2}{16}}(K_\Sigma+1)+1\right)  B_\Sigma[u].
\end{align*}		
The first inequality follows by taking 
$$
C_1^2=2C\left(e^{\frac{R_1^2}{16}}(K_\Sigma+1)+1\right)
$$
and observing that $R_1$ and $C$ depend only on $n$ and $K_\Sigma$.

We next observe that, up to increasing $C$ a bit we have
$$
(1-\Omega_{\mathbf{f}})^2(1+\Omega_{\mathbf{f}}^{-1})e^{\frac{|\mathbf{x}|^2}{4}} \leq 2 C^2 e^{ \frac{11}{32}|\mathbf{x}|^2}\leq2 C^2 e^{ \frac{3}{8}|\mathbf{x}|^2}\ .
$$ 
In particular,  by \eqref{kappaEst} one has for a large enough $R_1=R_1(K_\Sigma,n,\epsilon^2)>0$,
$$
\int_{\Sigma\backslash B_{R_1}} \left(|\nabla_\Sigma u|^2+ u^2\right) (1-\Omega_{\mathbf{f}})^2 (1+\Omega_{\mathbf{f}}^{-1})  e^{\frac{|\mathbf{x}|^2}{4}} d\mathcal{H}^n\leq \epsilon^2 B_\Sigma[u]
$$
where we will again fix $\epsilon=\epsilon(K_\Sigma, n, \delta)$ later. Notice that by taking $\epsilon_1=\epsilon_1(K_{\Sigma},n,\epsilon)$ sufficiently small one has that on $B_{R_1}$,
$$
(1-\Omega_{\mathbf{f}})^2 (1+\Omega_{\mathbf{f}}^{-1})\leq \epsilon^2.
$$
Hence,
\begin{align*}
&\int_{\Sigma\cap \bar{B}_{R_1}} \left(|\nabla_\Sigma u|^2+ u^2\right) (1-\Omega_{\mathbf{f}})^2 (1+\Omega_{\mathbf{f}}^{-1}) e^{\frac{|\mathbf{x}|^2}{4}} \, d\mathcal{H}^n \\
& \leq \epsilon^2 \int_{\Sigma\cap \bar{B}_{R_1}} \left(|\nabla_\Sigma u|^2+ u^2\right)  e^{\frac{|\mathbf{x}|^2}{4}} \, d\mathcal{H}^n \\
& \leq \epsilon^2 \left( D_{\Sigma}[u]+B_\Sigma[u]\right)\leq \epsilon^2 (K_\Sigma+1) B_{\Sigma}[u].
\end{align*}
As such, the second inequality follows by taking 
$
\epsilon^2= \frac{\delta^2}{2(K_\Sigma+1)}< \frac{\delta^2}{2}.
$
\end{proof}

\begin{prop}\label{EstProp}
Let $\Sigma\in\mathcal{ACH}^{k,\alpha}_n$ be a self-expander. Fix a $\beta\in \left[\frac{1}{4}, \frac{3}{8}\right], \mu\leq \frac{1}{4}$ and $\delta\in\left(0, \frac{1}{2}\right)$. Let $u\in W^1_{\frac{1}{4}}(\Sigma)\setminus\set{0}$ satisfy 
$$
-L_{\Sigma} u=\mu u.
$$
There is an $\epsilon_1=\epsilon_1(K_\Sigma, n, \delta)>0$ so that if $\mathbf{f}\in \mathcal{B}_{\epsilon_1}(\mathbf{x}|_\Sigma;\mathcal{ACH}^{k,\alpha}_n(\Sigma))$, then, for any $v\in W^1_\beta(\Sigma)$,
$$
\left| B_{\mathbf{f}}[u,v]-B_{\Sigma}[u,v]\right|^2\leq \delta^2 B_{\Sigma} [u] B_{\Sigma}[v]
$$
and
$$
\left| B_{\mathbf{f}}[u,v]-B_{\Sigma}[u,v]\right|^2\leq \delta^2 B_{\Sigma} [u] B_{\mathbf{f}}[v]\leq 2\delta^2 B_{\mathbf{f}}[u]B_{\mathbf{f}}[v].
$$
In addition, one has
$$
\left| Q_{\mathbf{f}}[u,v]-Q_{\Sigma}[u,v]\right|^2\leq \delta^2 B_{\Sigma} [u] \left( D_{\Sigma}[v]+B_\Sigma[v]\right)
$$
and
$$
\left| Q_{\mathbf{f}}[u,v]-Q_{\Sigma}[u,v]\right|^2\leq \delta^2 B_{\Sigma} [u] \left( D_{\mathbf{f}}[v]+B_{\mathbf{f}}[v]\right).
$$
\end{prop}

\begin{proof}
Let $\delta^\prime=\delta^\prime(K_\Sigma, n, \delta)$ be a small positive number which we will specify in the course of the proof. Using this $\delta^\prime$, pick $\epsilon_1=\epsilon_1(K_\Sigma, n, \delta^\prime)<\frac{1}{16}$ as in Lemma \ref{EigenEst} for the given $u$. As long as $\delta^\prime\leq \delta$, one has
\begin{align*}
\left| B_{\mathbf{f}}[u,v]-B_\Sigma[u,v]\right|^2 &=\left| \int_{\Sigma} u v (\Omega_{\mathbf{f}}-1) e^{\frac{|\mathbf{x}|^2}{4}} \, d\mathcal{H}^n\right|^2\\
&\leq  \int_{\Sigma} u^2 (\Omega_{\mathbf{f}}-1)^2 e^{\frac{|\mathbf{x}|^2}{4}} \, d\mathcal{H}^n\int_{\Sigma}  v^2 e^{\frac{|\mathbf{x}|^2}{4}} \, d\mathcal{H}^n \\
&\leq (\delta')^2 B_\Sigma[u] B_\Sigma[v]\leq  \delta^2 B_\Sigma[u] B_\Sigma[v].  
\end{align*}
Here the last inequality follows from Lemma \ref{EigenEst}. Similarly, one has 
\begin{align*}
\left| B_{\mathbf{f}}[u,v]-B_\Sigma[u,v]\right|^2 &=\left| \int_{\Sigma} u v (\Omega_{\mathbf{f}}-1) e^{\frac{|\mathbf{x}|^2}{4}} \, d\mathcal{H}^n\right|^2\\
&\leq  \int_{\Sigma} u^2 \Omega_{\mathbf{f}}^{-1} (\Omega_{\mathbf{f}}-1)^2 e^{\frac{|\mathbf{x}|^2}{4}} \, d\mathcal{H}^n\int_{\Sigma}  v^2 \Omega_{\mathbf{f}} e^{\frac{|\mathbf{x}|^2}{4}}\, d\mathcal{H}^n \\
&\leq (\delta')^2 B_\Sigma[u] B_\mathbf{f}[v]\leq \delta^2 B_\Sigma[u] B_\mathbf{f}[v].
\end{align*}
Here again the first part of the second inequality follows from Lemma \ref{EigenEst}.  The final inequality follows by observing that by taking $u=v$ what we have already shown gives
\begin{align*}
\left| B_{\mathbf{f}}[u]-B_{\Sigma}[u]\right|^2 &\leq \delta^2 B_{\mathbf{f}}[u] B_{\Sigma}[u] \leq \delta^2 B_{\mathbf{f}}[u]^2 + \delta^2 B_{\mathbf{f}}[u] \left( B_{\Sigma}[u] -B_{\mathbf{f}}[u]\right)\\
&\leq \frac{3}{2} \delta^2B_{\mathbf{f}}[u]^2 +\frac{1}{2} \left| B_{\mathbf{f}}[u]-B_{\Sigma}[u]\right|^2
\end{align*}
where the last inequality follows from the absorbing inequality. Hence, 
$$
\left| B_{\mathbf{f}}[u]-B_{\Sigma}[u]\right|^2\leq 3 \delta^2 B_{\mathbf{f}}[u]^2\leq  B_{\mathbf{f}}[u]^2
$$
and so
$$
B_{\Sigma}[u]\leq B_{\mathbf{f}}[u] + B_{\mathbf{f}}[u]\leq 2  B_{\mathbf{f}}[u].
$$
This verifies the second part of the inequality.

To prove the second pair of inequalities one first observes that, by shrinking $\epsilon_1$ and $\delta^\prime$ as needed, one can ensure that
$$
\frac{1}{2} g_\Sigma^{-1} \leq \left(1- \delta' \right) g_\Sigma^{-1} \leq g_{\mathbf{f}}^{-1}\leq \left(1+ \delta^\prime\right) g_\Sigma^{-1}\leq  2 g_\Sigma^{-1} 
$$
and
$$
-\delta^\prime+|A_{\Sigma}|^2\leq |A_{\mathbf{f}(\Sigma)}\circ \mathbf{f}|^2 \leq \delta^\prime+|A_{\Sigma}|^2\leq K_{\Sigma}+1.
$$
As such, one computes,
\begin{align*}
& \left|D_{\mathbf{f}}[u,v]-D_\Sigma[u,v]\right| = \left|\int_{\Sigma} \left( \Omega_{\mathbf{f} }g_{\mathbf{f}}^{-1} -g_{\mathbf{f}}^{-1}+g_{\mathbf{f}}^{-1}-g_\Sigma^{-1}\right)^{ij}\nabla_i u \nabla_j v \, e^{\frac{|\mathbf{x}|^2}{4}} \, d\mathcal{H}^n \right|\\
&\leq \int_{\Sigma} \left|(g_{\mathbf{f}}^{-1})^{ij}  \nabla_i u \nabla_j v \right| \left| \Omega_{\mathbf{f}}-1\right| e^{\frac{|\mathbf{x}|^2}{4}} \, d\mathcal{H}^n +\int_{\Sigma}  \left| \left(g_{\mathbf{f}}^{-1} -g_{\Sigma}^{-1}\right)^{ij} \nabla_i u \nabla_j v \right|  e^{\frac{|\mathbf{x}|^2}{4}} \, d\mathcal{H}^n\\
&\leq 2\int_{\Sigma} |\nabla_{\Sigma} u| |\nabla_\Sigma v|  \left| \Omega_{\mathbf{f}}-1\right| e^{\frac{|\mathbf{x}|^2}{4}}  \, d\mathcal{H}^n+\delta^\prime\int_{\Sigma}  \left|  \nabla_\Sigma u \right| \left|\nabla_\Sigma v \right|  e^{\frac{|\mathbf{x}|^2}{4}} \, d\mathcal{H}^n\\
&\leq 2  D_{\Sigma}[v]^{\frac{1}{2}} \left(\int_{\Sigma}|\nabla_\Sigma u|^2 (\Omega_{\mathbf{f}}-1)^2  e^{\frac{|\mathbf{x}|^2}{4}} \, d\mathcal{H}^n\right)^{\frac{1}{2}}+\delta^\prime D_\Sigma[v]^{\frac{1}{2}} D_\Sigma[u]^{\frac{1}{2}}\\
&\leq \delta^\prime \left(2+K_\Sigma^{\frac{1}{2}}\right)  D_{\Sigma}[v]^{\frac{1}{2}}B_\Sigma[u]^{\frac{1}{2}}. 
\end{align*}
Similarly, one has
\begin{align*}
\left|D_\mathbf{f}[u,v]-D_\Sigma[u,v]\right| & \leq 2D_\mathbf{f}[v]^{\frac{1}{2}} \left(\int_\Sigma |\nabla_\Sigma u|^2 (\Omega_\mathbf{f}-1)^2\Omega_\mathbf{f}^{-1} e^{\frac{|\mathbf{x}|^2}{4}} \, d\mathcal{H}^n\right)^{\frac{1}{2}}\\
&+2\delta^\prime D_\mathbf{f}[v]^{\frac{1}{2}} \left(\int_\Sigma |\nabla_\Sigma u|^2 \Omega^{-1}_\mathbf{f} e^{\frac{|\mathbf{x}|^2}{4}} \, d\mathcal{H}^n\right)^{\frac{1}{2}} \\
& \leq 2\delta^\prime (C_1+1)D_\mathbf{f}[v]^{\frac{1}{2}} B_\Sigma[u]^{\frac{1}{2}}.
\end{align*}
Hence, by taking 
$$
\delta^\prime\leq \frac{\delta}{8}\min\set{\left(2+K_\Sigma^{1/2}\right)^{-1}, (C_1+1)^{-1}},
$$
one has
$$
|D_{\mathbf{f}}[u,v]-D_\Sigma[u,v]| \leq \frac{\delta}{4}  B_\Sigma[u]^{\frac{1}{2}} \min\set{D_{\Sigma}[v]^{\frac{1}{2}},D_\mathbf{f}[v]^{\frac{1}{2}}}.
$$
	
To complete the proof of the third inequality we observe that
\begin{align*}
\left| Q_{\mathbf{f}}[u,v]-Q_{\Sigma}[u,v]\right| &\leq \left|D_\mathbf{f}[u,v]-D_{\Sigma}[u,v]\right| +\frac{1}{2} \left| B_\mathbf{f}[u,v]-B_{\Sigma}[u,v]\right|\\
&+\left| \int_{\Sigma} uv \left(|A_{\mathbf{f}(\Sigma)}\circ \mathbf{f}|^2 \Omega_{\mathbf{f}} - |A_{\Sigma}|^2\right) e^{\frac{|\mathbf{x}|^2}{4}} d\mathcal{H}^n\right|\\
& \leq \left|D_\mathbf{f}[u,v]-D_{\Sigma}[u,v]\right| +\frac{1}{2} \left| B_\mathbf{f}[u,v]-B_{\Sigma}[u,v]\right|\\
&+(K_\Sigma+1) \int_{\Sigma}  \left|uv(\Omega_{\mathbf{f}}-1)\right| e^{\frac{|\mathbf{x}|^2}{4}} d\mathcal{H}^n+ \delta' \int_{\Sigma}|uv|  e^{\frac{|\mathbf{x}|^2}{4}} \, d\mathcal{H}^n.
\end{align*}
Combining the previous estimates with the Cauchy-Schwarz inequality and Lemma \ref{EigenEst} gives
$$
\left| Q_{\mathbf{f}}[u,v]-Q_{\Sigma}[u,v]\right| \leq \frac{\delta}{4}  B_\Sigma[u]^{\frac{1}{2}} D_{\Sigma}[v]^{\frac{1}{2}} +\left(\frac{\delta^\prime}{2}+(K_{\Sigma}+2) \delta^\prime\right) B_\Sigma[u]^{\frac{1}{2}} B_\Sigma[v]^{\frac{1}{2}}.
$$
Taking $\delta^\prime\leq \frac{\delta}{4(K_{\Sigma}+2)}  \leq \frac{\delta}{4}$ one has
\begin{align*}
\left| Q_{\mathbf{f}}[u,v]-Q_{\Sigma}[u,v]\right| &\leq  \frac{\delta}{2} B_\Sigma[u]^{\frac{1}{2}} \left(D_{\Sigma}[v]^{\frac{1}{2}} + B_\Sigma[v]^{\frac{1}{2}}\right)\\
&\leq \delta B_\Sigma[u]^{\frac{1}{2}} \left(D_{\Sigma}[v] + B_\Sigma[v]\right)^{\frac{1}{2}},
\end{align*}
which proves the third inequality.
  
The fourth inequality is proved in a similar fashion. Namely, one estimates as previous:
\begin{align*}
\left|Q_\mathbf{f}[u,v]-Q_\Sigma[u,v]\right| & \leq \left|D_\mathbf{f}[u,v]-D_{\Sigma}[u,v]\right| +\frac{1}{2} \left| B_\mathbf{f}[u,v]-B_{\Sigma}[u,v]\right| \\
& +(K_\Sigma+1) B_\mathbf{f}[v]^{\frac{1}{2}} \left(\int_\Sigma u^2 (\Omega_\mathbf{f}-1)^2\Omega_\mathbf{f}^{-1} e^{\frac{|\mathbf{x}|^2}{4}} \, d\mathcal{H}^n\right)^{\frac{1}{2}} \\
& + \delta^\prime B_\mathbf{f}[v]^{\frac{1}{2}} \left( \int_\Sigma u^2 \Omega_\mathbf{f}^{-1} e^{\frac{|\mathbf{x}|^2}{4}} \, d\mathcal{H}^n\right)^{\frac{1}{2}} \\
& \leq \frac{\delta}{4} B_\Sigma[u]^{\frac{1}{2}} D_\mathbf{f}[v]^{\frac{1}{2}} +\left(\frac{\delta^\prime}{2}+(K_\Sigma+1+C_1^{1/2})\delta^\prime\right) B_\Sigma[u]^{\frac{1}{2}} B_\mathbf{f}[v]^{\frac{1}{2}}.
\end{align*}
Hence, taking $\delta^\prime \leq \frac{\delta}{4\left(K_\Sigma+1+C_1^{1/2}\right)}\leq\frac{\delta}{4}$ one has the fourth inequality by the Cauchy-Schwarz inequality.
\end{proof}

\begin{cor} \label{EigenSpacePerturbCor}
Let $\Sigma\in\mathcal{ACH}^{k,\alpha}_n$ be a self-expander, and let $\mathbf{v}\in C^{k,\alpha}_0\cap C^{k}_{0,\mathrm{H}}(\Sigma;\mathbb{R}^{n+1})$ be a transverse section on $\Sigma$ so that $\mathcal{C}[\mathbf{v}]$ is transverse to $\mathcal{C}(\Sigma)$. Let 
$$
C_{\Sigma,\mathbf{v}}=\sup_{\Sigma} |\mathbf{v}\cdot\mathbf{n}_\Sigma|^{-1}+\Vert\mathbf{v}\Vert_{1}.
$$
Fix a $\beta\in \left[\frac{1}{4}, \frac{3}{8}\right], \mu\leq \frac{1}{4}$ and $\delta\in \left(0,\frac{1}{2}\right)$. Suppose $u\in W^1_{\frac{1}{4}}(\Sigma)\setminus\set{0}$ satisfies 
$$
-L_{\Sigma, \mathbf{v}}^\prime u=\mu u.
$$
There is an $\epsilon_2=\epsilon_2(K_\Sigma, C_{\Sigma,\mathbf{v}}, n, \delta)>0$ so that if $\mathbf{f}\in \mathcal{B}_{\epsilon_2}(\mathbf{x}|_\Sigma; \mathcal{ACH}^{k,\alpha}_n(\Sigma))$, then, for any $v\in W^1_{\beta}(\Sigma)$,
$$
\left|B_{\mathbf{f}, \mathbf{v}}[u,v]-B_{\Sigma, \mathbf{v}}[u,v]\right|^2\leq \delta^2 B_{\Sigma, \mathbf{v}} [u] B_{\Sigma, \mathbf{v}}[v]
$$
and
$$
\left|B_{\mathbf{f}, \mathbf{v}}[u,v]-B_{\Sigma, \mathbf{v}}[u,v]\right|^2\leq \delta^2 B_{\Sigma, \mathbf{v}} [u] B_{\mathbf{f}, \mathbf{v}}[v]\leq  2\delta^2 B_{\mathbf{f},\mathbf{v}}[u]B_{\mathbf{f},\mathbf{v}}[v].
$$
Likewise, 
$$
\left| Q_{\mathbf{f}, \mathbf{v}}[u,v]-Q_{\Sigma, \mathbf{v}}[u,v]\right|^2\leq  \delta^2 B_{\Sigma, \mathbf{v}} [u] \left(D_{\Sigma,\mathbf{v}}[v]+B_{\Sigma, \mathbf{v}}[v]\right)
$$
and
$$
\left| Q_{\mathbf{f}, \mathbf{v}}[u,v]-Q_{\Sigma, \mathbf{v}}[u,v]\right|^2\leq  \delta^2 B_{\Sigma, \mathbf{v}} [u] \left(D_{\mathbf{f},\mathbf{v}}[v]+B_{\mathbf{f}, \mathbf{v}}[v]\right).
$$
\end{cor}

\begin{proof}
We first observe that if $-L_{\Sigma,\mathbf{v}}^\prime u=\mu u$, then $U=(\mathbf{v}\cdot\mathbf{n}_\Sigma) u$ satisfies
$$
-L_\Sigma U=\mu U. 
$$
If $V=(\mathbf{v}\cdot\mathbf{n}_\Sigma) v$, then 
$$
B_{\Sigma,\mathbf{v}}[u,v]=B_\Sigma[U,V], D_{\Sigma,\mathbf{v}}[u,v]=D_\Sigma[U,V] \mbox{ and } Q_{\Sigma,\mathbf{v}}[u,v]=Q_\Sigma[U,V].
$$
Let $\Lambda=\mathbf{f}(\Sigma)$. There is a $C=C(K_\Sigma,C_{\Sigma,\mathbf{v}},n)$ so that 
$$
\Vert\mathbf{v}\cdot(\mathbf{n}_\Lambda\circ\mathbf{f})-\mathbf{v}\cdot\mathbf{n}_\Sigma\Vert_1\leq C\epsilon_2.
$$
Thus, by choosing $\epsilon_2$ sufficiently small, we have that
$$
\frac{1}{4}B_\mathbf{f}[V]\leq B_{\mathbf{f},\mathbf{v}}[v] \leq 4 B_\mathbf{f}[V]
$$
and
$$
\frac{1}{4}\left(B_{\mathbf{f}}[V]+D_\mathbf{f}[V]\right) \leq B_{\mathbf{f},\mathbf{v}}[v]+D_{\mathbf{f},\mathbf{v}}[v] \leq 4\left(B_{\mathbf{f}}[V]+D_\mathbf{f}[V]\right).
$$
And there is a $C^\prime$ depending on $K_\Sigma,C_{\Sigma,\mathbf{v}}$ and $n$ so that 
\begin{equation} \label{EigenPerturbEqn}
\left|B_{\mathbf{f},\mathbf{v}}[u,v]-B_\mathbf{f}[U,V]\right| \leq C^\prime\epsilon_2 B_\mathbf{f}[|U|,|V|]
\end{equation}
and 
\begin{align*}
\left|Q_{\mathbf{f},\mathbf{v}}[u,v]-Q_\mathbf{f}[U,V]\right| \leq C^\prime\epsilon_2 & \left(B_\mathbf{f}[|U|,|V|]+B_\mathbf{f}[|\nabla_\Sigma U|,|V|]\right. \\
& \left.+B_\mathbf{f}[|U|,|\nabla_\Sigma V|]+B_\mathbf{f}[|\nabla_\Sigma U|,|\nabla_\Sigma V|]\right).
\end{align*}

By the Cauchy-Schwarz inequality and Lemma \ref{EigenEst}, for $\epsilon_2$ sufficiently small
$$
B_\mathbf{f}[|U|,|V|] \leq B_\Sigma[U(\Omega_\mathbf{f}-1)]^{\frac{1}{2}} B_\Sigma [V]^{\frac{1}{2}}+B_\Sigma[U]^{\frac{1}{2}} B_\Sigma[V]^{\frac{1}{2}} \leq 2B_\Sigma[U]^{\frac{1}{2}} B_\Sigma[V]^{\frac{1}{2}}.
$$
Thus, by choosing $\epsilon_2$ sufficiently small
$$
\left|B_{\mathbf{f},\mathbf{v}}[u,v]-B_\mathbf{f}[U,V]\right| \leq \frac{\delta}{2} B_\Sigma[U]^{\frac{1}{2}} B_\Sigma[V]^{\frac{1}{2}}.
$$
Hence, by Proposition \ref{EstProp} and possibly shrinking $\epsilon_2$,
\begin{align*}
\left| B_{\mathbf{f},\mathbf{v}}[u,v]-B_{\Sigma,\mathbf{v}}[u,v]\right| & \leq \left|B_{\mathbf{f},\mathbf{v}}[u,v]-B_\mathbf{f}[U,V]\right|+\left| B_\mathbf{f}[U,V]-B_\Sigma[U,V]\right| \\
& \leq \delta B_\Sigma[U]^{\frac{1}{2}} B_\Sigma [V]^{\frac{1}{2}}=\delta B_{\Sigma,\mathbf{v}}[u]^{\frac{1}{2}} B_{\Sigma,\mathbf{v}}[v]^{\frac{1}{2}},
\end{align*}
proving the first inequality.

To prove the second inequalities, one again uses the Cauchy-Schwarz inequality and Lemma \ref{EigenEst} to compute,
\begin{align*}
B_\mathbf{f}[|U|,|V|] & \leq B_\Sigma\left[U(\Omega_\mathbf{f}-1)\Omega_\mathbf{f}^{-\frac{1}{2}}\right]^{\frac{1}{2}} B_\mathbf{f}[V]^{\frac{1}{2}}+B_\Sigma \left[U\Omega_\mathbf{f}^{-\frac{1}{2}}\right]^{\frac{1}{2}} B_\mathbf{f}[V]^{\frac{1}{2}} \\
& \leq (1+C_1)B_\Sigma[U]^{\frac{1}{2}} B_\mathbf{f}[V]^{\frac{1}{2}}.
\end{align*}
Thus, by choosing $\epsilon_2$ sufficiently small \eqref{EigenPerturbEqn} gives
$$
\left|B_{\mathbf{f},\mathbf{v}}[u,v]-B_\mathbf{f}[U,V]\right| \leq \frac{\delta}{4} B_\Sigma[U]^{\frac{1}{2}} B_\mathbf{f}[V]^{\frac{1}{2}}.
$$
Hence, invoking Proposition \ref{EstProp} and shrinking $\epsilon_2$ (if needed) one has
$$
\left| B_{\mathbf{f},\mathbf{v}}[u,v]-B_{\Sigma,\mathbf{v}}[u,v]\right| \leq \frac{\delta}{2} B_\Sigma[U]^{\frac{1}{2}} B_\mathbf{f}[V]^{\frac{1}{2}} \leq \delta B_{\Sigma,\mathbf{v}}[u]^{\frac{1}{2}} B_{\mathbf{f},\mathbf{v}}[v]^{\frac{1}{2}}.
$$
Now taking $u=v$ it follows that
\begin{align*}
\left|B_{\mathbf{f},\mathbf{v}}[u]-B_{\Sigma,\mathbf{v}}[u]\right|^2 & \leq \delta^2 B_{\Sigma,\mathbf{v}}[u]B_{\mathbf{f},\mathbf{v}}[u]\leq \delta^2 B_{\mathbf{f},\mathbf{v}}[u]^2+\delta^2 B_{\mathbf{f},\mathbf{v}}[u]\left|B_{\mathbf{f},\mathbf{v}}[u]-B_{\Sigma,\mathbf{v}}[u]\right| \\
& \leq \frac{3}{2} \delta^2 B_{\mathbf{f},\mathbf{v}}[u]^2+\frac{1}{2} \left|B_{\mathbf{f},\mathbf{v}}[u]-B_{\Sigma,\mathbf{v}}[u]\right|^2.
\end{align*}
Thus, as $\delta<1/2$ we have
$$
\left|B_{\mathbf{f},\mathbf{v}}[u]-B_{\Sigma,\mathbf{v}}[u]\right|^2 \leq 3\delta^2 B_{\mathbf{f},\mathbf{v}}[u]^2 \leq B_{\mathbf{f},\mathbf{v}}[u]^2
$$
and so
$$
B_{\Sigma,\mathbf{v}}[u] \leq B_{\mathbf{f},\mathbf{v}}[u]+\left|B_{\mathbf{f},\mathbf{v}}[u]-B_{\Sigma,\mathbf{v}}[u]\right| \leq 2B_{\mathbf{f},\mathbf{v}}[u].
$$
This completes the proof of the second inequalities.

The other pair of inequalities are proved in a similar fashion.
\end{proof}

\begin{cor}\label{PerturbCor}
Let $\Sigma\in\mathcal{ACH}^{k,\alpha}_n$ be a self-expander, and let $\mathbf{v}\in C^{k,\alpha}_0\cap C^{k}_{0,\mathrm{H}}(\Sigma;\mathbb{R}^{n+1})$ be a transverse section on $\Sigma$ so that $\mathcal{C}[\mathbf{v}]$ is transverse to $\mathcal{C}(\Sigma)$. Let
$$
C_{\Sigma,\mathbf{v}}=\sup_{\Sigma}|\mathbf{v}\cdot\mathbf{n}_\Sigma|^{-1}+\Vert\mathbf{v}\Vert_1.
$$
Fix a $\beta\in \left[\frac{1}{4}, \frac{3}{8}\right]$, $\mu\leq \frac{1}{4}$ and $\delta\in \left(0,\frac{1}{2}\right)$. Let $V$ be a subspace of $W^1_{\frac{1}{4}}(\Sigma)$ spanned by $N$ linearly independent eigenfunctions of $-L_{\Sigma,\mathbf{v}}^\prime$ with eigenvalues less than or equal to $\mu$. There is an $\epsilon_3=\epsilon_3(K_\Sigma, C_{\Sigma,\mathbf{v}},n, \delta, N)>0$ so that if $\mathbf{f}\in \mathcal{B}_{\epsilon_3}(\mathbf{x}|_\Sigma;\mathcal{ACH}^{k,\alpha}_n(\Sigma))$, then, for any $u\in V\setminus\set{0}$ and $v\in W^1_\beta(\Sigma)$,
$$
\left|B_{\mathbf{f}, \mathbf{v}}[u,v]-B_{\Sigma, \mathbf{v}}[u,v]\right|^2\leq \delta^2 B_{\Sigma, \mathbf{v}} [u] B_{\Sigma, \mathbf{v}}[v]
$$
and
$$
\left|B_{\mathbf{f}, \mathbf{v}}[u,v]-B_{\Sigma, \mathbf{v}}[u,v]\right|^2\leq \delta^2 B_{\Sigma, \mathbf{v}}[u]B_{\mathbf{f}, \mathbf{v}} [v]\leq 2\delta^2 B_{\mathbf{f}, \mathbf{v}} [u]B_{\mathbf{f}, \mathbf{v}} [v].
$$
Likewise,
$$
\left| Q_{\mathbf{f}, \mathbf{v}}[u,v]-Q_{\Sigma, \mathbf{v}}[u,v]\right|^2\leq  \delta^2 B_{\Sigma, \mathbf{v}} [u] \left(D_{\Sigma,\mathbf{v}}[v]+B_{\Sigma, \mathbf{v}}[v]\right)
$$
and
$$
\left| Q_{\mathbf{f}, \mathbf{v}}[u,v]-Q_{\Sigma, \mathbf{v}}[u,v]\right|^2\leq  \delta^2 B_{\Sigma, \mathbf{v}} [u] \left(D_{\mathbf{f},\mathbf{v}}[v]+B_{\mathbf{f}, \mathbf{v}}[v]\right).
$$
\end{cor}

\begin{proof}
Pick a $B_{\Sigma, \mathbf{v}}$-orthonormal basis, $u_1, \ldots, u_N$, of $V$ of eigenfunctions of $-L_{\Sigma, \mathbf{v}}^\prime$. This is possible as $-L_{\Sigma, \mathbf{v}}^\prime$ is self-adjoint with respect to $B_{\Sigma, \mathbf{v}}$. In particular, $Q_{\Sigma, \mathbf{v}}[u_i,u_j]=\mu_i \delta_{ij}$ where $\mu_i$ is the eigenvalue of $u_i$. That is, this orthonormal basis also diagonalizes $Q_{\Sigma, \mathbf{v}}$. Pick $\epsilon_3$ so for any $u_i$ one has 
$$
\left|B_{\mathbf{f}, \mathbf{v}}[u_i,v]-B_{\Sigma, \mathbf{v}}[u_i,v]\right|^2\leq \frac{\delta^2}{N} B_{\Sigma, \mathbf{v}}[u_i] B_{\Sigma, \mathbf{v}}[v].
$$
	
Write $u=\sum_{i=1}^N a_i u_i$. One has, by the Cauchy-Schwarz inequality,
\begin{align*}
\left|B_{\mathbf{f}, \mathbf{v}}[u,v]-B_{\Sigma, \mathbf{v}}[u,v]\right|^2 &= \left| \sum_{i=1}^N  a_i \left(B_{\mathbf{f}, \mathbf{v}}[u_i,v]- B_{\Sigma, \mathbf{v}}[u_i,v]\right)\right|^2\\
&\leq N \sum_{i=1}^N a_i^2 \left|B_{\mathbf{f}, \mathbf{v}}[u_i,v]-B_{\Sigma, \mathbf{v}}[u_i,v]\right|^2\\
&\leq N \sum_{i=1}^N a_i^2 \frac{\delta^2}{N} B_{\Sigma, \mathbf{v}}[u_i] B_{\Sigma, \mathbf{v}}[v]=\delta^2 B_{\Sigma, \mathbf{v}}[u] B_{\Sigma, \mathbf{v}}[v]
\end{align*}
where the last equality used that $\set{u_i}$ is an orthonormal basis.
	
The proof of remaining inequalities proceed in a similar fashion.
\end{proof}

\section{Stability of bilinear form} \label{StableFormSec}
In this section we show the stability of the positive and negative spaces of the bilinear form $Q_{\mathbf{f},\mathbf{v}}$ as one perturbs $\mathbf{f}$ around $\mathbf{x}|_{\Sigma}$ in $\mathcal{ACH}^{k,\alpha}_n(\Sigma)$ for $\Sigma$ an asymptotically conical self-expander. This is a non-trivial result as such a perturbation can radically change the weight that appears.  In what follows we will use the (easily checked) fact that if $\mathbf{f}\in \mathcal{B}_{\frac{1}{8}}(\mathbf{x}|_\Sigma; \mathcal{ACH}^{k,\alpha}_n(\Sigma))$, then for all $u\in W^1_{\frac{3}{8}}(\Sigma)$ one has
$$
Q_{\mathbf{f}, \mathbf{v}}[u]+D_{\mathbf{f}, \mathbf{v}}[u]+ B_{\mathbf{f}, \mathbf{v}}[u]<\infty.
$$
This explains the choice of weighted space that will appear below.

We define 
$$
\mathcal{E}_{\Sigma, \mathbf{v}}^\prime[\mu]=\mathrm{span} \set{u\in W^1_{\frac{1}{4}}(\Sigma)\colon -L_{\Sigma, \mathbf{v}}^\prime u = \mu^\prime u \mbox{ for some $\mu^\prime\leq \mu$}}
$$
to be the subspace of $W^1_{\frac{1}{4}}(\Sigma)$ spanned by eigenfunctions of $-L_{\Sigma, \mathbf{v}}^\prime$ with eigenvalues less than or equal to $\mu$. We also let
$$
\mathcal{E}_{\Sigma, \mathbf{v}}^\prime(\mu)=\mathrm{span} \set{u\in W^1_{\frac{1}{4}}(\Sigma)\colon -L_{\Sigma, \mathbf{v}}^\prime u = \mu^\prime u \mbox{ for some $\mu^\prime< \mu$}}
$$
be the space spanned by eigenfunctions with eigenvalues strictly less than $\mu$. Let $\mathcal{N}_{\Sigma, \mathbf{v}}^\prime=\mathcal{E}_{\Sigma, \mathbf{v}}^\prime(0)$ be the space spanned by the eigenfunctions with negative eigenvalues and let 
$$
\mathcal{K}_{\Sigma, \mathbf{v}}^\prime=\mathcal{E}_{\Sigma, \mathbf{v}}^\prime [0]\setminus \mathcal{E}_{\Sigma, \mathbf{v}}^\prime(0)=\mathcal{K}_{\mathbf{v}}
$$
be the space of elements in $W^1_{\frac{1}{4}}(\Sigma)$ in the kernel of $L_{\Sigma, \mathbf{v}}^\prime$. Observe that, by Proposition \ref{DecayEigenfunProp}, for any $\mu$,
$$
\mathcal{E}_{\Sigma, \mathbf{v}}^\prime(\mu)\subset \mathcal{E}_{\Sigma, \mathbf{v}}^\prime[\mu]\subset W^1_{\frac{3}{8}}(\Sigma).
$$
and so the same is true of $\mathcal{N}_{\Sigma, \mathbf{v}}^\prime$ and $\mathcal{K}_{\Sigma, \mathbf{v}}^\prime$. Using this fact, let
$$
\mathcal{P}_{\Sigma, \mathbf{v}}^\prime=\set{f\in W^1_{\frac{3}{8}}(\Sigma)\colon B_{\Sigma, \mathbf{v}}[f, u]=0, \forall u\in \mathcal{E}_{\Sigma, \mathbf{v}}^\prime[0]}. 
$$
By construction and the fact that $-L_{\Sigma, \mathbf{v}}^\prime$ is self-adjoint with respect to $B_{\Sigma, \mathbf{v}}$ and has a discrete spectrum,
$$
W^1_{\frac{3}{8}}(\Sigma)=\mathcal{N}_{\Sigma, \mathbf{v}}^\prime\oplus\mathcal{K}_{\Sigma, \mathbf{v}}^\prime\oplus\mathcal{P}_{\Sigma, \mathbf{v}}^\prime
$$
and these three subspaces are orthogonal with respect to $B_{\Sigma, \mathbf{v}}$. Moreover, for all $u\in \mathcal{N}_{\Sigma, \mathbf{v}}^\prime$,
$$
Q_{\Sigma, \mathbf{v}}[u]\leq \mu_{\Sigma, \mathbf{v}}^- B_{\Sigma, \mathbf{v}}[u]\leq 0,
$$
and for all $u\in\mathcal{P}_{\Sigma, \mathbf{v}}^\prime$,
$$
Q_{\Sigma, \mathbf{v}}[u]\geq \mu_{\Sigma, \mathbf{v}}^+ B_{\Sigma, \mathbf{v}}[u]\geq 0.
$$
Here $ \mu_{\Sigma, \mathbf{v}}^- <0< \mu_{\Sigma, \mathbf{v}}^+$ are, respectively, the largest negative and smallest positive eigenvalues of $-L_{\Sigma, \mathbf{v}}^\prime$.  Likewise, for $u\in \mathcal{K}_{\Sigma,\mathbf{v}}^\prime$ and $v\in W^1_{\frac{3}{8}}(\Sigma)$,  $Q_{\Sigma, \mathbf{v}}[u, v]=0$.  

First, a simple consequence of the inverse function theorem in H\"older spaces.

\begin{lem} \label{InverseLem}  
Fix $\Sigma\in\mathcal{ACH}^{k,\alpha}_n$ and $\epsilon>0$. There is an $\eta=\eta(\Sigma,\epsilon)>0$ so that if $\mathbf{f}\in \mathcal{B}_{\eta}(\mathbf{x}|_\Sigma;\mathcal{ACH}^{k,\alpha}_n(\Sigma))$, then 
$$
\mathbf{f}^{-1} \in \mathcal{B}_{\epsilon}(\mathbf{x}|_{\mathbf{f}(\Sigma)};\mathcal{ACH}^{k,\alpha}_n(\mathbf{f}(\Sigma))).
$$
\end{lem}

\begin{proof}
Let $\Lambda=\mathbf{f}(\Sigma)$. For $\eta$ sufficiently small, $\mathbf{f}\in \mathcal{ACH}^{k,\alpha}_n(\Sigma)$ and so $\Lambda\in\mathcal{ACH}^{k,\alpha}_n$. We first observe that as long as $\eta\leq \frac{1}{8}$, then by the implicit function theorem in $C^{k,\alpha}_1$
$$
\Vert \mathbf{f}^{-1}\Vert_{k,\alpha}^{(1)}\leq C(n)
$$
Observe that
$$
\mathbf{f}^{-1}-\mathbf{x}|_{\Gamma}=\mathbf{x}|_{\Sigma} \circ  \mathbf{f}^{-1}-\mathbf{f}\circ \mathbf{f}^{-1},
$$
and so by \cite[Proposition 3.1 (4)]{BWBanachManifold}
$$
\Vert \mathbf{f}^{-1}-\mathbf{x}|_{\Lambda}\Vert_{k,\alpha}^{(1)} \leq  C^\prime\Vert \mathbf{f}-\mathbf{x}|_{\Sigma}\Vert_{k,\alpha}^{(1)}
$$
for some $C^\prime=C^\prime(\Sigma,C(n))$. The claim follows immediately.
\end{proof}

We now show the stability of the positive and negative spaces for the quadratic form $Q_{\mathbf{f}, \mathbf{v}}$ for $\mathbf{f}$ a small perturbation of $\mathbf{x}|_{\Sigma}$. The main challenge in proving this result is the fact that the weight is allowed to change a large amount near infinity and we cannot directly use the estimates of Section \ref{ConcentrationSec} on the positive part.

\begin{prop}\label{QuadFormProp}
Let $\Sigma\in\mathcal{ACH}^{k,\alpha}_n$ be a self-expander, and let $\mathbf{v}\in C^{k,\alpha}_0\cap C^{k}_{0,\mathrm{H}}(\Sigma;\mathbb{R}^{n+1})$ a transverse section on $\Sigma$ so that $\mathcal{C}[\mathbf{v}]$ is transverse to $\mathcal{C}(\Sigma)$. There is an $\epsilon^\prime=\epsilon^\prime(\Sigma,\mathbf{v})$ so that if $\mathbf{f}\in \mathcal{B}_{\epsilon^\prime}(\mathbf{x}|_\Sigma;\mathcal{ACH}^{k,\alpha}_n(\Sigma))$ is $E$-stationary (i.e., $\mathbf{f}(\Sigma)$ is a self-expander), then
$$
Q_{\mathbf{f}, \mathbf{v}}[u]< 0,
$$
for any $u\in \mathcal{N}_{\Sigma, \mathbf{v}}^\prime\setminus\set{0}$ and
$$
Q_{\mathbf{f}, \mathbf{v}}[u] >0,
$$
for any $u\in \mathcal{P}_{\Sigma,\mathbf{v}}^\prime\setminus\set{0}$. 
\end{prop}

\begin{proof}
First observe that if $u\in \mathcal{N}_{\Sigma,\mathbf{v}}^\prime\setminus\set{0}$, then we have
$$
Q_{\Sigma, \mathbf{v}}[u]<\frac{1}{2} \mu_{\Sigma, \mathbf{v}}^-B_{\Sigma, \mathbf{v}}[u]<0.
$$
Moreover, as 
$$
Q_{\Sigma,\mathbf{v}}[u]\geq D_{\Sigma,\mathbf{v}}[u]-K_\Sigma B_{\Sigma,\mathbf{v}}[u]
$$
one has
$$
D_{\Sigma,\mathbf{v}}[u]<K_\Sigma B_{\Sigma,\mathbf{v}}[u].
$$
Thus, by Corollary \ref{PerturbCor} with $\delta= -\frac{\mu_{\Sigma, \mathbf{v}}^-}{4\sqrt{K_\Sigma+1}}>0$ one has an $\epsilon_3$ given in the corollary so that if $\epsilon<\epsilon_3$, then
$$
\left| Q_{\mathbf{f}, \mathbf{v}}[u]-Q_{\Sigma, \mathbf{v}}[u]\right| \leq \delta \sqrt{K_\Sigma+1}\, B_{\Sigma, \mathbf{v}}[u].
$$
Hence, 
\begin{align*}
Q_{\mathbf{f},\mathbf{v}}[u] & \leq\delta \sqrt{K_\Sigma+1}\,B_{\Sigma, \mathbf{v}}[u] +Q_{\Sigma, \mathbf{v}}[u] \\
& <\left(\frac{1}{2} \mu_{\Sigma, \mathbf{v}}^-+\delta \sqrt{K_\Sigma+1}\right)B_{\Sigma, \mathbf{v}}[u]= \frac{1}{4} \mu_{\Sigma, \mathbf{v}}^-B_{\Sigma, \mathbf{v}}[u]<0.
\end{align*}
This proves the first part of the proposition.
	
For the second part, first observe that, by definition, for all $u\in \mathcal{E}_{\Sigma, \mathbf{v}}^\prime[0]=\mathcal{N}_{\Sigma,\mathbf{v}}^\prime\oplus \mathcal{K}_{\Sigma,\mathbf{v}}^\prime$ one has $Q_{\Sigma,\mathbf{v}}[u]\leq 0$. Hence, up to shrinking $\epsilon_3$,  Corollary \ref{PerturbCor} gives, as above, that
$$
Q_{\mathbf{f},\mathbf{v}}[u]\leq \frac{1}{4} \mu^+ B_{\mathbf{f},\mathbf{v}}[u],
$$
where
$$
\mu^+=\frac{1}{4}\min \set{\mu_{\Sigma, \mathbf{v}}^+,\frac{1}{4}}>0.
$$
We now argue by contradiction. That is, suppose there were a $v_0\in \mathcal{P}^\prime_{\Sigma, \mathbf{v}}\setminus\set{0}$ with $Q_{\mathbf{f},\mathbf{v}}[v_0]\leq 0$. Observe that, by definition, $v_0$ is orthogonal (with respect to $B_{\Sigma, \mathbf{v}}$) to $\mathcal{E}_{\Sigma, \mathbf{v}}^\prime[0]$. Up to further shrinking $\epsilon_3$, we may assume that for all $\mathbf{f}\in \mathcal{B}_{\epsilon_3}(\mathbf{x}|_\Sigma;\mathcal{ACH}^{k,\alpha}_n(\Sigma))$ if $\Lambda=\mathbf{f}(\Sigma)$, then
$$
K_{\Lambda}\leq K_\Sigma+1.
$$
Moreover, as 
$$
0\geq Q_{\mathbf{f}, \mathbf{v}}[v_0]\geq D_{\mathbf{f},\mathbf{v}}[v_0]-K_\Lambda B_{\mathbf{f}, \mathbf{v}}[v_0]
$$
one has
$$
D_{\mathbf{f},\mathbf{v}}[v_0]\leq \left(K_\Sigma+1\right) B_{\mathbf{f}, \mathbf{v}}[v_0].
$$
Now let
$$
V=\mathrm{span}\set{v_0}\oplus\mathcal{E}_{\Sigma, \mathbf{v}}^\prime[0]=\mathrm{span}\set{v_0}\oplus \mathcal{N}_{\Sigma,\mathbf{v}}^\prime\oplus\mathcal{K}_{\Sigma,\mathbf{v}}^\prime.
$$
The choice of $v_0$ ensures that  
$$
\dim V= 1+ \dim \mathcal{E}_{\Sigma, \mathbf{v}}^\prime [0]=1+ \dim \mathcal{N}_{\Sigma,\mathbf{v}}^\prime+\dim \mathcal{K}_{\Sigma,\mathbf{v}}^\prime=1+N.
$$
	
We claim that, up to shrinking $\epsilon_3$, for any $v\in V$ one has
$$
Q_{\mathbf{f}, \mathbf{v}}[v]\leq \mu^+ B_{\mathbf{f},\mathbf{v}}[v].
$$
Indeed, write
$$
v=c_0 v_0+ u
$$
where $c_0\in \Real$ and $u\in \mathcal{E}_{\Sigma, \mathbf{v}}^\prime[0]$. Using Corollary \ref{PerturbCor} and the fact that $Q_{\Sigma, \mathbf{v}}[v_0,u]=0$, one computes,
\begin{align*}
|Q_{\mathbf{f}, \mathbf{v}}[v_0, u]|^2 & \leq \delta^2 \left(D_{\mathbf{f}, \mathbf{v}}[v_0]+B_{\mathbf{f}, \mathbf{v}}[v_0]\right) B_{\Sigma, \mathbf{v}}[u] \\
& \leq 2\delta^2  \left( K_{\Sigma}+2\right) B_{\mathbf{f}, \mathbf{v}}[v_0]  B_{\mathbf{f}, \mathbf{v}}[u].
\end{align*}
Hence, picking $\delta$ (and hence $\epsilon_3$) small enough so $\delta\leq \frac{\mu^+}{16\sqrt{K_{\Sigma}+2}}$ yields
\begin{align*}
Q_{\mathbf{f}, \mathbf{v}}[v] &=c_0^2 Q_{\mathbf{f}, \mathbf{v}}[v_0]+2c_0 Q_{\mathbf{f}, \mathbf{v}}[v_0, u]+Q_{\mathbf{f}, \mathbf{v}}[u] \\
& \leq \frac{1}{4} |c_0| \mu^+ B_{\mathbf{f}, \mathbf{v}}[v_0]^{1/2}  B_{\mathbf{f}, \mathbf{v}}[u]^{1/2}+\frac{1}{4} \mu^+ B_{\mathbf{f},\mathbf{v}}[u]\\
&\leq  \frac{1}{2}\mu^+ \left( B_{\mathbf{f}, \mathbf{v}}[c_0 v_0]+ B_{\mathbf{f},\mathbf{v}}[u]\right)
\\& =  \frac{1}{2}\mu^+ \left( B_{\mathbf{f}, \mathbf{v}}[c_0 v_0+u]-2B_{\mathbf{f}, \mathbf{v}}[c_0v_0, u]\right).
\end{align*}
As $B_{\Sigma, \mathbf{v}}[v_0,u]=0$, up to shrinking $\delta$ some more one has
\begin{align*}
|B_{\mathbf{f},\mathbf{v}}[c_0 v_0,u]| &\leq \frac{1}{2} B_{\mathbf{f},\mathbf{v}}[c_0 v_0]^{\frac{1}{2}} B_{\mathbf{f},\mathbf{v}}[u]^{\frac{1}{2}}\leq \frac{1}{4} \left( B_{\mathbf{f},\mathbf{v}}[c_0 v_0]+ B_{\mathbf{f},\mathbf{v}}[u]\right)\\
& \leq \frac{1}{4} \left( B_{\mathbf{f},\mathbf{v}}[c_0 v_0+u]+2|B_{\mathbf{f},\mathbf{v}}[c_0 v_0, u]|\right). 
\end{align*}
That is,
$$
|B_{\mathbf{f},\mathbf{v}}[c_0 v_0,u]| \leq \frac{1}{2}  B_{\mathbf{f},\mathbf{v}}[c_0 v_0+u],
$$
and so
$$
Q_{\mathbf{f}, \mathbf{v}}[v]\leq \mu^+ B_{\mathbf{f}, \mathbf{v}}[c_0 v_0+u]=\mu^+ B_{\mathbf{f}, \mathbf{v}}[v],
$$
verifying the claim.
	
Hence, there is an $N+1$ dimensional subspace, $V\subset W^1_{\frac{3}{8}}(\Sigma)$, so that, for all $v\in V$,
$$
Q_{\mathbf{f}, \mathbf{v}}[v]\leq \mu^+B_{\mathbf{f}, \mathbf{v}}[v].
$$
By Lemma \ref{ParaSpectrumLem}, $-L_{\mathbf{f},\mathbf{v}}^\prime$ is self-adjoint with respect to $B_{\mathbf{f},\mathbf{v}}$ and has a discrete spectrum. As such, it follows from the min-max characterization of eigenvalues that the $(N+1)$-st eigenvalue (counted with multiplicities), $\mu_{\mathbf{f},\mathbf{v}}^{N+1}$, of $-L_{\mathbf{f},\mathbf{v}}'$ satisfies $\mu_{\mathbf{f},\mathbf{v}}^{N+1}\leq \mu^+$. Now let 
$$
\mathcal{E}_{\mathbf{f},\mathbf{v}}^\prime[\mu_{\mathbf{f},\mathbf{v}}^{N+1}]=\mathrm{span} \set{u\in W^1_{\frac{3}{8}}(\Sigma)\colon -L_{\mathbf{f},\mathbf{v}}^\prime u=\mu^\prime u \mbox{ for some $\mu^\prime\leq \mu^{N+1}_{\mathbf{f},\mathbf{v}}$}},
$$
and choose $W\subset \mathcal{E}_{\mathbf{f},\mathbf{v}}^\prime [\mu_{\mathbf{f},\mathbf{v}}^{N+1}]$ spanned by $(N+1)$ linearly independent eigenfunctions. 
	
Using $K_{\Lambda}$ and $\delta=\frac{\mu^+}{4(K_\Lambda+1)}$ let $\epsilon_3^\prime$ be given by Corollary \ref{PerturbCor} with $\Lambda$ in place of $\Sigma$. Observe $\epsilon_3^\prime=\epsilon_3^\prime(K_\Lambda, \mu^+, n, N+1)$ depends only on $\Sigma$. By Lemma \ref{InverseLem}, we may shrink $\epsilon_3$, in a manner depending only on $\Sigma$, so that $\mathbf{f}^{-1}\in \mathcal{B}_{\epsilon_3^\prime}(\mathbf{x}|_\Lambda;\mathcal{ACH}^{k,\alpha}_n(\Lambda))$. As such, we may apply Corollary \ref{PerturbCor}, with $\Lambda$ in place of $\Sigma$, to $w\in W$ to see that
\begin{align*}
Q_{\Sigma, \mathbf{v}}[w]&=Q_{\mathbf{f}^{-1}, \mathbf{v}\circ \mathbf{f}^{-1}} [w\circ \mathbf{f}^{-1}]-Q_{\mathbf{f}, \mathbf{v} }[w]+Q_{\mathbf{f}, \mathbf{v} }[w]\\
&= Q_{\mathbf{f}^{-1}, \mathbf{v}\circ \mathbf{f}^{-1}} [w\circ \mathbf{f}^{-1}]-Q_{\Lambda,  \mathbf{v}\circ \mathbf{f}^{-1}}[w\circ \mathbf{f}^{-1}]+Q_{\mathbf{f}, \mathbf{v} }[w]\\
&\leq \frac{1}{2} \mu^+ B_{\Lambda,  \mathbf{v}\circ \mathbf{f}^{-1}}[w\circ \mathbf{f}^{-1}]+Q_{\mathbf{f}, \mathbf{v}}[w]\\
&=\frac{1}{2}\mu^+ B_{\mathbf{f},\mathbf{v}}[w]+Q_{\mathbf{f}, \mathbf{v}}[w]
\leq \frac{3}{2} \mu^+ B_{\mathbf{f}, \mathbf{v}}[w].
\end{align*}
Similarly, applying Corollary \ref{PerturbCor}, with $\Lambda$ in place of $\Sigma$, gives
$$
B_{\mathbf{f}, \mathbf{v}}[w]= B_{\Lambda, \mathbf{v}\circ\mathbf{f}^{-1}}[w\circ\mathbf{f}^{-1}]\leq 2 B_{\mathbf{f}^{-1}, \mathbf{v}\circ\mathbf{f}^{-1}}[w\circ\mathbf{f}^{-1}]=2 B_{\Sigma, \mathbf{v}}[w].
$$
Hence, as $\mu^+> 0$, 
$$
Q_{\Sigma, \mathbf{v}}[w]\leq 3 \mu^+B_{\Sigma, \mathbf{v}}[w].
$$
Recall, by Lemma \ref{SpectrumLem}, $-L_{\Sigma,\mathbf{v}}^\prime$ is self-adjoint with respect to $B_{\Sigma,\mathbf{v}}$ and has a discrete spectrum. As such, by the min-max characterization of eigenvalues, one has that $\mu_{\Sigma,\mathbf{v}}^{N+1}$, the $(N+1)$-st eigenvalue of $-L_{\Sigma, \mathbf{v}}^\prime$ satisfies $\mu_{\Sigma,\mathbf{v}}^{N+1}\leq 3 \mu^+\leq \frac{3}{4}\mu_{\Sigma, \mathbf{v}}^+$. However, as $N$ is the dimension of $\mathcal{E}_{\Sigma, \mathbf{v}}^\prime[0]$, one must have $\mu_{\Sigma,\mathbf{v}}^{N+1}=\mu_{\Sigma,\mathbf{v}}^+>\frac{3}{4}\mu_{\Sigma, \mathbf{v}}^+>0$. This contradiction proves the second part of the proposition and completes the arguments. 
\end{proof}

Ultimately, we will show an extension of \cite[Proposition 4]{WhiteEI} for asymptotically conical self-expanders. However, before stating the theorem we introduce some notation inspired by \cite{WhiteVM}. Fix a self-expander $\Sigma\in \mathcal{ACH}^{k,\alpha}_n$ and a transverse section on $\Sigma$, $\mathbf{v}\in C^{k,\alpha}_0\cap C^{k}_{0,\mathrm{H}}(\Sigma;\mathbb{R}^{n+1})$ so that $\mathcal{C}[\mathbf{v}]$ is transverse to $\mathcal{C}(\Sigma)$ and $\mathscr{L}_\Sigma^0\mathbf{v}\in C^{k-2,\alpha}_{-2}(\Sigma;\mathbb{R}^{n+1})$. Let $\hat{G}_{\mathbf{v}}$ and $\hat{F}_{\mathbf{v}}$ be given in Theorem \ref{ModifiedSmoothDependThm}, and $(\varphi, \kappa)$ is in a small neighborhood of $(\mathbf{x}|_{\mathcal{L}(\Sigma)},0)$ in $C^{k,\alpha}(\mathcal{L}(\Sigma); \Real^{n+1})\times \mathcal{K}_{\mathbf{v}}$. Formally, define the following function (which is infinite valued)
$$
\hat{g}_{\mathbf{v}}[\varphi, \kappa]=\int_{\hat{F}_{\mathbf{v}}[\varphi,\kappa](\Sigma)} e^{\frac{|\mathbf{x}|^2}{4}} \, d\mathcal{H}^n=E[\hat{F}_{\mathbf{v}}[\varphi,\kappa]].
$$
Differentiating (formally) in the second variable gives the following well-defined function $D_2 \hat{g}_{\mathbf{v}}$ on $\mathcal{K}_{\mathbf{v}}$:
\begin{align*}
D_2 \hat{g}_{\mathbf{v}}(\varphi, \kappa)\kappa_1&=-B_{\Sigma, \mathbf{v}}[\hat{\Xi}_{\hat{F}_{\mathbf{v}}[\varphi, \kappa],\mathbf{v}}[0], D_2\hat{\mathcal{F}}_\mathbf{v}(\varphi, \kappa)\kappa_1]\\
&=-B_{\Sigma, \mathbf{v}}[\hat{G}_{\mathbf{v}}[\varphi, \kappa], D_2\hat{\mathcal{F}}_\mathbf{v}(\varphi, \kappa)\kappa_1]\\
&=-B_{\Sigma, \mathbf{v}}[\hat{G}_{\mathbf{v}}[\varphi, \kappa], \kappa_1]
\end{align*}
where the last equality follows from the fact that $\hat{G}_{\mathbf{v}}[\phi, \kappa]\in \mathcal{K}_{\mathbf{v}}$ and \eqref{OrthogFcalEqn}.  An immediate consequence of this is that 
$$
D_2 \hat{g}_{\mathbf{v}}(\varphi, \kappa)=0 \iff \hat{G}_{\mathbf{v}}[\varphi, \kappa]=0.
$$ 
That is, $\hat{F}_{\mathbf{v}}[\varphi, \kappa]$ is $E$-stationary if and only if $D_2 \hat{g}_{\mathbf{v}}(\varphi, \kappa)=0$. 

Finally, by differentiating again at a $(\varphi, \kappa)$ where $\hat{G}_{\mathbf{v}}[\varphi, \kappa]=0$, that is where $\hat{F}_{\mathbf{v}}[\varphi, \kappa]$ is $E$-stationary, one has the bilinear form on $\mathcal{K}_{\mathbf{v}}$ given by 
\begin{align*}
(D_{22}^2 \hat{g}_{\mathbf{v}}(\varphi, \kappa))[\kappa_1, \kappa_2]&=-B_{\Sigma, \mathbf{v}}[D\hat{\Xi}_{\hat{F}_{\mathbf{v}}[\varphi, \kappa]}(0)\circ D_2\hat{\mathcal{F}}_{\mathbf{v}}(\varphi,\kappa)\kappa_1,  D_2\hat{\mathcal{F}}_{\mathbf{v}}(\varphi,\kappa)\kappa_2]\\
&=-B_{\Sigma, \mathbf{v}}[\Omega_{\hat{F}_{\mathbf{v}}[\varphi, \kappa],\mathbf{v}}L_{\hat{F}_{\mathbf{v}}[\varphi, \kappa], \mathbf{v}}^\prime\circ D_2\hat{\mathcal{F}}_{\mathbf{v}}(\varphi,\kappa)\kappa_1,  D_2\hat{\mathcal{F}}_{\mathbf{v}}(\varphi,\kappa)\kappa_2]\\
&=Q_{\hat{F}_{\mathbf{v}}[\varphi, \kappa],\mathbf{v}}[D_2\hat{\mathcal{F}}_{\mathbf{v}}(\varphi, \kappa)\kappa_1,D_2\hat{\mathcal{F}}_{\mathbf{v}}(\varphi, \kappa)\kappa_2].
\end{align*}

\begin{thm}\label{MainTechThm}
Let $\Sigma\in \mathcal{ACH}^{k,\alpha}_n$ be a self-expander, and let $\mathbf{v}$ be a transverse section on $\Sigma$ so that $\mathcal{C}[\mathbf{v}]$ is transverse to $\mathcal{C}(\Sigma)$ and $\mathscr{L}_\Sigma^0\mathbf{v}\in C^{k-2,\alpha}_{-2}(\Sigma;\mathbb{R}^{n+1})$. There is an $\bar{\epsilon}=\bar{\epsilon}(\Sigma, \mathbf{v})>0$ so that if $(\varphi, \kappa)\in C^{k,\alpha}(\mathcal{L}(\Sigma); \Real^{n+1}) \times \mathcal{K}_{\mathbf{v}}$ satisfies $\hat{G}_{\mathbf{v}}[\varphi, \kappa]=0$ and  $\Vert \varphi -\mathbf{x}|_{\mathcal{L}(\Sigma)}\Vert_{k, \alpha}+\Vert \kappa \Vert_{\mathcal{D}^{k,\alpha}\cap W^2_{3/8}} <\bar{\epsilon}$, then 
\begin{enumerate}
\item $\mathrm{ind} (\hat{F}_{\mathbf{v}}[\varphi, \kappa])= \mathrm{ind}(\Sigma)+ \mathrm{ind}(D_{22}^2 \hat{g}_{\mathbf{v}}(\varphi,\kappa))$.  
\item $\mathrm{null}(\hat{F}_{\mathbf{v}}[\varphi, \kappa])= \mathrm{null}(D_{22}^2 \hat{g}_{\mathbf{v}}(\varphi,\kappa))$.
\end{enumerate}
\end{thm}

\begin{proof}
Pick $\epsilon^\prime>0$ as in Proposition \ref{QuadFormProp} and choose $\bar{\epsilon}$ so $\hat{F}_{\mathbf{v}}[\varphi, \kappa]\in \mathcal{B}_{\epsilon^\prime}(\mathbf{x}|_\Sigma; \mathcal{ACH}^{k,\alpha}_n(\Sigma))$. For convenience, write $\mathbf{f}=\hat{F}_{\mathbf{v}}[\varphi, \kappa]$. Note that $\mathcal{K}_{\mathbf{v}}=\mathcal{K}_{\Sigma, \mathbf{v}}^\prime$. Set 
$$
\tilde{\mathcal{K}}=\set{D_2 \hat{\mathcal{F}}_{\mathbf{v}}(\varphi, \kappa) \kappa^\prime\colon \kappa^\prime\in \mathcal{K}_{\mathbf{v}}}=\mathrm{Im}(D_2 \hat{\mathcal{F}}_{\mathbf{v}}(\varphi, \kappa)).
$$
By definition, 
$$D_2\hat{\mathcal{F}}_{\mathbf{v}}(\varphi,\kappa)\kappa^\prime=\kappa^\prime+ D_2\hat{F}_{\mathbf{v},*}(\varphi, \kappa)\kappa^\prime,
$$
and so \eqref{OrthogFcalEqn} implies 
$$
\tilde{\mathcal{K}}\subset \mathcal{K}_{\mathbf{v}}+\hat{\mathcal{K}}_{\mathbf{v}, *}^\perp\subset \mathcal{D}^{k,\alpha}\cap W^{2}_{\frac{3}{8}}(\Sigma).
$$
Moreover, as 
$$
D_2 \hat{G}_{\mathbf{v}}(\varphi, \kappa)\colon \mathcal{K}_{\mathbf{v}}\to \mathcal{K}_{\mathbf{v}}
$$
it follows from \eqref{DiffIdentEqn} that, setting  $J^\prime=-\Omega_{\mathbf{f}, \mathbf{v}} L_{\mathbf{f}, \mathbf{v}}^\prime$, one has 
$$
D_2 \hat{G}_{\mathbf{v}}(\varphi, \kappa)=-J^\prime \circ D_2 \hat{\mathcal{F}}_{\mathbf{v}}(\varphi, \kappa)
$$
and so $J^\prime(\tilde{\mathcal{K}})\subset \mathcal{K}_{\mathbf{v}}=K_{\Sigma, \mathbf{v}}^\prime$. Furthermore, as $D_2\hat{F}_{\mathbf{v},*}(\varphi, \kappa)$ has image in $\hat{\mathcal{K}}_{\mathbf{v}, *}^\perp$, $\ker(D_2\hat{\mathcal{F}}_{\mathbf{v}})=\set{0}$ and so 
$$
l=\dim\tilde{\mathcal{K}}=\dim\mathcal{K}_{\mathbf{v}}=\mathrm{null}(\Sigma).
$$
Finally,
$$
Q_{\mathbf{f}, \mathbf{v}}[u, v]=B_{\Sigma, \mathbf{v}}[u, -\Omega_{\mathbf{f}, \mathbf{v}}L_{\mathbf{f}, \mathbf{v}}^\prime v]= B_{\Sigma, \mathbf{v}}[u, J^\prime v]= B_{\Sigma, \mathbf{v}}[J^\prime u, v].
$$

Now let $\tilde{Q}_{\mathbf{f}, \mathbf{v}}\colon \tilde{\mathcal{K}}\times \tilde{\mathcal{K}}\to \Real$ be the bilinear form defined, for $\tilde{\kappa}_1, \tilde{\kappa}_2\in \tilde{\mathcal{K}}$ by
$$
\tilde{Q}_{\mathbf{f}, \mathbf{v}}[\tilde{\kappa}_1, \tilde{\kappa}_2]=Q_{\mathbf{f}, \mathbf{v}}[\tilde{\kappa}_1, \tilde{\kappa}_2].
$$
Clearly, for $\kappa_1, \kappa_2\in\mathcal{K}_{\mathbf{v}}$,
$$
(D_{22}^2\hat{g}_{\mathbf{v}}(\varphi, \kappa))[\kappa_1, \kappa_2]=\tilde{Q}_{\mathbf{f}, \mathbf{v}}[D_2 \hat{\mathcal{F}}_{\mathbf{v}}(\varphi, \kappa)\kappa_1, D_2 \hat{\mathcal{F}}_{\mathbf{v}}(\varphi, \kappa)\kappa_2].
$$
As the map $\kappa^\prime\mapsto D_2 \hat{\mathcal{F}}_{\mathbf{v}}(\varphi, \kappa)\kappa^\prime$ is injective, one has
$$
\mathrm{null}(\tilde{Q}_{\mathbf{f}, \mathbf{v}})=\mathrm{null}(D_{22}^2\hat{g}_{\mathbf{v}}(\varphi, \kappa))
$$
and
$$
\mathrm{ind}(\tilde{Q}_{\mathbf{f}, \mathbf{v}})=\mathrm{ind}(D_{22}^2\hat{g}_{\mathbf{v}}(\varphi, \kappa)).
$$

We are now ready to complete the proof. First, consider the decomposition of $W^{1}_{\frac{3}{8}}(\Sigma)$ into
$$
 W^1_{\frac{3}{8}}(\Sigma)=\mathcal{N}_{\Sigma, \mathbf{v}}^\prime\oplus \mathcal{K}_{\Sigma, \mathbf{v}}^\prime\oplus \mathcal{P}_{\Sigma, \mathbf{v}}^\prime.
$$
By Proposition \ref{QuadFormProp}, for $\epsilon^\prime$ sufficiently small, one has $Q_{\mathbf{f}, \mathbf{v}}$ positive definite on $\mathcal{P}_{\Sigma, \mathbf{v}}^\prime$ and negative definite on $\mathcal{N}_{\Sigma, \mathbf{v}}^\prime$. A consequence of this is that
$$
W^1_{\frac{3}{8}}(\Sigma)=\mathcal{N}_{\Sigma, \mathbf{v}}^\prime\oplus\tilde{\mathcal{K}}\oplus \mathcal{P}_{\Sigma, \mathbf{v}}^\prime.
$$
Indeed, consider the projection map 
$$
\pi\colon W^1_{\frac{3}{8}}(\Sigma)\to W^1_{\frac{3}{8}}(\Sigma)/(\mathcal{N}_{\Sigma, \mathbf{v}}^\prime+\mathcal{P}_{\Sigma, \mathbf{v}}^\prime)=\pi( \mathcal{K}_{\Sigma, \mathbf{v}}^\prime).
$$
The orthogonal decomposition of  $W^{1}_{\frac{3}{8}}(\Sigma)$ implies $\pi( \mathcal{K}_{\Sigma, \mathbf{v}}^\prime) \simeq \Real^l$ where $l=\dim \mathcal{K}_{\Sigma, \mathbf{v}}^\prime=\dim \mathcal{K}_{\mathbf{v}}$. Hence, it suffices to show $\pi|_{\tilde{\mathcal{K}}}$ is surjective.  As $\dim \tilde{\mathcal{K}}=l$ this is equivalent to showing this map is injective. If $\tilde{\kappa}\in \ker \pi|_{\tilde{\mathcal{K}}}$, then $\tilde{\kappa}=\tilde{p}+\tilde{n}$ for $\tilde{p}\in \mathcal{P}_{\Sigma, \mathbf{v}}^\prime$ and $\tilde{n}\in \mathcal{N}_{\Sigma, \mathbf{v}}^\prime$.  
However, 
$$
Q_{\mathbf{f}, \mathbf{v}}[\tilde{n},\tilde{n}]+Q_{\mathbf{f}, \mathbf{v}}[\tilde{p}, \tilde{n}]=Q_{\mathbf{f}, \mathbf{v}}[\tilde{\kappa}, \tilde{n}]=B_{\Sigma, \mathbf{v}}[J^\prime\tilde{\kappa}, \tilde{n}]=0
$$
where the last equality follows from the fact that $J^\prime\tilde{\kappa}\in \mathcal{K}_{\mathbf{v}}=\mathcal{K}_{\Sigma,\mathbf{v}}^\prime$ which is orthogonal with respect to $B_{\Sigma, \mathbf{v}}$ to $\mathcal{N}_{\Sigma, \mathbf{v}}^\prime$. Similar reasoning shows
$$
Q_{\mathbf{f}, \mathbf{v}}[\tilde{p}, \tilde{p}]+ Q_{\mathbf{f}, \mathbf{v}}[\tilde{p}, \tilde{n}]= Q_{\mathbf{f}, \mathbf{v}}[\tilde{p}, \tilde{\kappa}]=B_{\Sigma, \mathbf{v}}[\tilde{p}, J^\prime\tilde{\kappa}]=0.
$$
Subtracting the first equality from the second gives
$$
Q_{\mathbf{f}, \mathbf{v}}[\tilde{p}, \tilde{p}]-Q_{\mathbf{f}, \mathbf{v}}[\tilde{n}, \tilde{n}]=0.
$$
As $Q_{\mathbf{f}, \mathbf{v}}$ is positive definite on $\mathcal{P}_{\Sigma, \mathbf{v}}^\prime$ and negative definite on  $\mathcal{N}_{\Sigma, \mathbf{v}}^\prime$ one must have $\tilde{p}=\tilde{n}=0$. That is, $\tilde{\kappa}=0$, which proves the claim.

Now decompose $\tilde{\mathcal{K}}$ into
$$
\tilde{\mathcal{K}}=\mathcal{N}^\prime\oplus \mathcal{K}^\prime\oplus\mathcal{P}^\prime
$$
where each subspace $\mathcal{N}^\prime, \mathcal{K}^\prime, \mathcal{P}^\prime$ is orthogonal (with respect to $B_{\Sigma, \mathbf{v}}$) and $\tilde{Q}_{\mathbf{f}, \mathbf{v}}$ is, respectively, negative definite, zero and positive definite on $\mathcal{N}^\prime, \mathcal{K}^\prime$ and $\mathcal{P}^\prime$. Now consider $p\in \mathcal{P}^\prime_{\Sigma, \mathbf{v}}$ and $p^\prime\in \mathcal{P}^\prime$ one has
\begin{align*}
Q_{\mathbf{f}, \mathbf{v}}[p+p^\prime, p+p^\prime] &= Q_{\mathbf{f}, \mathbf{v}}[p,p]+2Q_{\mathbf{f},\mathbf{v}}[p,p^\prime]+Q_{\mathbf{f},\mathbf{v}}[p^\prime,p^\prime] \\
&=Q_{\mathbf{f},\mathbf{v}}[p,p]+2B_{\Sigma,\mathbf{v}}[p,J^\prime p^\prime]+Q_{\mathbf{f},\mathbf{v}}[p^\prime,p^\prime]\\
&=Q_{\mathbf{f}, \mathbf{v}}[p,p]+\tilde{Q}_{\mathbf{f}, \mathbf{v}}[p^\prime, p^\prime].
\end{align*}
Here the last equality uses the fact that $J^\prime p^\prime\in \mathcal{K}_{\mathbf{v}}=\mathcal{K}_{\Sigma, \mathbf{v}}'$ is orthogonal to $p$ and that $J^\prime$ is self-adjoint with respect to $B_{\Sigma, \mathbf{v}}$. By Proposition \ref{QuadFormProp} and the fact that $p^\prime\in\mathcal{N}$ one has for $p+p^\prime\neq 0$ that
$$
Q_{\mathbf{f}, \mathbf{v}}[p+p^\prime, p+p^\prime]>0
$$
That is, $Q_{\mathbf{f}, \mathbf{v}}$ is positive definite on $\mathcal{P}_{\Sigma, \mathbf{v}}\oplus\mathcal{P}^\prime$. Similar reasoning shows that $Q_{\mathbf{f}, \mathbf{v}}$ is negative definite on $\mathcal{N}_{\Sigma, \mathbf{v}}\oplus\mathcal{N}^\prime$,  negative semi-definite on $\mathcal{N}_{\Sigma, \mathbf{v}}^\prime\oplus\mathcal{N}^\prime\oplus\mathcal{K}^\prime$ and positive semi-definite on $\mathcal{P}_{\Sigma, \mathbf{v}}\oplus\mathcal{P}^\prime\oplus\mathcal{K}^\prime$. As a consequence, 
\begin{align*}
\mathrm{ind}(\hat{F}_{\mathbf{v}}[\varphi, \kappa])&= \mathrm{ind}(\mathbf{f})=\mathrm{ind}(Q_{\mathbf{f}, \mathbf{v}})\\
&= \dim (\mathcal{N}_{\Sigma, \mathbf{v}}^\prime\oplus\mathcal{N}^\prime)=\dim\mathcal{N}_{\Sigma, \mathbf{v}}^\prime+\dim\mathcal{N}^\prime\\
&= \mathrm{ind}(\Sigma)+\mathrm{ind}(\tilde{Q}_{\mathbf{f}, \mathbf{v}})=\mathrm{ind}(\Sigma)+\mathrm{ind}(D^2_{22}\hat{g}_{\mathbf{v}}(\varphi,\kappa)).
\end{align*}
One further concludes that, for $u\in W^1_{\frac{3}{8}}(\Sigma)$, $J^\prime u=0$ if and only if $u\in \mathcal{K}^\prime$ and so
$$
\mathrm{null}(\hat{F}_{\mathbf{v}}[\varphi, \kappa])=\dim\mathcal{K}^\prime=\mathrm{null}(D^2_{22}\hat{g}_{\mathbf{v}}(\varphi,\kappa)).
$$
This completes the proof of the result.
\end{proof}

\section{Proof of Theorem \ref{MainThm}} \label{MainThmSec}
We now include a proof of theorem \ref{MainThm}.  We adapt the proof from \cite{WhiteEI}, using Theorem \ref{MainTechThm} in place of \cite[Proposition 4]{WhiteEI}. The details are included for the sake of the reader.

First of all, for $\Gamma\in\mathcal{ACH}^{k,\alpha}_n$ and $l\geq 1$, let
$$
\mathcal{S}_l=\set{[\mathbf{f}]\in \mathcal{ACE}_n^{k,\alpha}(\Gamma)\colon \mathrm{null}(\mathbf{f})\geq l}.
$$
We record some technical facts proved in \cite{WhiteEI} that readily adapt to the setting of this paper.

\begin{lem}\label{BaireLem}
For $\Gamma\in\mathcal{ACH}^{k,\alpha}_n$ the following is true:
\begin{enumerate}
\item $\Pi(\mathcal{S}_1)$ is of first Baire category in $C^{k,\alpha}(\mathcal{L}(\Gamma); \Real^{n+1})$. 
\item For each $\varphi_0, \varphi_1\in \mathcal{V}^{k,\alpha}_{\mathrm{emb}}(\Gamma)$ a generic, in the sense of Baire category, curve $\Phi\in C^2([0,1]; \mathcal{V}_{\mathrm{emb}}^{k,\alpha}(\Gamma))$ joining $\varphi_0$ to $\varphi_1$ will be embedded and transverse to $\Pi$.  
\end{enumerate}
\end{lem}

\begin{proof}
This is proved in \cite[Theorem 5.1]{WhiteEI} using only properties of the spaces and of $\Pi$ that are the same as in our setting.
\end{proof}

\begin{lem} \label{ComposeLem}
Let $\Sigma\in\mathcal{ACH}^{k,\alpha}_n$ be a self-expander, and let $\mathbf{f}\in\mathcal{ACH}^{k,\alpha}_n(\Sigma)$ satisfy $\mathscr{L}_\Sigma\mathbf{f}\in C^{k-2,\alpha}_{-1}(\Sigma;\mathbb{R}^{n+1})$. Let $\Lambda=\mathbf{f}(\Sigma)$. Then the following is true:
\begin{enumerate} 
\item $\mathscr{L}_\Lambda\mathbf{f}^{-1}\in C^{k-2,\alpha}_{-1}(\Lambda;\mathbb{R}^{n+1})$.
\item If $u\in C^{k,\alpha}_{0}(\Sigma)$ satisfies $\mathscr{L}^0_\Sigma u\in C^{k-2,\alpha}_{-2}(\Sigma)$, then $\mathscr{L}_{\Lambda}^0(u\circ\mathbf{f}^{-1}) \in C^{k-2,\alpha}_{-2}(\Lambda).
$
\end{enumerate}
\end{lem}

\begin{proof}
For $\mathbf{f}\in\mathcal{ACH}^{k,\alpha}_n(\Sigma)$, by \cite[Proposition 3.3]{BWBanachManifold}, $\Lambda\in\mathcal{ACH}^{k,\alpha}_n$ and $\mathbf{f}^{-1}\in\mathcal{ACH}^{k,\alpha}_n(\Lambda)$. Consequently, $\Delta_\Lambda\mathbf{f}^{-1}\in C^{k-2,\alpha}_{-1}(\Lambda;\mathbb{R}^{n+1})$. To prove the first claim, it suffices to show $\mathbf{x}\cdot\nabla_\Lambda\mathbf{f}^{-1}-\mathbf{f}^{-1}\in C^{k-2,\alpha}_{-1}(\Lambda;\mathbb{R}^{n+1})$. As $\Sigma\in\mathcal{ACH}^{k,\alpha}_n$ is a self-expander, one has $\mathscr{L}_\Sigma\mathbf{x}=0$ and so $\mathbf{x}\cdot\nabla_\Sigma\mathbf{x}-\mathbf{x}\in C^{k-2,\alpha}_{-1}(\Sigma;\mathbb{R}^{n+1})$. Substituting $\mathbf{x}|_\Sigma=\mathbf{f}^{-1}\circ\mathbf{f}$ and applying the chain rule give
$$
\mathbf{x}\cdot\nabla_\Sigma (\mathbf{f}^{-1}\circ\mathbf{f})-\mathbf{f}^{-1}\circ\mathbf{f}=(\nabla_\Lambda\mathbf{f}^{-1}\circ\mathbf{f})\cdot(\mathbf{x}\cdot\nabla_\Sigma\mathbf{f}-\mathbf{f})+(\mathbf{x}\cdot\nabla_\Lambda\mathbf{f}^{-1}-\mathbf{f}^{-1})\circ\mathbf{f}.
$$
As $\mathbf{f}\in C^{k,\alpha}_1(\Sigma;\mathbb{R}^{n+1})$ and $\mathscr{L}_\Sigma\mathbf{f}\in C^{k-2,\alpha}_{-1}(\Sigma;\mathbb{R}^{n+1})$, the claim follows by invoking \cite[Propositions 3.1 and 3.3]{BWBanachManifold}.

If $u\in C^{k,\alpha}_0(\Sigma)$ and $\mathbf{f}\in\mathcal{ACH}^{k,\alpha}_n(\Sigma)$, then, by \cite[Proposition 3.3]{BWBanachManifold}, $u\circ\mathbf{f}^{-1}\in C^{k,\alpha}_0(\Lambda)$ and so $\Delta_\Lambda (u\circ\mathbf{f}^{-1})\in C^{k-2,\alpha}_{-2}(\Lambda)$. One further computes,
$$
\mathbf{x}\cdot\nabla_\Lambda(u\circ\mathbf{f}^{-1})=(\nabla_\Sigma u\circ\mathbf{f}^{-1})\cdot(\mathbf{x}\cdot\nabla_\Lambda\mathbf{f}^{-1}-\mathbf{f}^{-1})+(\mathbf{x}\cdot\nabla_\Sigma u)\circ\mathbf{f}^{-1},
$$
which, by what we have shown and our hypotheses on $u$, is an element of $C^{k-2,\alpha}_{-2}(\Lambda)$. Thus, the second claim follows easily from this.
\end{proof}

\begin{lem} \label{OrientationLem}
Let $\Sigma\in\mathcal{ACH}^{k,\alpha}_n$ be a self-expander, and let $\mathbf{v}\in C^{k,\alpha}_0\cap C^{k}_{0,\mathrm{H}}(\Sigma;\mathbb{R}^{n+1})$ be a transverse section on $\Sigma$ so that $\mathcal{C}[\mathbf{v}]$ is transverse to $\mathcal{C}(\Sigma)$ and $\mathscr{L}_\Sigma^0\mathbf{v}\in C^{k-2,\alpha}_{-2}(\Sigma;\mathbb{R}^{n+1})$. Suppose that $\Pi^{-1}(\mathbf{x}|_{\mathcal{L}(\Sigma)})\cap\mathcal{W}$ is a connected one-dimensional submanifold of $\mathcal{ACE}^{k,\alpha}_n(\Sigma)$ for $\mathcal{W}$ a neighborhood of $[\mathbf{x}|_\Sigma]\in\mathcal{ACE}^{k,\alpha}_n(\Sigma)$ and that $\kappa\in\mathcal{K}_\mathbf{v}$ with $B_{\Sigma,\mathbf{v}}[\kappa]=1$. Given $\varphi\in C^{k,\alpha}(\mathcal{L}(\Sigma);\mathbb{R}^{n+1})$ there is an $\epsilon=\epsilon(\Sigma,\mathbf{v},\varphi)>0$ so that the following holds:
\begin{enumerate}
\item If $|s|<\epsilon$, then $\Lambda_s=\hat{F}_{\mathbf{v}}[\mathbf{x}|_{\mathcal{L}(\Sigma)},s\kappa](\Sigma)\in\mathcal{ACH}^{k,\alpha}_n$ is a self-expander with $\mathcal{L}(\Lambda_s)=\mathcal{L}(\Sigma)$ and  $\mathbf{w}_s=\mathbf{v}\circ \hat{F}_{\mathbf{v}}[\mathbf{x}|_{\mathcal{L}(\Sigma)},s\kappa]^{-1}$ is a transverse section to $\Lambda_s$. Moreover, setting
$$
\kappa_s=\left( \frac{d}{ds}\hat{\mathcal{F}}_\mathbf{v}[\mathbf{x}|_{\mathcal{L}(\Sigma)},s\kappa]\right)\circ\hat{F}_\mathbf{v}[\mathbf{x}|_{\mathcal{L}(\Sigma)},s\kappa]^{-1}\in \mathcal{K}_{\mathbf{w}_s},
$$
  $B_{\Sigma,\mathbf{v}}[D_1\hat{G}_\mathbf{v}(\mathbf{x}|_{\mathcal{L}(\Sigma)},0)\varphi,\kappa]>0$ implies $B_{\Lambda_s,\mathbf{w}_s}[D_1\hat{G}_{\mathbf{w}_s}(\mathbf{x}|_{\mathcal{L}(\Lambda_s)},0)\varphi, \kappa_s]>0$. 
\item For $\mathbf{v}^\prime\in C^{k,\alpha}_0\cap C^{k}_{0,\mathrm{H}}(\Sigma;\mathbb{R}^{n+1})$ satisfying that $\Vert\mathbf{v}^\prime-\mathbf{v}\Vert^{(0)}_{k,\alpha}<\epsilon$ and $\mathscr{L}_\Sigma^0\mathbf{v}^\prime\in C^{k-2,\alpha}_{-2}(\Sigma;\mathbb{R}^{n+1})$, setting 
$$
\kappa^\prime=\frac{(\mathbf{v}\cdot\mathbf{n}_\Sigma)\kappa}{\mathbf{v}^\prime\cdot\mathbf{n}_\Sigma} \in \mathcal{K}_{\mathbf{v}^\prime},
$$
if $B_{\Sigma,\mathbf{v}}[D_1\hat{G}_\mathbf{v}(\mathbf{x}|_{\mathcal{L}(\Sigma)},0)\varphi,\kappa]>0$, then $B_{\Sigma,\mathbf{v}}[D_1\hat{G}_{\mathbf{v}^\prime}(\mathbf{x}|_{\mathcal{L}(\Sigma)},0)\varphi,\kappa^\prime]>0$.
\end{enumerate}
Here, $\hat{F}_\mathbf{v}$, $\hat{G}_\mathbf{v}$, $\hat{G}_{\mathbf{v}^\prime}$, $\hat{G}_{\mathbf{w}_s}$, $\mathcal{K}_\mathbf{v}$, $\mathcal{K}_{\mathbf{v}^\prime}$, $\mathcal{K}_{\mathbf{w}_s}$ and $\hat{\mathcal{F}}_{\mathbf{v}}$ are given in Section \ref{ModifiedSmoothDependSec}.
\end{lem}

\begin{proof}
First, for simplicity, denote by $\mathbf{f}_s=\hat{F}_\mathbf{v}[\mathbf{x}|_{\mathcal{L}(\Sigma)},s\kappa]$. Observe that by our constructions, for $\epsilon$ sufficiently small $\mathbf{f}_s\in\mathcal{ACH}^{k,\alpha}_n(\Sigma)$ and $\mathscr{L}_\Sigma \mathbf{f}_s\in C^{k-2,\alpha}_{-1}(\Sigma;\mathbb{R}^{n+1})$. Thus, by Lemma \ref{ComposeLem}, for $\epsilon$ sufficiently small $\hat{F}_{\mathbf{w}_s}$ and $\hat{G}_{\mathbf{w}_s}$ are well defined. 

Using definitions of the maps from Theorem \ref{ModifiedSmoothDependThm} and the rapid decay of the various terms one has
$$
B_{\Sigma,\mathbf{v}}[D_1\hat{G}_\mathbf{v}(\mathbf{x}|_{\mathcal{L}(\Sigma)},s\kappa)\varphi,\kappa_s\circ \mathbf{f}_s]=B_{\Sigma, \mathbf{v}}\left[{\frac{d}{dt}\vline}_{t=0}\hat{\Xi}_{\mathbf{v}}[\hat{F}_{\mathbf{v}}[\mathbf{x}|_{\mathcal{L}(\Sigma)}+t\varphi ,s\kappa]], \kappa_s\circ \mathbf{f}_s\right].
$$
Arguing as in Section \ref{VariationSubsec}, the fact that the $\mathbf{f}_s$ are stationary for the functional $E$ and the second variation formula for $E$ give
\begin{align*}
 {\frac{d}{dt}\vline}_{t=0} & \hat{\Xi}_{\mathbf{v}}[\hat{F}_{\mathbf{v}}[\mathbf{x}|_{\mathcal{L}(\Sigma)}+t\varphi ,s\kappa]]\\
 &= {\frac{d}{dt}\vline}_{t=0}\hat{\Xi}_{\mathbf{f}_s, \mathbf{v}}[ t (\mathbf{v}\cdot (\mathbf{n}_{\Lambda_s}\circ \mathbf{f}_s))^{-1} D_1  \hat{F}_{\mathbf{v}}(\mathbf{x}|_{\mathcal{L}(\Sigma)} ,s\kappa)\varphi\cdot (\mathbf{n}_{\Lambda_s}\circ \mathbf{f}_s) ]\\
 &= D\hat{\Xi}_{\mathbf{f}_s, \mathbf{v}}(0) ( (\mathbf{v}\cdot (\mathbf{n}_{\Lambda_s}\circ \mathbf{f}_s))^{-1} D_1  \hat{F}_{\mathbf{v}}(\mathbf{x}|_{\mathcal{L}(\Sigma)} ,s\kappa)\varphi\cdot (\mathbf{n}_{\Lambda_s}\circ \mathbf{f}_s))\\
 &= \Omega_{\mathbf{f}_s, \mathbf{v}} L_{\mathbf{f}_s, \mathbf{v}}^\prime( (\mathbf{v}\cdot (\mathbf{n}_{\Lambda_s}\circ \mathbf{f}_s))^{-1} D_1  \hat{F}_{\mathbf{v}}(\mathbf{x}|_{\mathcal{L}(\Sigma)} ,s\kappa)\varphi\cdot (\mathbf{n}_{\Lambda_s}\circ \mathbf{f}_s))\\
 &=\Omega_{\mathbf{f}_s, \mathbf{v}}(\mathbf{v}\cdot (\mathbf{n}_{\Lambda_s}\circ \mathbf{f}_s))^{-1} L_{\mathbf{f}_s}(  D_1  \hat{F}_{\mathbf{v}}(\mathbf{x}|_{\mathcal{L}(\Sigma)} ,s\kappa)\varphi\cdot (\mathbf{n}_{\Lambda_s}\circ \mathbf{f}_s)).
\end{align*} 
Hence,
\begin{align*}
& B_{\Sigma,\mathbf{v}}[D_1\hat{G}_\mathbf{v}(\mathbf{x}|_{\mathcal{L}(\Sigma)},s\kappa)\varphi,\kappa_s\circ \mathbf{f}_s] \\
&= B_{\Sigma, \mathbf{v}}[  \Omega_{\mathbf{f}_s, \mathbf{v}} ( \mathbf{v}\cdot (\mathbf{n}_{\Lambda_s}\circ  \mathbf{f}_s)  )^{-1} L_{\mathbf{f}_s}( D_1  \hat{F}_{\mathbf{v}}(\mathbf{x}|_{\mathcal{L}(\Sigma)} ,s\kappa)\varphi\cdot (\mathbf{n}_{\Lambda_s}\circ \mathbf{f}_s)), \kappa_s\circ \mathbf{f}_s]\\
& =B_{\Sigma}[ \Omega_{\mathbf{f}_s} ( \mathbf{v}\cdot (\mathbf{n}_{\Lambda_s}\circ  \mathbf{f}_s) )  L_{\mathbf{f}_s}(D_1  \hat{F}_{\mathbf{v}}(\mathbf{x}|_{\mathcal{L}(\Sigma)} ,s\kappa)\varphi\cdot  (\mathbf{n}_{\Lambda_s}\circ \mathbf{f}_s)), \kappa_s\circ \mathbf{f}_s]\\
 &=B_{\mathbf{f}_s} [ L_{\mathbf{f}_s}(D_1  \hat{F}_{\mathbf{v}}(\mathbf{x}|_{\mathcal{L}(\Sigma)} ,s\kappa)\varphi \cdot  (\mathbf{n}_{\Lambda_s}\circ \mathbf{f}_s)),  ( \mathbf{v}\cdot ( \mathbf{n}_{\Lambda_s}\circ \mathbf{f}_s) ) \kappa_s\circ \mathbf{f}_s]\\
 &= B_{\Lambda_s} [( \mathbf{w}_s\circ \mathbf{n}_{\Lambda_s} ) \kappa_s , L_{\Lambda_s}((D_1  \hat{F}_{\mathbf{v}}(\mathbf{x}|_{\mathcal{L}(\Sigma)} ,s\kappa)\varphi\circ \mathbf{f}_s^{-1}) \cdot  \mathbf{n}_{\Lambda_s}]\\
 & = \int_{\Lambda_s} (\mathbf{w}_s\cdot\mathbf{n}_{\Lambda_s})\kappa_s L_{\Lambda_s}\left((D_1\hat{F}_{\mathbf{v}}(\mathbf{x}|_{\mathcal{L}(\Sigma)},s\kappa)\varphi\circ \mathbf{f}_s^{-1})\cdot\mathbf{n}_{\Lambda_s}\right) e^{\frac{|\mathbf{x}|^2}{4}} \, d\mathcal{H}^n.
\end{align*}
Furthermore, the above integral can be simplified by using integration by parts (which is justified by Proposition 6.2 of \cite{BWBanachManifold}) to
$$
\int_{\mathcal{L}(\Lambda_s)} \mathrm{tr}_\infty^*[(\mathbf{w}_s\cdot\mathbf{n}_{\Lambda_s})\kappa_s] \mathrm{tr}^1_\infty[(D_1\hat{F}_\mathbf{v}(\mathbf{x}|_{\mathcal{L}(\Sigma)},s\kappa)\varphi\circ\mathbf{f}_s^{-1})\cdot\mathbf{n}_{\Lambda_s}]\, d\mathcal{H}^{n-1}
$$
where $\mathrm{tr}_\infty^*$ is given in Section 6 of \cite{BWBanachManifold}. By Item \eqref{TraceFItem} of Theorem \ref{ModifiedSmoothDependThm} and the linearity of $\mathrm{tr}_\infty^1$, we have that $\mathcal{L}(\Lambda_s)=\mathcal{L}(\Sigma)$, $\mathrm{tr}_\infty^1[\mathbf{f}_s]=\mathbf{x}|_{\mathcal{L}(\Sigma)}$ and 
$$
\mathrm{tr}_\infty^1[(D_1\hat{F}_{\mathbf{v}}(\mathbf{x}|_{\mathcal{L}(\Sigma)},s\kappa)\varphi\circ\mathbf{f}_s^{-1})\cdot\mathbf{n}_{\Lambda_s}]=\varphi\cdot\mathbf{n}_{\mathcal{L}(\Lambda_s)}.
$$
Thus, it follows that 
$$
B_{\Sigma,\mathbf{v}}[D_1\hat{G}_\mathbf{v}(\mathbf{x}|_{\mathcal{L}(\Sigma)},s\kappa)\varphi,\kappa_s\circ \mathbf{f}_s]=\int_{\mathcal{L}(\Lambda_s)} \mathrm{tr}_\infty^*[(\mathbf{w}_s\cdot\mathbf{n}_{\Lambda_s})\kappa_s] (\varphi\cdot\mathbf{n}_{\mathcal{L}(\Lambda_s)}) \, d\mathcal{H}^{n-1}.
$$
In particular, evaluating the identity at $s=0$ gives
$$
B_{\Sigma,\mathbf{v}}[D_1\hat{G}_\mathbf{v}(\mathbf{x}|_{\mathcal{L}(\Sigma)},0)\varphi,\kappa]=\int_{\mathcal{L}(\Sigma)} \mathrm{tr}_\infty^*[(\mathbf{v}\cdot\mathbf{n}_{\Sigma})\kappa] (\varphi\cdot\mathbf{n}_{\mathcal{L}(\Sigma)}) \, d\mathcal{H}^{n-1}.
$$
Notice that, in the derivation of the above identity, we only use the fact that $\Sigma$ is an asymptotically conical self-expander embedded in a family of asymptotically conical self-expanders with the same asymptotic cone, that $\mathbf{v}$ is a transverse section on $\Sigma$ satisfying the hypotheses of the lemma and that $\kappa\in \mathcal{K}_{\mathbf{v}}$. Thus, the same identity holds true with $\Sigma$ replaced by $\Lambda_{s}$, $\mathbf{v}$ by $\mathbf{w}_{s}$ and $\kappa$ by $\kappa_s$, that is,
$$
B_{\Lambda_s,\mathbf{w}_s}[D_1\hat{G}_{\mathbf{w}_s}(\mathbf{x}|_{\mathcal{L}(\Lambda_s)},0)\varphi,\kappa_s]=\int_{\mathcal{L}(\Lambda_s)} \mathrm{tr}_\infty^*[(\mathbf{w}_s\cdot\mathbf{n}_{\Lambda_s})\kappa_s] (\varphi\cdot\mathbf{n}_{\mathcal{L}(\Lambda_s)}) \, d\mathcal{H}^{n-1}.
$$
Hence,
$$
B_{\Lambda_{s},\mathbf{w}_{s}}[D_1\hat{G}_{\mathbf{w}_{s}}(\mathbf{x}|_{\mathcal{L}(\Lambda_{s})},0)\varphi,\kappa_{s}]=B_{\Sigma,\mathbf{v}}[D_1\hat{G}_\mathbf{v}(\mathbf{x}|_{\mathcal{L}(\Sigma)},s\kappa)\varphi,\kappa_{s}\circ \mathbf{f}_{s}].
$$
Therefore,  after possibly shrinking $\epsilon$, the first claim follows from this and the fact that 
$$
s\mapsto B_{\Sigma,\mathbf{v}}[D_1\hat{G}_\mathbf{v}(\mathbf{x}|_{\mathcal{L}(\Sigma)},s\kappa)\varphi,\kappa_s\circ \mathbf{f}_s]
$$
is smooth near $s=0$. 

Finally, by what we have shown, 
$$
B_{\Sigma,\mathbf{v}}[D_1\hat{G}_\mathbf{v}(\mathbf{x}|_{\mathcal{L}(\Sigma)},0)\varphi,\kappa]=B_{\Sigma,\mathbf{v}^\prime}[D_1\hat{G}_{\mathbf{v}^\prime}(\mathbf{x}|_{\mathcal{L}(\Sigma)},0)\varphi,\kappa^\prime],
$$
proving the second claim.
\end{proof}

\begin{proof}[Proof of Theorem \ref{MainThm}]
Given two points $\varphi_0, \varphi_1 \in \mathcal{V}\backslash\Pi(\mathcal{S}_1)$ let $C=\Phi([0,1]])$ be the image of a generic curve in $\mathcal{V}$ connecting $\varphi_0$ to $\varphi_1$ as given by Lemma \ref{BaireLem}.  As $C$ is embedded and transverse to $\Pi$, $C^\prime=\Pi^{-1}(C)\cap \mathcal{U}$ is a (not necessarily connected) one-dimensional submanifold with boundary.  Moreover,
$$
\partial C^\prime=\Pi^{-1}(\partial C)\cap \mathcal{U}=\Pi^{-1}(\set{\varphi_0, \varphi_1})\cap \mathcal{U}.
$$
Notice that $C^\prime$ is compact as $\Pi|_{\mathcal{U}}$ is proper.  

Orient $C$ so $\varphi_0$ is the initial point. We now orient $C^\prime$ as follows:
\begin{enumerate}
\item If $[\mathbf{f}]\in C^\prime\backslash \mathcal{S}_1$, then $\Pi$ is a local diffeormorphism at $[\mathbf{f}]$ (by \cite[Theorem 1.1]{BWBanachManifold} or Theorem \ref{ModifiedSmoothDependThm}). That is, there is an open neighborhood $\mathcal{U}^\prime\subset \mathcal{U}$ of $[\mathbf{f}]$ so that $\Pi|_{\mathcal{U}^\prime}$ is a diffeomorphism onto its image. Notice $\Pi(C^\prime\cap \mathcal{U}^\prime)\subset C$. Let $I^\prime=\mathcal{U}^\prime\cap C^\prime$ and orient $I^\prime$ so $\Pi|_{I^\prime}$ is orientation preserving if $[\mathbf{f}]$ has even Morse index and $\Pi|_{I^\prime}$ is orientation-reversing if $[\mathbf{f}]$ has odd Morse index.
\item If $[\mathbf{f}]\in C^\prime\cap \mathcal{S}_1$, then first observe that as $C$ is transverse to $\Pi$, $[\mathbf{f}]\in \mathcal{S}_1\backslash \mathcal{S}_2$.  In this case a neighborhood, $I^\prime$, of $[\mathbf{f}]$ embeds in $\tilde{C}\times \Real$ as $\hat{G}_{\mathbf{v}}^{-1}(0)$ (with $[\mathbf{f}]$ corresponding to $(\mathbf{x}|_{\mathcal{L}(\mathbf{f}(\Gamma))},0)$) where 
$$
\tilde{C}=\set{\mathrm{tr}_\infty^1[\mathscr{E}^\mathrm{H}_1[\varphi]\circ\mathcal{C}[\mathbf{f}]^{-1}]\colon \varphi\in C}
$$
and
$$
D_{1} \hat{G}_{\mathbf{v}}(\mathbf{x}|_{\mathcal{L}(\mathbf{f}(\Gamma))},0)\neq 0.
$$
The orientation of $\tilde{C}$ and the standard one of $\Real$ determine an orientation on $\tilde{C}\times \Real$. If $\mathrm{ind}(\mathbf{f})$ is even, orient $I^\prime=\hat{G}_{\mathbf{v}}^{-1}(0)$ so it is the boundary of $\hat{G}_{\mathbf{v}}>0$ while if $\mathrm{ind}(\mathbf{f})$ is odd, give $I^\prime$ the reverse orientation.
\end{enumerate}
If $[\mathbf{f}]\in C^\prime\cap\mathcal{S}_1$ can be approximated by a sequence of regular points on $C^\prime$, then, by Theorem \ref{MainTechThm}, these two conventions are consistent. If $\Pi^{-1}(\mathrm{tr}_\infty^1[\mathbf{f}])$ has a component containing $[\mathbf{f}]$ that has a nonempty open subset of $C^\prime\cap\mathcal{S}_1$, then Lemma \ref{OrientationLem} ensures that the second convention orients this component independent of choices. Thus, the above convention orients $C^\prime$  and the degree $\mathrm{deg}(\Pi|_\mathcal{U}, \mathcal{V})$ is just the usual degree of the map $\Pi|_{C^\prime}\colon C^\prime\to C$. 
\end{proof}

\section{Proof of Theorem  \ref{ApplicationThm}} \label{ApplicationSec}
Finally, we prove Theorem \ref{ApplicationThm}. First, we use the shrinking rotationally symmetric torus described by Angenent \cite{Angenent} to prove the following non-existence result for connected asymptotically conical self-expanders whose asymptotic cone is a ``narrow" double cone.

\begin{lem} \label{NonExistLem}
There is a $\delta_0=\delta_0(n)>0$ so that: If $\mathcal{C}$ is a $C^{k,\alpha}$-regular cone in $\Real^{n+1}$ with both $\mathcal{C}\cap \set{x_{n+1}>0}$ and $\mathcal{C}\cap \set{x_{n+1}<0}$ non-empty and $\mathcal{C}\subset  \set{x_1^2+\ldots +x_n^2\leq \delta_0^2 x_{n+1}^2}$, then any self-expander, $\Sigma\in \mathcal{ACH}_n^{k,\alpha}$ with $\mathcal{C}(\Sigma)=\mathcal{C}$ is disconnected.
\end{lem}

\begin{proof}
Let $\Lambda$ to the rotationally symmetric self-shrinker given by Angenent in \cite{Angenent} chosen so the axis of symmetry is $\ell_{n+1}$,  the $x_{n+1}$-axis. That is, $\Lambda$ is of the form
$$
\Lambda= \set{f(x_1^2+\ldots+ x_n^2, x_{n+1})=0}
$$
for some function $f\colon\Real^2\to \Real$. The construction of $\Lambda$ ensures there are $0<R_-<\sqrt{2n}< R_+$ so that
$$
 \Lambda \subset \bar{B}_{R_+}\backslash B_{R_-} \mbox{ and } \Lambda\cap \set{x_{n+1}=0}=\left(\partial B_{R_+} \cup \partial B_{R_-}\right)\cap \set{x_{n+1}=0}.
$$
As $\Lambda$ is disjoint from $\ell_{n+1}$, there is a $\delta_0=\delta_0(n)>0$ so that 
$$
\Lambda \cap \left(\set{x_1^2+\ldots +x_n^2\leq \delta^2_0 x_{n+1}^2}\cup B_{R_-}\right)=\emptyset.
$$
In particular, for the given cone $\mathcal{C}\cap \Lambda=\emptyset$. 

Consider the mean curvature flows $\mathcal{M}=\set{\sqrt{t}\, \Sigma}_{t>0}$ and $\tilde{\mathcal{M}}=\set{\sqrt{1-t}\, \Lambda}_{t<1}$ of $\Sigma$ and of $\Lambda$, respectively. As $\mathcal{C}(\Sigma)\cap \Lambda=\emptyset$ and $\Lambda$ is compact, the parabolic maximum principle implies that 
$$
\sqrt{t}\, \Sigma\cap \sqrt{1-t}\, \Lambda=\emptyset
$$
for all $t\in (0,1)$.  

If $\Sigma$ is connected, then $\Sigma_0=\Sigma\cap \set{x_{n+1}=0}$ is non-empty as $\set{x_{n+1}>0}\cap \Sigma$ and $\set{x_{n+1}<0}\cap \Sigma$ are both non-empty by the hypotheses on $\mathcal{C}=\mathcal{C}(\Sigma)$. Moreover, as $\mathcal{C}(\Sigma)\cap  \set{x_{n+1}=0}=\set{\mathbf{0}}$, $\Sigma_0$ is compact. In particular, for $t$ sufficiently small $\sqrt{t}\, \Sigma_0\subset B_{\sqrt{1-t}\, R_-}$. Hence, there is a $t_*\in (0,1)$ so $\sqrt{t_*}\, \Sigma_0\cap \partial B_{\sqrt{1-t_*} \, R_-}\neq \emptyset$. The definition of $R_-$ means $\sqrt{t_*}\, \Sigma_0\cap \partial B_{\sqrt{1-t_*} \, R_-}\subset \sqrt{1-t_*} \, \Lambda$. As this contradicts the disjointness of the flows, one must have that $\Sigma$ is not connected.
\end{proof}

We next use a construction of Angenent-Ilmanen-Chopp \cite{AIC} of connected rotationally symmetric self-expanders and a variational argument sketched by Ilmanen \cite{IlmanenNotes} and carried out by Ding \cite{Ding} in order to show existence of a connected self-expander asymptotic to a ``wide" double cone.

\begin{lem}\label{ExistLem}
Fix $2\leq n \leq 6$. There is a $\delta_1=\delta_1(n)>0$ so that: If $\mathcal{C}$ is a $C^{k,\alpha}$-regular cone in $\Real^{n+1}$ with $\mathcal{C}\subset  \set{x_1^2+\ldots +x_n^2\geq \delta^2_1 x_{n+1}^2}$ and so that $\mathcal{L}[\mathcal{C}]=\mathcal{L}_+\cup \mathcal{L}_-$ where $\mathcal{L}_\pm$ are connected, $\mathcal{L}_+ \subset \set{ x_{n+1}>0}$ and $\mathcal{L}_-\subset \set{x_{n+1}<0}$ and $0\neq [\mathcal{L}_-]=[\mathcal{L}_+]\in H^{n-1}(\Real^{n+1}\backslash \ell_{n+1}; \mathbb{Z}_2)$, where $\ell_{n+1}$ is the $x_{n+1}$-axis, then there is a connected asymptotically conical self-expander, $\Sigma\in \mathcal{ACH}_n^{k,\alpha}$ with $\mathcal{C}(\Sigma)=\mathcal{C}$. Moreover, this $\Sigma$ may be chosen to be (weakly) stable.
\end{lem}

\begin{proof}
By the construction of Angenent-Ilmanen-Chopp \cite{AIC}, there is a $\delta>\delta_0>0$ so that there is a connected rotationally symmetric self-expander $\Lambda \in \mathcal{ACH}^{k,\alpha}_n$ with
$$
\mathcal{C}(\Lambda)=\set{x_1^2+\ldots+x_n^2=\delta^2 x_{n+1}^2}. 
$$
This self-expander has the additional property that $\Lambda\cap \ell_{n+1}=\emptyset$, that is, it is disjoint from the $x_{n+1}$-axis. As such we may choose $\Omega$ to be the connected component of $\Real^{n+1}\backslash \Lambda$ that is disjoint from $\ell_{n+1}$.

Now let $\delta_1=2\delta>0$. This choice ensures that there is an $R_0>0$ so for $R>R_0$, $R\mathcal{L}[\mathcal{C}]\subset \Omega\cap \partial B_R$. Using $\Lambda$ and $\partial B_R$ as barriers, one uses the direct method to find a smooth self-expander $\Sigma_R\subset \Omega\cap B_R$ with $\partial \Sigma_R= R\mathcal{L}[\mathcal{C}]$ -- here we use the hypotheses that $[\mathcal{L}[\mathcal{C}]]=[\mathcal{L}_+]+[\mathcal{L}_-]=2[\mathcal{L}_+]=0$ and so $\mathcal{L}[\mathcal{C}]$ is null-homologous in $H^{n-1}(\Real^{n+1}\backslash \ell_{n+1}; \mathbb{Z}_2)$.  As $\mathcal{L}_\pm$ are connected and there are no closed self-expanders,  $\Sigma_R$ is connected. 

Finally, by suitably modifying the argument of \cite{IlmanenNotes} (see also \cite[Theorem 6.3]{Ding} for full details) one verifies that, up to passing to a subsequence, one has $\lim_{R_i\to \infty} \Sigma_{R_i}=\Sigma$ with $\Sigma\in \mathcal{ACH}^{k,\alpha}_n$ and $\mathcal{C}(\Sigma)=\mathcal{C}$.  Clearly,  $\Sigma$ is connected and (weakly) stable.

For the sake of completeness, we include a brief explanation of how the argument of Ding in \cite[Theorem 6.3]{Ding} should be modified. The main issue is that Ding uses barriers to show that the limit $\Sigma$ is $C^0$ close to the cone $\mathcal{C}$ and one might worry that these are not consistent with the barriers used to construct $\Sigma_R$. To that end, consider the following family of functions depending on parameters $\mathbf{v}\in \mathbb{S}^n$, $\eta>0$ and $h\geq 0$ :
$$
f_{\mathbf{v}, \eta, h}(\mathbf{x})=2n+ |\mathbf{x}|^2-\left(1+\eta^{-2}\right) (\mathbf{x}\cdot \mathbf{v})^2+h.
$$
We observe that the connected closed set 
$$
E_{\mathbf{v}, \eta, h}=\set{\mathbf{x}\in\mathbb{R}^{n+1}\colon f_{\mathbf{v}, \eta, h}(\mathbf{x})\leq 0\mbox{ and } \mathbf{x}\cdot \mathbf{v}\geq 0}
$$
has boundary asymptotic to the connected, rotationally symmetric cone,
$$
\mathcal{C}_{\mathbf{v}, \eta}=\set{|\mathbf{x}|^2= \left(1+\eta^{-2}\right)(\mathbf{x}\cdot \mathbf{v})^2, \mathbf{x}\cdot \mathbf{v}\geq 0},
$$
that lies in the half-space $\mathbf{x\cdot \mathbf{v}\geq 0}$ and has axis parallel to $\mathbf{v}$ and cone angle $\eta^{-1}$. Furthermore, $E_{\mathbf{v}, \eta, h}$ lies entirely in the component of $\Real^{n+1}\backslash \mathcal{C}_{\mathbf{v}, \eta}$ that contains $\mathbf{v}$. Observe that, by construction, $f_{\mathbf{v}, \eta,h}>0$ on $\set{|\mathbf{x}\cdot \mathbf{v}|<\eta (2n+h)}$ and so
$$
E_{\mathbf{v}, \eta, h}\cap \set{\mathbf{x}\cdot \mathbf{v}<\eta (2n+h)}=\emptyset.
$$  
A straightforward computation gives that on any self-expander, $\Sigma$, 
$$
\mathscr{L}_{\Sigma}^0 f_{\mathbf{v}, \eta, h} -f_{\mathbf{v}, \eta, h} =-2\left(\eta^{-2}+1\right)|\nabla_\Sigma (\mathbf{v}\cdot\mathbf{x})|^2 -h\leq 0.
$$
 
One readily checks that for each $p\in \mathcal{L}(C)$ and choice of normal $\mathbf{N}$ to $\mathcal{L}[\mathcal{C}]$ at $p$, there are $\eta>0$, $h>0$ and $\mathbf{v}_\pm\in \mathbb{S}^{n}$ so that $E_p=E_{\mathbf{v}_\pm, \eta, h}$ satisfies $E_p\subset \Omega \backslash \mathcal{C}$, $E_p\subset H_p$ where $H_p$ is the component of $\set{x_{n+1}\neq 0}$ containing $p$ and $\mathcal{L}[\mathcal{C}_{\mathbf{v_\pm}, \eta}]$ lies on the $\pm \mathbf{N}$ side of $\mathcal{L}[\mathcal{C}]$ and touches $\mathcal{L}[\mathcal{C}]$ only at $p$. By construction, $f_{\mathbf{v}_\pm, \eta, h}>0$ on the component of $\partial \Sigma_R$ that lies in $H_p$. Moreover, $f_{\mathbf{v}_\pm, \eta,h}>0$ on $\set{x_{n+1}=0}$. Hence, by the strong maximum principle, $f_{\mathbf{v}_\pm, \eta, h}>0$ on $\Sigma_R\cap H_p$. That is, $\Sigma_R\cap E_p=\emptyset$. As 
$$
\left(\mathscr{L}_\Sigma^0-1\right)(|\mathbf{x}|^2+2n)=0
$$
the strong maximum principle and compactness of $\mathcal{L}[\mathcal{C}]$ imply that, there is an $R^\prime>0$ so that $B_{R^\prime}\cap E_p\neq\emptyset$ for all $p\in\mathcal{L}[\mathcal{C}]$ and that $\Sigma_R\setminus B_{R^\prime}$, $R>R^\prime$, has exactly two components, one of which contains $R\mathcal{L}_+$ and the other contains $R\mathcal{L}_-$. Thus, by construction and further enlarging $R^\prime$, we have that, for $p\in\Sigma_R\setminus B_{R^\prime}$, $\mathrm{dist}(p,\mathcal{C})=O(|\mathbf{x}(p)|^{-1})$. Hence one obtains from this that $\Sigma$ is $C^0$ asymptotic to $\mathcal{C}$.  The argument that $\mathcal{C}(\Sigma)=\mathcal{C}$ then proceeds exactly as in \cite[Theorem 6.3]{Ding}.
\end{proof}

We need to following refinement of a result of Fong-McGrath \cite{FongMcGrath}.

\begin{lem}\label{SymmetryLem}
Let $\Sigma\in \mathcal{ACH}^{k,\alpha}_n$ be a (weakly) stable self-expander. If $\mathcal{C}(\Sigma)$ is symmetric with respect to rotation around the $x_{n+1}$-axis, $\ell_{n+1}$,  then so is $\Sigma$.
\end{lem}

\begin{proof}
Let $\mathbf{R}$ be a non-trivial rotational Killing vector field in $\Real^{n+1}$ that is a symmetry of $\mathcal{C}(\Sigma)$. Consider  $f=\mathbf{R}\cdot \mathbf{n}_{\Sigma}$.  One readily checks that $f\in C^{k-1,\alpha}_{1, 0}\cap C^{\infty}_{loc}(\Sigma)$ and $L_{\Sigma} f =0$ (see, e.g., \cite[Lemma 2.3]{FongMcGrath}). As $\Sigma$ has a weakly conical end (cf. \eqref{REstimatesEqn}), it follows from \cite[Theorem 9.1]{Bernstein} that $f\in W_{\frac{1}{4}}^{1}(\Sigma)$.  

In particular, either $f$ identically vanishes and so $\mathbf{R}$ is a symmetry of $\Sigma$ or $f$ is an eigenfunction of $-L_{\Sigma}$ with eigenvalue $0$. As $\Sigma$ is (weakly) stable this latter situation would imply that $f$ is the lowest eigenfunction of $-L_\Sigma$ and so $f$ cannot change sign.   However, as $\ell_{n+1}\cap \mathcal{C}(\Sigma)=\set{\mathbf{0}}$, there must be a $p\in \Sigma$ that satisfies $\dist(p, \ell_{n+1})=\dist(\Sigma, \ell_{n+1})$. Clearly, $f(p)=0$ and so one is in the situation that $f$ identically vanishes.  As $\mathbf{R}$ was arbitrary, one deduces that $\Sigma$ is also rotationally symmetric about $\ell_{n+1}$.
\end{proof}

By \cite{BWProperness}, the map $\Pi$, for surfaces of a fixed topological type in $\Real^3$, is proper. Lemma \ref{NonExistLem} implies the degree of $\Pi$ is zero for annuli and so the existence of unstable solutions will follow from Lemma \ref{ExistLem}. 
 
\begin{proof}[Proof of Theorem \ref{ApplicationThm}]
Let 
$$
\Gamma_0= \set{x_1^2+x_2^2=x_3^2+1}\subset \Real^3.
$$
Observe that $\Gamma_0\in \mathcal{ACH}_2^{k,\alpha}$ and is a topological annulus. Let $\mathcal{V}$ be the component of $\mathcal{V}^{k,\alpha}_{\mathrm{emb}}(\Gamma_0)$ that contains $\mathbf{x}|_{\mathcal{L}(\Gamma_0)}$. Let $\mathcal{U}=\Pi^{-1}(\mathcal{V})$ where 
$$
\Pi\colon \mathcal{ACE}_{2}^{k,\alpha}(\Gamma_0)\to C^{k,\alpha}(\mathcal{L}(\Gamma_0); \Real^3)
$$
is the natural projection map from \cite{BWBanachManifold}. By \cite{BWProperness}, $\Pi_0=\Pi|_{\mathcal{U}}\colon \mathcal{U}\to \mathcal{V}$ is a proper map. Hence, by Theorem \ref{MainThm}, $\Pi_0$ possesses an integer degree $\mathrm{deg}(\Pi|_\mathcal{U}, \mathcal{V})$. 

Write $\varphi=(\varphi_1, \varphi_2, \varphi_3)$. For $\lambda>0$, let
$$
\varphi_\lambda= (\varphi_1, \varphi_2, \lambda \varphi_3).
$$
One readily checks that $\varphi_{\lambda}\in \mathcal{V}$ for all $\lambda > 0$. Moreover, one has
$$
\mathcal{C}_\lambda =\mathcal{C}[\varphi_\lambda(\mathcal{L}(\Gamma_0))]=\set{ x_1^2+x_2^2=\lambda^{-2} x_3^2}.
$$
On the one hand, for $\lambda_0=2 \delta_0^{-1}$ where $\delta_0$ is given by Lemma \ref{NonExistLem},  there is no connected self-expander with asymptotic cone $\mathcal{C}_{\lambda_0}$.  As $\Gamma_0$ is connected, one has $\Pi^{-1}(\varphi_{\lambda_0})=\emptyset$ and so $\mathrm{deg}(\Pi|_\mathcal{U}, \mathcal{V})=0$. On the other hand, by the construction of Angenent-Ilmanen-Chopp \cite{AIC}, there is a $\lambda_1\in (0, \lambda_0)$, so that there is a connected rotationally symmetric expander, $\Lambda$, asymptotic to $\mathcal{C}_{\lambda_1}$. That is, $\Pi^{-1}(\varphi_{\lambda_1})\neq \emptyset$.  Clearly, we may assume this $\lambda_1$ satisfies $\lambda_1^{-1}>\delta_1(2)$ where $\delta_1(2)$ is given by Lemma \ref{ExistLem} for $n=2$.

To conclude the proof we need to find a $\varphi\in \mathcal{V}$ so that $\varphi$ is a regular value of $\Pi$ and so $\Pi^{-1}(\varphi)$ is non-empty.  To that end, first let $\Omega$ be the unique component of $\Real^3 \backslash \mathcal{C}(\Lambda)$ that is disjoint from the $x_3$-axis. Next, observe that, as the set of regular values of $\Pi$ is generic (by \cite{BWBanachManifold}), there is a sequence, $\varphi_i\in \mathcal{V}$ of regular values of $\Pi$ with the property that $\varphi_i\to \varphi$ in $C^{k,\alpha}(\mathcal{L}(\Gamma_0); \Real^3)$ and $\varphi_i(\mathcal{L}(\Gamma_0)) \subset \Omega$.  

It follows from Lemma \ref{ExistLem} that there is a sequence of (weakly) stable connected self-expanders $\Sigma_i\in \mathcal{ACH}^{k,\alpha}_2$ with $\mathcal{L}(\Sigma_i)= \varphi_i(\mathcal{L}(\Gamma_0))$. By the compactness results of \cite{BWProperness}, up to passing to a subsequence and relabeling, the $\Sigma_i\to \Sigma$ in the smooth topology where $\Sigma$ is a (weakly) stable, connected,  self-expander with $\mathcal{C}(\Sigma)=\mathcal{C}_{\lambda_1}$. By Lemma \ref{SymmetryLem}, as $\Sigma$ is (weakly) stable and $\mathcal{C}(\Sigma)$ is rotationally symmetric around the $x_3$-axis, $\Sigma$ is also rotationally symmetric.  As $\Sigma$ is also connected this means that $\Sigma$ is an annulus and so, by the nature of the convergence, up to throwing out a finite number of terms, the $\Sigma_i$ are also annuli for all $i$.  

As a consequence, there are $\mathbf{f}_i\in \mathcal{ACH}_2^{k,\alpha}(\Gamma_0)$ with $\mathbf{f}_i(\Gamma_0)=\Sigma_i$ and $\mathrm{tr}_\infty^1 [\mathbf{f}_i]=\varphi_i$.  In particular, $[\mathbf{f}_i]\in \Pi^{-1}_0(\varphi_i)$ are stable elements. As the degree of $\Pi_0$ is $0$, this means that there must be at least one unstable element $[\mathbf{f}_i^\prime]\in \Pi^{-1}_0(\varphi_i)$. As $\Pi$ is a local diffeomorphism, at each regular point, the claim is proved by fixing some $i_0\geq 1$ large and letting the open set be some small neighborhood, in $C^{k,\alpha}$, of $\mathcal{C}(\Sigma_{i_0})$.
\end{proof}

\appendix

\section{Poincar\'{e} inequality} \label{PoincareIneqnSec}

\begin{prop}\label{GradControlsL2Prop}
Let $\Sigma\in \mathcal{ACH}_{n}^{k,\alpha}$ be a self-expander. For $\beta \geq \frac{1}{4}$ and $f\in W^{1}_{\beta}(\Sigma)$,
\begin{equation}\label{PI1eqn}
\int_{\Sigma} \left(4 n +|\mathbf{x}|^2\right)f^2 e^{\beta |\mathbf{x}|^2}\, d\mathcal{H}^n \leq 16 \int_{\Sigma}|\nabla_\Sigma f|^2 e^{\beta |\mathbf{x}|^2}\, d\mathcal{H}^n.
\end{equation}
Likewise
\begin{equation}\label{PI2eqn}
(4\beta-1)	\int_{\Sigma}|\mathbf{x}^\top|^2 f^2 e^{\beta |\mathbf{x}|^2} \, d\mathcal{H}^n \leq 4 \int_{\Sigma}|\nabla_\Sigma f|^2 e^{\beta |\mathbf{x}|^2}\, d\mathcal{H}^n.
\end{equation}
\end{prop}

\begin{proof}
First observe that as $\Sigma$ is a self-expander one has that
$$
\mathscr{L}_{\Sigma}^0 \left( |\mathbf{x}|^2+2n\right)=|\mathbf{x}|^2+2n.
$$
Now suppose that $f$ has compact support. Integrating by parts gives
\begin{align*}
&\int_\Sigma \left(2 n +|\mathbf{x}|^2\right)f^2 e^{\beta |\mathbf{x}|^2} \, d\mathcal{H}^n = \int_\Sigma \left(\mathscr{L}_\Sigma^0 \left(2 n +|\mathbf{x}|^2\right) \right) f^2 e^{\beta |\mathbf{x}|^2}\, d\mathcal{H}^n  \\
&= - \int_\Sigma \nabla_\Sigma \left(2 n +|\mathbf{x}|^2\right) \cdot \nabla_\Sigma \left( e^{\left( \beta-\frac{1}{4}\right)|\mathbf{x}|^2} f^2 \right) e^{\frac{|\mathbf{x}|^2}{4}}\, d\mathcal{H}^n \\
&=-\int_{\Sigma} \left( (4\beta-1) |\mathbf{x}^\top|^2 f^2 + 4 \phi \mathbf{x}^\top \cdot \nabla_\Sigma f \right)  e^{\beta |\mathbf{x}|^2} \, d\mathcal{H}^n  \\
&= - (4\beta-1) \int_{\Sigma}  |\mathbf{x}^\top|^2 f^2  e^{\beta |\mathbf{x}|^2} \, d\mathcal{H}^n -4 \int_{\Sigma} f \mathbf{x}^\top \cdot \nabla_\Sigma f \, e^{\beta |\mathbf{x}|^2}\,  d\mathcal{H}^n.
\end{align*}
The absorbing inequality and the hypothesis that $\beta\geq \frac{1}{4}$ give
$$
\int_\Sigma \left(2 n +|\mathbf{x}|^2\right) f^2 e^{\beta |\mathbf{x}|^2}\, d\mathcal{H}^n \leq  \int_{\Sigma}\left( \frac{1}{2}  |\mathbf{x}^\top|^2 f^2 + 8| \nabla_\Sigma f|^2 \right)  e^{\beta |\mathbf{x}|^2}\, d\mathcal{H}^n.
$$
Rearranging this gives the first inequality. The second inequality follows from the same computation but using a different version of the absorbing inequality. 

To see that these estimates hold for all $f\in W^{1}_{\beta}(\Sigma)$ one uses the compactly supported cutoff $\psi_R\in C^\infty(\Real^{n+1})$ with the property that $0\leq \psi_R\leq 1$, $\psi_R=1$ in $B_R$,  $\spt \psi_R\subset B_{2R}$ and $|\nabla \psi_R|\leq 2$. The estimates hold for $f_R=\psi_R f$ and hence, by the dominated convergence theorem, hold also for $f$.
\end{proof}

\begin{prop}\label{GradControlsBoundaryProp}
Let $\Sigma\in \mathcal{ACH}_{n}^{k,\alpha}$ be a self-expander. Suppose $R_0\geq 1$ is such that on $\bar{E}_{R_0}$ the estimates \eqref{REstimatesEqn} hold with constant $C_0$. There is an $R_1=R_1(C)\geq R_0$ so that: For $\beta \geq \frac{1}{4}$,  $f \in W^{1}_{\beta}(\Sigma)$ and $R>R_1$ one has the estimate
\begin{equation*}
R e^{\beta R^2} \int_{S_R}  f^2\, d\mathcal{H}^{n-1} \leq 2 \int_{\bar{E}_R} |\nabla_\Sigma f|^2 e^{\beta |\mathbf{x}|^2} \, d\mathcal{H}^n.
\end{equation*}
\end{prop}

\begin{proof}
Suppose first that $f$ has compact support. Let 
$$
\mathbf{Z}= \frac{|\mathbf{x}| \mathbf{x}^\top}{|\mathbf{x}^\top|}=\frac{\frac{1}{2}\nabla_\Sigma |\mathbf{x}|^2}{|\nabla_\Sigma |\mathbf{x}||}.
$$
By \eqref{REstimatesEqn}, $|\nabla_\Sigma |\mathbf{x}||\geq \frac{1}{2}$ so this vector field is well-defined. Using \eqref{REstimatesEqn}, one sees that, by taking $R_1$ sufficiently large, on $\bar{E}_{R_1}$,  one has
$$
|\mathbf{x}^\top| \geq |\mathbf{x}|-\frac{C}{|\mathbf{x}|^3}\geq  |\mathbf{x}| -\frac{1}{|\mathbf{x}|} \mbox{ and }  \mathrm{div}_\Sigma \mathbf{Z} \geq n-\frac{1}{2} \geq \frac{1}{2}.
$$ 
Hence, on $\bar{E}_R$,
\begin{align*} 
\mathrm{div}_\Sigma \left( f^2 e^{\beta |\mathbf{x}|^2} \mathbf{Z} \right) &= \left( f^2 \mathrm{div}_\Sigma \mathbf{Z}+ 2 f \nabla_\Sigma f \cdot \mathbf{Z} + 2\beta f^2\mathbf{x}^\top \cdot \mathbf{Z}\right) e^{\beta|\mathbf{x}|^2} \\
& \geq  \left(\frac{1}{2} f^2+ 2 f \nabla_\Sigma f \cdot \mathbf{Z} + \frac{1}{2} f^2 |\mathbf{Z}|\left(|\mathbf{x}|-|\mathbf{x}|^{-1}\right) \right) e^{\beta|\mathbf{x}|^2} \\
&\geq \left( \frac{1}{2} f^2-2 |\nabla_\Sigma f |^2 - \frac{1}{2} f^2 |\mathbf{Z}|^2  + \frac{1}{2} f^2 |\mathbf{Z}| \left(|\mathbf{x}|-|\mathbf{x}|^{-1}\right)\right) e^{\beta|\mathbf{x}|^2}
\end{align*}
where the last estimate uses the absorbing inequality. As $|\mathbf{Z}|=|\mathbf{x}|$ one has
$$
\mathrm{div}_\Sigma \left( f^2 e^{\beta |\mathbf{x}|^2} \mathbf{Z} \right) \geq -2 | \nabla_\Sigma f |^2 e^{\beta|\mathbf{x}|^2}.
$$
Integrating over $\bar{E}_R$ and appealing to the divergence theorem give
\begin{align*}
R e^{\beta R^2} \int_{S_R} f^2 \, d\mathcal{H}^{n-1} &= -\int_{S_R} f^2 \mathbf{Z}\cdot \frac{-\mathbf{x}^\top}{|\mathbf{x}^\top|} e^{\beta|\mathbf{x}|^2} \, d\mathcal{H}^{n-1}\\
&=\int_{\bar{E}_R} -\mathrm{div}_\Sigma \left( f^2 e^{\beta |\mathbf{x}|^2} \mathbf{Z} \right) \, d\mathcal{H}^n \leq 2 \int_{\bar{E}_R} |\nabla_\Sigma f |^2 e^{\beta|\mathbf{x}|^2} \, d\mathcal{H}^n,
\end{align*}
which proves the claim for $f$ of compact support. For general $f\in W^{1}_\beta(\Sigma)$ one uses cutoffs and the dominated convergence theorem. 
\end{proof}

\end{document}